\documentclass[12pt, a4paper, reqno]{amsart}
\usepackage{amssymb}

\usepackage[all]{xy}
\usepackage{amsthm}
\usepackage{mathrsfs}
\usepackage[utf8]{inputenc}
\usepackage{bbm}
\usepackage{amsfonts}
\usepackage{amsmath}
\usepackage{geometry}
\usepackage{enumerate}
\usepackage{tikz-cd}

\usepackage{hyperref}
\usepackage[capitalise]{cleveref}

\usetikzlibrary{decorations.markings}
\usetikzlibrary{cd}
\usetikzlibrary{patterns}
\usetikzlibrary{shapes}

\tikzset{anchorbase/.style={baseline={([yshift=-0.5ex]current bounding 
box.center)}}}
\tikzset{wipe/.style={white,line width=4pt}}

\usepackage{xcolor}
\usepackage{framed}

\usepackage{subfiles} % Best loaded last in the preamble
\newcommand{\glmn}{\mathfrak{gl}(m|n)}
\newcommand{\gram}{(a_1,\, \dots, a_m\,|b_1,\,\dots, b_n)}
\newcommand{\even}[1]{#1_{\bar{0}}}
\newcommand{\odd}[1]{#1_{\bar{1}}}
\newcommand{\fg}{\mathfrak{g}}
\newcommand{\fh}{\mathfrak{h}}

\newcommand{\fgl}[2]{\mathfrak{gl}(#1|#2)}
\newcommand{\distinguished}{\Sigma^{\text{dist}}}

\newcommand{\distHW}{{\Lambda^+_{m|n}}}
\newcommand{\antidist}{\Sigma^{\text{anti}}}
\newcommand{\ror}[2]{\varepsilon_{#1}-\delta_{#2}}
\newcommand{\eps}{\varepsilon}
\newcommand{\Z}{\mathbb{Z}}

\newcommand{\comment}[1]{}

\newtheorem{fact}{Fact}[subsection]
\newtheorem{theorem}[fact]{Theorem}
\newtheorem{prop}[fact]{Proposition}
\newtheorem{lemma}[fact]{Lemma}
\newtheorem{cor}[fact]{Corollary}
\newtheorem{conj}[fact]{Conjecture}
\newtheorem{definition}[fact]{Definition}
\newtheorem{notation}[fact]{Notation}
\newtheorem{example}[fact]{Example}
\newtheorem{remark}[fact]{Remark}
\newtheorem{introtheorem}{Theorem}
\newtheorem{introcor}[introtheorem]{Corollary}

\newcommand{\Innas}[1]{\fbox{\tt{\color{magenta}{#1}}}}
\newcommand{\Matan}[1]{\begin{framed}{\tt{\color{orange}{#1}}} \end{framed}}

\title{Weight diagrams of finite dimensional highest weight $\glmn$-modules}
\author{Matan Pinkas}
\address{Matan Pinkas, Department of Mathematics, Ben Gurion University of the Negev, Beer Sheva, Israel.}
\email{Matan.Pinkas@gmail.com}
\date{\today}

\begin{document}

\maketitle

\begin{abstract}
    Many properties of simple finite dimensional $\glmn$-modules may be better understood by assigning weight diagrams to the highest weights with respect to a given base of simple roots. 
    
    In this paper we consider bases that are compatible with the standard Borel subalgebra in $\glmn_{\bar 0} = \mathfrak{gl}(m)\times \mathfrak{gl}(n)$; namely, the bases that differ from the distinguished base $\distinguished$ of simple roots by a sequence of odd reflections. We examine the weight diagrams that arise from the highest weights of a simple module $L(\lambda)$ with respect to such bases.

    Further, we provide combinatorial tools to describe all the weight diagrams of highest weights of $L(\lambda)$ provided only with the the weight diagram of $L(\lambda)$ with respect to distinguished highest weight $\lambda$.
    
    Finally, we study the maximal cardinality of incomparable sets of positive odd roots with respect to $\distinguished$ which are orthogonal to some highest weight of $L(\lambda)$ with respect to a base as above. We provide explicit formulas for this value, connecting it to the combinatorics of the weight diagrams. Based on this study, we respond to the work of M. Gorelik and Th. Heidersdorf by providing a counterexample to the Tail Conjecture appearing in \cite{tailconjpaper}.

\end{abstract}

\section{Introduction}

\subsection{Motivation}

The study of superalgebras and Lie superalgebras is the study of $\Z / 2\Z$-graded objects, with sign rules governing the commutation relations. 

Such objects occur naturally when studying subjects such as exterior algebras of vector spaces, cohomology and tensor categories.

A systematic study of the structure and the representation theory of Lie superalgebras was initiated in the famous papers \cite{kacsuperalgdef, kacclassicalreptheory} by V. Kac.
Since then, the theory of Lie superalgebras has become an active field of study, see for example \cite{repsofsuperalgebras, glmnrep, weightdiagraminvetion,mussonbook, cheng2012dualities} 
and the references therein.

As with classical Lie algebras, to study Lie superalgebras one needs a combinatorial framework including the notions of root systems, bases of simple roots and more. The goal of this paper is to contribute to this framework for the general linear Lie superalgebra.

\subsection{The general linear Lie superalgebra \texorpdfstring{$\glmn$}{}}
Lie superalgebras are a natural generalization of the notion of Lie algebras. 

A vector space $V$ with a $\Z / 2\Z$-grading $V= \even{V} \oplus \odd{V}$ is called a vector superspace. We denote the parity of a homogeneous vector $a\in V_{\bar{i}}$ by $p(a):=i$. 

We define the Lie superalgebra $\mathfrak{gl}(V)$ to be the $\Z / 2\Z$-graded vector space $End(V)$ together with the super-commutator $$[a,b]=ab-(-1)^{p(a) p(b)}ba.$$
Taking the $\mathbb{Z}/2\mathbb{Z}$-graded space $V=\mathbb{C}^m \oplus \mathbb{C}^n$, we denote $\mathfrak{gl}(V)=\glmn$.

Much like the Lie algebra $\mathfrak{gl}(n)$, $\glmn$ has a root system. One can consider different simple bases for this root system and study the highest weights of simple finite-dimensional $\glmn$-modules with respect to these bases.

Since the "super" setting adds parity into the picture, some differences arise. The root system of $\glmn$ has even roots (corresponding to eigenvectors of parity $\overline{0}$) and odd root (corresponding to eigenvectors of parity $\overline{1}$).

Similarly to the theory of semisimple Lie algebras, we may consider the Weyl group corresponding to the Lie superalgebra $\glmn$. This group is isomorphic to $S_m\times S_n$. It is generated by reflections with respect to the even roots of $\glmn$, thus encoding the information about the even part of the Lie superalgebra. 

However, not all the bases of the root system for $\glmn$ are conjugate to one another via the action of the Weyl group. To remedy that, V. Serganova defined in \cite[Appendix]{LSS86} the notion of an odd reflection: that is, a transformation of the set of all roots which behaves as a reflection with respect to an odd root should (See Definition \ref{odd ref def}). more details on odd reflections, see \cite[Sections 1.3, 2.4]{cheng2012dualities} and \cite[Section 3.5]{mussonbook}. We denote by $r_{\alpha}$ the odd reflection with respect to a root $\alpha\in \odd{\Delta}$.

Consider the Cartan subalgebra $\mathfrak{h}$ of $\glmn$. This Lie sub-superalgebra is purely even, and is just the Cartan subalgebra of $\glmn_{\bar 0} = \mathfrak{gl}(m)\times \mathfrak{gl}(n)$. We write the natural basis of the dual space $\mathfrak{h}^*$ (the space of weights) as $\eps_1, \ldots, \eps_m, \delta_1, \ldots, \delta_n$. Then the set of even roots of $\glmn$ is $$\{\eps_i-\eps_j\,|\, 1\leq i\neq j\leq m\}\cup \{ \delta_i-\delta_j\, |\, 1\leq i\neq j\leq n\}$$ and the set of odd roots is $\{\pm (\eps_i-\delta_j)\,|\, 1\leq i\leq m, \, 1\leq j\leq n\}$.

The distinguished base for $\glmn$ is defined to be $$\distinguished:=\{\eps_i-\eps_{i+1}, \,\eps_m- \delta_1, \, \delta_j-\delta_{j+1}\,|\, 1\leq i\leq m-1, \, 1\leq j\leq n-1\}$$

The positive odd roots with respect to $\distinguished$ are of the form $\ror{i}{j}$ for some $i, j$, and we denote $\mathcal{R}:=\{\eps_i-\delta_j\,|\, 1\leq i\leq m, \, 1\leq j\leq n\}$. These will be called {\it right odd roots}.

Any base of $\glmn$ can be obtained from $\distinguished$ by a sequence of even and odd reflections, but the set of all the bases which can be obtained from $\distinguished$ using {\it only odd reflections by right odd roots} is denoted by $\mathbb{S}$. For any base $\Sigma\in \mathbb{S}$, the positive even roots with respect to the base $\Sigma$ are exactly the positive even roots for the base $\distinguished$.

\subsection{Definition of the weight diagram} 
As it turns out, a lot of information regarding a simple module over $\glmn$ may be encoded into what is known as a weight diagram of that module. Let us first explain how to construct this diagram.

Let $\distHW$ denote the set of all integral dominant weights with respect to the distinguished base $\distinguished$. This set is in bijection with the set of isomorphism classes of (finite-dimensional) integrable simple modules over $\glmn$, with the simple module corresponding to $\lambda\in \distHW$ denoted by $L(\lambda)$. 

Let $\lambda\in \distHW$ and let $\Sigma$ be any base of $\glmn$. We consider the positive and the negative roots corresponding to $\Sigma$ and denote by $\rho_{\Sigma}$ half the difference between the sums of the even positive roots and the sums of the odd positive roots. We then denote by $\lambda_{\Sigma}$ the highest weight of $L(\lambda)$ with respect to the base $\Sigma$, and $\overline{\lambda}_{\Sigma}:=\lambda_{\Sigma} + \rho_{\Sigma}$, with $\overline{\lambda}_{\distinguished}$ denoted by $\overline{\lambda}$ for short.

This shifted highest weight $\overline{\lambda}_{\Sigma}$ is then encoded by a very useful combinatorial tool, due to Brundan and Stroppel \cite{weightdiagraminvetion} as well as Musson and Serganova \cite{MS11}. 

The weight diagram $D_{\lambda}^{\Sigma}$ of $\overline{\lambda}_{\Sigma}$ is defined as a function on the integers, whose value in $p\in \mathbb{Z}$ is a finite collection of symbols $\times, >, <, \circ$, subject to the following rules:
\begin{itemize}
    \item Let $k:=\sharp \{i\in \{1,\dots, m\}|(\overline{\lambda}_\Sigma)_i=p\}$, $s:= \sharp \{j\in \{1,\dots, n\}|(\overline{\lambda}_\Sigma)_{-j}=-p\}$. Then $D^\Sigma_\lambda(p)$ contains exactly $\min(k, s)$ symbols $\times$.
    \item If $k>s$ then $D^\Sigma_\lambda(p)$ contains exactly $k-s$ symbols $>$.
      \item If $k<s$ then $D^\Sigma_\lambda(p)$ contains exactly $s-k$ symbols $<$.
      \item $D^\Sigma_\lambda(p)$ contains a (unique) symbol $\circ$ iff it does not contain any other symbol (in such a case, we day that $p$ is an empty position).
\end{itemize}

From the definition of $\Sigma\in \mathbb{S}$, it is easy to see that $|s-k|\leq 1$, so the collection $D^\Sigma_\lambda(p)$ may contain at most one symbol $>$ or $<$. Furthermore, for $\Sigma=\distinguished$, each $D^\Sigma_\lambda(p)$ exactly one symbol (perhaps $\circ$).

    This diagram can by interpreted visually by drawing the symbols corresponding to a $p\in \mathbb{Z}$ at the position $p$ on the real line. 
    
    For example, if we consider $\lambda\in \Lambda^+_{3|3}$ such that $\overline{\lambda}=(4\, 2\, 0\,| 0\, -3\, -5)$. Then the diagram $D^{\distinguished}_\lambda$ is given by 

    
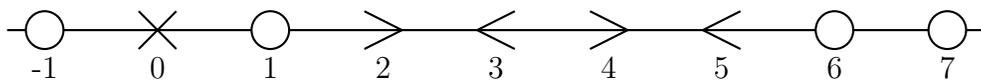
\begin{equation*}
    \begin{tikzpicture}
    
        \draw[thick, -] (-6.5 ,0)--(6.5, 0);
        
        \draw[thick, fill=white] (-6,0) circle [radius=0.25cm];
        \node[] at (-6,-0.5) {-1};
        
        % times
        \draw[thick, -] (-4.5,0)--(-4.5 + 0.25 ,0.25);
        \draw[thick, -] (-4.5,0)--(-4.5 - 0.25 ,0.25);
        \draw[thick, -] (-4.5,0)--(-4.5 + 0.25 ,-0.25);
        \draw[thick, -] (-4.5,0)--(-4.5 - 0.25 ,-0.25);
        \node[] at (-4.5,-0.5) {0};
        
        \draw[thick, fill=white] (-3,0) circle [radius=0.25cm];
        \node[] at (-3,-0.5) {1};
        
        % ge
        \draw[thick, -] (-1.5 +0.25, 0)--(-1.5 - 0.25, 0.25);
        \draw[thick, -] (-1.5 +0.25, 0)--(-1.5 - 0.25, -0.25);
        \node[] at (-1.5,-0.5) {2};
        
         % le
        \draw[thick, -] (0 -0.25, 0)--(0 + 0.25, 0.25);
        \draw[thick, -] (0 -0.25, 0)--(0 + 0.25, -0.25);
        \node[] at (0,-0.5) {3};
        
        % ge
        \draw[thick, -] (1.5 +0.25, 0)--(1.5 - 0.25, 0.25);
        \draw[thick, -] (1.5 +0.25, 0)--(1.5 - 0.25, -0.25);
        \node[] at (1.5,-0.5) {4};
        
        % le
        \draw[thick, -] (3 -0.25, 0)--(3 + 0.25, 0.25);
        \draw[thick, -] (3 -0.25, 0)--(3 + 0.25, -0.25);
        \node[] at (3,-0.5) {5};
        
        \draw[thick, fill=white] (4.5,0) circle [radius=0.25cm];
        \node[] at (4.5,-0.5) {6};
        
        \draw[thick, fill=white] (6,0) circle [radius=0.25cm];
        \node[] at (6,-0.5) {7};
        
        % cross
        %\draw[thick, -] (-6,0)--(-6 + 0.25 ,0.25);
        %\draw[thick, -] (-6,0)--(-6 - 0.25 ,0.25);
        %\draw[thick, -] (-6,0)--(-6 + 0.25 ,-0.25);
        %\draw[thick, -] (-6,0)--(-6 - 0.25 ,-0.25);
        
        % less than 
        %\draw[thick, -] (-3 -0.25, 0)--(-3 + 0.25, 0.25);
        %\draw[thick, -] (-3 -0.25, 0)--(-3 + 0.25, -0.25);
        
        % greater than 
        %\draw[thick, -] (-3 +0.25, 0)--(-3 - 0.25, 0.25);
        %\draw[thick, -] (-3 +0.25, 0)--(-3 - 0.25, -0.25);
        
    \end{tikzpicture}
\end{equation*}

    A natural question that arises from the definition of weight diagrams is: given a integrable simple finite dimensional $\glmn$-module $L$, what are all possible weight diagrams of highest weights of $L$ with respect to the different bases; that is, given $\lambda\in \distHW$, can we describe all the possible diagrams $D^{\Sigma}_{\lambda}$? 
    
    This is a complex combinatorial question, but in this paper we attempt to give a satisfactory answer to this question when $\Sigma$ is any base in $\mathbb{S}$.
    
\subsection{Main results}

The partial order on $\mathfrak{h}^*$ given by $\distinguished$ induces a partial order on $\mathcal{R}$: $\ror{i}{j} \le \ror{i'}{j'}$ if $i \ge i'$ and $j \le j'$. A subset of $\mathcal{R}$ is called {\it incomparable} if no two elements of this subset can be compared in the partial order.

Our first result describes the bases in $\mathbb{S}$ via the right odd roots lying in each base.
\begin{introtheorem}[See Lemma \ref{ror are incomparable}, Theorem \ref{B Sigma}]\label{introthrm:1} 

\mbox{}

\begin{enumerate}
    \item  There is a bijection between $\mathbb{S}$ and the collection of incomparable subsets of $\mathcal{R}$. This bijection is given by sending $\Sigma\in \mathbb{S}$ to $\Sigma\cap \mathcal{R}$.
    \item Let $B_\Sigma := \mathcal{R} \setminus \{\alpha \in \mathcal{R} | \; \exists \beta \in \Sigma \text{ so that } \alpha \geq \beta \}$ be the collection of all right odd roots which are not greater than some element of $\Sigma$.
    
    Then $\Sigma$ can be produced from $\distinguished$ by performing odd reflections with respect to every odd root in $B_\Sigma$, in some valid order. 
\end{enumerate}
\end{introtheorem}

To give a visualization of this statement, we interpret the partially ordered set $\mathcal{R}$ as a lattice of size $m \times n$, and we circle elements of $\Sigma\cap \mathcal{R}$.

\begin{example}
The diagram below describes a base $\Sigma\in \mathbb{S}$ for $\mathfrak{gl}(4|4)$, such that $\Sigma \cap \mathcal{R}=\{\ror{1}{1}, \ror{2}{2},\ror{4}{4}\}$. The elements of $\Sigma \cap \mathcal{R}$ are circled in blue, while the set $B_\Sigma$ is marked by a red polygon. 
    
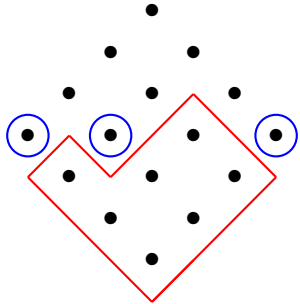
\begin{equation*}
\begin{tikzpicture}[anchorbase,scale=1.1]
% row 1
\node at (0,1.5) {$\bullet $};

% row 2
\node at (0.5,1) {$\bullet $};
\node at (-0.5, 1) {$\bullet $};

% row 3
\node at (1, 0.5) {$\bullet $};
\node at (0, 0.5) {$\bullet $};
\node at (-1, 0.5) {$\bullet $};

% row 4
\node at (1.5, 0) {$\bullet $};
\node at (0.5, 0) {$\bullet $};
\node at (-0.5, 0) {$\bullet $};
\node at (-1.5, 0) {$\bullet $};
\draw[thick, blue] (-1.5,0) circle [radius=0.25cm];
\draw[thick, blue] (-0.5,0) circle [radius=0.25cm];
\draw[thick, blue] (1.5,0) circle [radius=0.25cm];
% row 5
\node at (1, -0.5) {$\bullet $};
\node at (0, -0.5) {$\bullet $};
\node at (-1, -0.5) {$\bullet $};

% row 6
\node at (0.5, -1) {$\bullet $};
\node at (-0.5, -1) {$\bullet $};

% row 7
\node at (0, -1.5) {$\bullet $};

\draw[thick, red, -] (0.5 , 0.5)--(-0.5 ,-0.5);
\draw[thick, red, -] (0.5 , 0.5)--(1.5 ,-0.5);
\draw[thick, red, -] (1.5 ,-0.5)--(0 ,-2);

\draw[thick, red, -] (-1 , 0)--(-1.5 ,-0.5);
\draw[thick, red, -] (-1 , 0)--(-0.5 , -0.5);
\draw[thick, red, -] (-1.5 , -0.5)--(0 ,-2);
\draw[thick, red, -] (0.5 , -1.5)--(0 ,-2);

\end{tikzpicture}
\end{equation*} 
\end{example}

Understanding the result of performing odd reflections by right odd roots is the central question required to understand which weight diagrams correspond to a simple integrable finite-dimensional module $L$. If we denote by $\lambda$ the highest weight of $L$ with respect to $\distinguished$, we are interested in understanding the set of diagrams $\{D_{\lambda}^{\Sigma}\,|\, \Sigma\in \mathbb{S}\}$.

This is done using a new combinatorial tool which we introduce in this paper: the {\it arrow diagram} of $\lambda$.

Given $\lambda\in \distHW$, we define the arrow diagram of $\lambda$ by drawing $m$ arrows on the diagram $D^{\distinguished}_\lambda$ according to the following rules:
\begin{itemize}
    \item The arrows start at positions $\overline{\lambda}_m< \overline{\lambda}_{m-1} < \ldots <\overline{\lambda}_1$.
    \item The order in which we draw the arrows is that of the starting points (from left to right).
    \item Each arrow ends at the first position $p$ to the right of the starting point that has $D^{\distinguished}_\lambda (p) \in \{\circ, >\}$ and such that $p$ is not the endpoint of any previously drawn arrow.
\end{itemize}

We denote by $k_i$ the position of the endpoint of the arrow starting at position $\overline{\lambda}_i$ and by $M_i$ the total number of $<$ and $\times$ symbols in $D^{\distinguished}_\lambda$ to the left of position $k_i$.

\begin{example}
    Consider $\lambda \in \Lambda^+_{4|4}$ given by $\overline{\lambda}_{\distinguished}=(4\, 3\, 1\, 0| 0\, -1\, -4\, -5)$. The arrow diagram of of $\lambda$ is given below

    
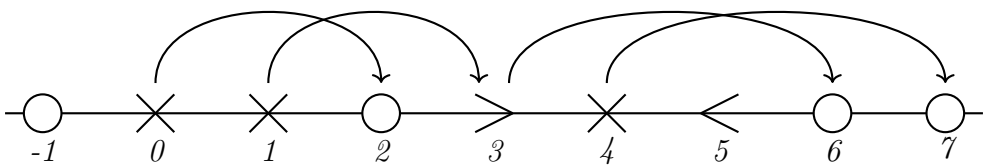
\begin{equation*}
\label{first arrow diagram}
    \begin{tikzpicture}
    
        \draw[thick, -] (-6.5 ,0)--(6.5, 0);
        
        \draw[thick, fill=white] (-6,0) circle [radius=0.25cm];
        \node[] at (-6,-0.5) {-1};
        
        % times
        \draw[thick, -] (-4.5,0)--(-4.5 + 0.25 ,0.25);
        \draw[thick, -] (-4.5,0)--(-4.5 - 0.25 ,0.25);
        \draw[thick, -] (-4.5,0)--(-4.5 + 0.25 ,-0.25);
        \draw[thick, -] (-4.5,0)--(-4.5 - 0.25 ,-0.25);
        \node[] at (-4.5,-0.5) {0};
        
        % times
        \draw[thick, -] (-3,0)--(-3 + 0.25 ,0.25);
        \draw[thick, -] (-3,0)--(-3 - 0.25 ,0.25);
        \draw[thick, -] (-3,0)--(-3 + 0.25 ,-0.25);
        \draw[thick, -] (-3,0)--(-3 - 0.25 ,-0.25);
        \node[] at (-3,-0.5) {1};
        
        \draw[thick, fill=white] (-1.5,0) circle [radius=0.25cm];
        \node[] at (-1.5,-0.5) {2};
        
        % ge
        \draw[thick, -] (0 +0.25, 0)--(0 - 0.25, 0.25);
        \draw[thick, -] (0 +0.25, 0)--(0 - 0.25, -0.25);
        \node[] at (0,-0.5) {3};
        
        % times
        \draw[thick, -] (1.5,0)--(1.5 + 0.25 ,0.25);
        \draw[thick, -] (1.5,0)--(1.5 - 0.25 ,0.25);
        \draw[thick, -] (1.5,0)--(1.5 + 0.25 ,-0.25);
        \draw[thick, -] (1.5,0)--(1.5 - 0.25 ,-0.25);
        \node[] at (1.5,-0.5) {4};
        
        % le
        \draw[thick, -] (3 -0.25, 0)--(3 + 0.25, 0.25);
        \draw[thick, -] (3 -0.25, 0)--(3 + 0.25, -0.25);
        \node[] at (3,-0.5) {5};
        
        \draw[thick, fill=white] (4.5,0) circle [radius=0.25cm];
        \node[] at (4.5,-0.5) {6};
        
        \draw[thick, fill=white] (6,0) circle [radius=0.25cm];
        \node[] at (6,-0.5) {7};
        
        % arrows
        \draw[thick, ->] (-4.5 ,0.4) .. controls (-4.5 ,1.65) and (-1.5 , 1.65) .. (-1.5 ,0.4);
        \draw[thick, ->] (-3,0.4) .. controls (-3 ,1.65) and (0 -0.2, 1.65) .. (0 -0.2,0.4);
        \draw[thick, ->] (0.2 ,0.4) .. controls (0.2 ,1.65) and (4.5 , 1.65) .. (4.5 ,0.4);

        \draw[thick, ->] (1.5 ,0.4) .. controls (1.5 ,1.65) and (6 , 1.65) .. (6 ,0.4);
        
        % cross
        %\draw[thick, -] (-6,0)--(-6 + 0.25 ,0.25);
        %\draw[thick, -] (-6,0)--(-6 - 0.25 ,0.25);
        %\draw[thick, -] (-6,0)--(-6 + 0.25 ,-0.25);
        %\draw[thick, -] (-6,0)--(-6 - 0.25 ,-0.25);
        
        % less than 
        %\draw[thick, -] (-3 -0.25, 0)--(-3 + 0.25, 0.25);
        %\draw[thick, -] (-3 -0.25, 0)--(-3 + 0.25, -0.25);
        
        % greater than 
        %\draw[thick, -] (-3 +0.25, 0)--(-3 - 0.25, 0.25);
        %\draw[thick, -] (-3 +0.25, 0)--(-3 - 0.25, -0.25);
        
    \end{tikzpicture}
\end{equation*}

    Here $k_4 = 2, k_3 = 3, k_2=6$ and $k_1=7$. $M_4 = M_3 = 2$ and $M_2=M_1=4$.
    
\end{example}
We now examine the difference between the diagrams $D^{r_\alpha \Sigma}_\lambda$ and $D^{\Sigma}_\lambda$ for $\alpha\in \Sigma \cap \mathcal{R}$. If these diagrams are indeed different, we say that {\it $r_{\alpha}$ changes the diagram $D^{\Sigma}_\lambda$}.
\begin{introtheorem}
    [See  Proposition \ref{changes the weight diagram of lambda}, Theorem \ref{theorem for ror change inequality}]\label{introthrm:2}  Let $\lambda \in \distHW$.
    \begin{enumerate}
        \item Let $\Sigma, \Sigma' \in \mathbb{S}$ and $\alpha \in \Sigma \cap \cap \Sigma' \cap \mathcal{R} $. Then $D^{r_\alpha \Sigma}_\lambda \neq D^{\Sigma}_\lambda \Rightarrow D^{r_\alpha \Sigma'}_\lambda \neq D^{\Sigma'}_\lambda$. \\ 
        That is, given $\alpha\in\mathcal{R}$, the odd reflection $r_{\alpha}$ will either change the weight diagram with respect to any bases in $\mathbb{S}$ containing it, or change none of these diagrams.
        \item Let $\ror{i}{j} \in \mathcal{R}$. The following statements are equivalent: 
        \begin{itemize}
            \item The odd reflection $r_{\ror{i}{j}}$ changes the diagram $D^{\Sigma}_\lambda$ for any base $\Sigma\in \mathbb{S}$ such that $\ror{i}{j}\in\Sigma$.
            \item The indices $i, j$ satisfy: $M_i - (k_i - \overline{\lambda}_i -1) \le j \le M_i$.
        \end{itemize} 
    \end{enumerate}
\end{introtheorem}

Combining \cref{introthrm:1} and \cref{introthrm:2}, we obtain a simple way to compute the weight diagram $D^{\Sigma}_\lambda$ with respect to any $\Sigma \in \mathbb{S}$. 

We define the change tracking diagram (CTD) of $\lambda\in \distHW$ as the function $c_\lambda: \mathcal{R}\to \{0,1\}$ given by 
    $$c_\lambda(\ror{i}{j})=
    \begin{cases}
        1 & M_i - (k_i - \overline{\lambda}_i -1) \le j \le M_i \\
        0 & \text{otherwise}
    \end{cases}$$

\begin{introcor} [See Corollary \ref{CTD formula}] Let $\lambda\in \distHW$. Then 
    $\overline{\lambda}_\Sigma = \overline{\lambda} + \sum_{\alpha\in B_\Sigma} c_\lambda(\alpha)\cdot \alpha$.
    \comment{
    %That is to say, the diagram of $\lambda$ with respect to base $\Sigma\in \mathbb{S}$ is given by shifted weight given by adding every root in $B_\Sigma$ that changes the diagram of $\lambda$ (as shown in Theorem 2) to the shifted weight defining $D_\lambda$. 
    }
\end{introcor}

We may interpret the CTD of $\lambda$ by expanding on the idea of the visual representation of $B_\Sigma$ seen above. This is done by coloring the node corresponding to every $\alpha\in \mathcal{R}$ according to the value of $c_\lambda(\alpha)$ (black if $c_\lambda(\alpha)=1$ and white otherwise). 
\begin{example}
For $\lambda\in \Lambda^+_{3|5}$ given by $\overline{\lambda} = (5\, 4\, 0|0\, -1\, -4\, -5\, -6)$ we have $c_\lambda$ given by 
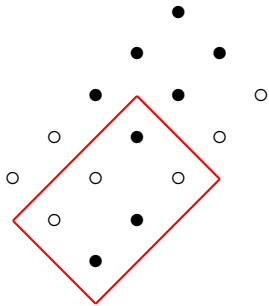
\begin{equation*}
\begin{tikzpicture}[anchorbase,scale=1.1]
% row 1
\node at (0,1.5) {$\bullet $};

% row 2
\node at (0.5,1) {$\bullet $};
\node at (-0.5, 1) {$\bullet $};

% row 3
\node at (1, 0.5) {$\circ $};
\node at (0, 0.5) {$\bullet $};
\node at (-1, 0.5) {$\bullet $};

% row 4
\node at (0.5, 0) {$\circ $};
\node at (-0.5, 0) {$\bullet $};
\node at (-1.5, 0) {$\circ $};

% row 5
\node at (0, -0.5) {$\circ $};
\node at (-1, -0.5) {$\circ $};
\node at (-2, -0.5) {$\circ $};

% row 6
\node at (-0.5, -1) {$\bullet $};
\node at (-1.5, -1) {$\circ $};

% row 7
\node at (-1, -1.5) {$\bullet $};

%\draw[-] (1.5,0.5)--(1.5,-1);
\draw[thick, red, -] (-1, -2)--(-2 ,-1);
\draw[thick, red, -] (-2 ,-1)--(-0.5 ,0.5);
\draw[thick, red, -] (-0.5 ,0.5)--(0.5 ,-0.5);
\draw[thick, red, -] (0.5 ,-0.5)--(-1, -2);

\end{tikzpicture}
\end{equation*}

Here we marked the set $B_\Sigma$ for $\Sigma\in \mathbb{S}$ by a red polygon. As described in Theorem \ref{introthrm:1}, $\Sigma$ corresponds to an incomparable set $\Sigma \cap \mathcal{R}$, which in this case is $\{\ror{1}{1},\ror{1}{4}\}$. Were we to compute $\overline{\lambda}_\Sigma$, this can be done by adding  to $\overline{\lambda}$ the right odd roots whose corresponding nodes are black and which are enclosed by the red polygon.
\end{example}
\subsection{The longtail of an integral dominant weight}

\mbox{}

Consider a symmetric bilinear form $(-|-)$ on $\fh^*$ defined by $$(\eps_i|\eps_j):=\delta_{i, j}, \; (\eps_i|\delta_j):=(\delta_j|\eps_i):=0, \; (\delta_i|\delta_j):=-\delta_{i, j}.$$

Let $\nu\in span_{\mathbb{Z}}
\{\eps_1, \ldots, \eps_m, \delta_1,\ldots, \delta_n\} \subset \fh^*$ be an integral weight for $\glmn$.  We define $s(\nu)$ to be the cardinality of the largest set of odd roots $S$ such that $$\forall \alpha, \beta\in S, \;\; (\alpha|\beta)=0, \, (\alpha|\nu)=0.$$ 
For $\lambda\in \distHW$ we define $$longtail(\lambda):=\max_{\Sigma\in \mathbb{S}}s(\overline{\lambda}_\Sigma).$$

\begin{introtheorem}[See Corollary \ref{cor:CTD_vs_longtail_eq}, Theorem \ref{thrm:max tail arrow}, Corollary \ref{thrm:longtail_and_cap_diagrams_formula}]\label{introthrm:4}
    The following values are equal to $longtail(\lambda)$:
    \begin{enumerate}
        \item The maximal cardinality of an incomparable subset $S\subset\mathcal{R}$ such that $c_{\lambda}(\alpha)=1$ for all $\alpha\in S$.
        \item $\max_{p\in \mathbb{R}}\sharp \{i \in \{1,\dots, m\}|p \in [\overline{\lambda}_i,k_i]\}$ which is to say the long arrow of $\lambda$ is the maximal number of arrows going over a single point in the arrow diagram of $\lambda$. 
        \item The maximal number of caps going over a single point in the cap diagram of $\lambda$, as defined in \cite{weightdiagraminvetion}.
    \end{enumerate}
\end{introtheorem}

The above theorem is a correction to the Tail Conjecture appearing in \cite{tailconjpaper} (see \cref{tail_conj}).

\subsection{Structure of the paper}
Section \ref{prelims} provides the required preliminaries on Lie superalgebras and weight diagrams. The remainder of this paper is divided into 4 sections: 

Section \ref{b sigma section} describes two approaches to describe bases in $\mathbb{S}$. Section \ref{arrow diagram section} describes the arrow diagram and provides insight into the changes to weight diagrams caused by odd reflections. Section \ref{ctd section} defines the CDT which can be used to compute the weight diagrams by any $\Sigma \in \mathbb{S}$ and explores it's structure. Finally, Section \ref{tail conj section} provides a counter example and correction to the tail conjecture appearing in \cite{tailconjpaper}.

\subsection{Acknowledgements}
During this research, the author was supported by the NSF-BSF grant 2019694 (PI: Inna Entova-Aizenbud, Vera Serganova).

I would like to thank my advisor Dr. Inna Entova-Aizenbud for the support and advice through out this project and Prof. Maria Gorelik for her help and interest in my work.

\section{Preliminaries}
\label{prelims}

\subsection{Vector superspaces and Lie superalgebras}
\mbox{}

In this section we follow \cite[Chapter 1]{cheng2012dualities} with some adaptation of notation. 

\begin{definition}
A \textbf{vector superspace} $V$ is a vector space together with $ \mathbb{Z}/2\mathbb{Z}$-grading $V=\even{V}\oplus \odd{V}$. The \textbf{dimension} of $V$ is given by the pair $\dim V = (\dim \even{V}|\,\dim  \odd{V})$, and the \textbf{superdimension} of $V$ is defined as $sdim V=\dim \even{V}-\dim  \odd{V}$. 

The \textbf{parity} of a homogeneous $a\in V_i$ is denoted by $p(a)=i$ for $i\in \mathbb{Z}/2\mathbb{Z}$. 
\end{definition}

\begin{example}
We denote by $\mathbb{C}^{m|n}$ the super vector space with even subspace $\mathbb{C}^m$ and odd subspace $\mathbb{C}^n$.
\end{example}

\begin{definition} An associative \textbf{superalgebra} $A$ is a super vector space $A=\even{A}\oplus \odd{A}$, with a multiplication operation $A\otimes A \to A$ that satisfies $A_iA_j\subseteq A_{i+j}$ for $i,j\in \mathbb{Z}/2\mathbb{Z}$ and which is associative. 
\end{definition}

\begin{definition} 
\label{module over superalgebra}
A \textbf{module} $M$ over a superalgebra $A$ is a module in the traditional sense, that must also agree with the $ \mathbb{Z}/2\mathbb{Z}$-grading. That is, $M=\even{M}\oplus \odd{M}$ and $A_iM_j\subseteq M_{i+j}$ for $i,j\in \mathbb{Z}/2\mathbb{Z}$.
\end{definition}

\begin{definition}
A \textbf{Lie superalgebra} is a vector superspace $\mathfrak{g}=\even{\mathfrak{g}}\oplus \odd{\mathfrak{g}}$ with a binary bilinear operation $[\cdot,\cdot]$ satisfying the following axioms for homogeneous $a,b,c\in \mathfrak{g}$:
\begin{enumerate}
    \item Supersymmetry: $[a,b]=-(-1)^{p(a)\cdot p(b)}[b,a]$.
    \item Super Jacobi identity: $[a,[b,c]]=[[a,b],c]+(-1)^{p(a)\cdot p(b)}[b,[a,c]]$.
\end{enumerate}
\end{definition}

\begin{example}
Any Lie algebra $\mathfrak{g}$ is a Lie superalgebra with even part given by $\mathfrak{g}$ and odd part equal to $\{0\}$.
\end{example}

\begin{example}
\label{commutator example}
Let $A=\even{A}\oplus \odd{A}$ be an associative superalgebra. We can give it the structure of a superalgebra by defining the bracket by the supercommutator 
$$[a,b]=ab-(-1)^{p(a)\cdot p(b)}ba$$ for homogeneous $a,b\in A$, and extending $[\cdot,\cdot]$ bilinearly.
\end{example}

\begin{definition}[The general linear Lie superalgebra] Let $V=\even{V}\oplus \odd{V}$ be a super vector space. The algebra $\text{End}(V)$ is naturally an associative superalgebra, and as in Example \ref{commutator example} we may equip it with the supercommutator to form a Lie superalgebra called the \textbf{general linear Lie superalgebra}. This superalgebra is denoted by $\mathfrak{gl}(V)$. 

When $V=\mathbb{C}^{m|n}$ we write $\glmn:=\mathfrak{gl}(V)$. 
\end{definition}
Observe that for a vector superspace $V=V_{\bar 0} \oplus V_{\bar{1}}$, we have: $\even{\mathfrak{gl}(V)}\cong \mathfrak{gl}(V_{\bar 0})\oplus \mathfrak{gl}(V_{\bar{1}})$ and $\odd{\mathfrak{gl}(V)}\cong (V_{\bar 0} \oplus V_{\bar{1}}^*) \oplus (V_{\bar 0}^* \oplus V_{\bar{1}})$.

Fixing a basis for $\even{V}$ and for $\odd{V}$ that combine to a basis of $V$, elements of $\mathfrak{gl}(V)$ can be realized as a superalgebra of $(m+n)\times (m+n)$ matrices, with basis given by elementary matrices $E_{ij}$.

\begin{definition}
    A \textbf{Cartan subalgebra} of $\glmn$ is defined as a Cartan subalgebra of the even subalgebra $\even{\glmn}\cong \mathfrak{gl}(m)\oplus\mathfrak{gl}(n)$ (which in turn is defined as a subalgebra $\mathfrak{h}$ of $\even{\glmn}$ that is both nilpotent and equal to its normalizer).
\end{definition}

Note that Cartan subalgebras need not be unique, and there can (and often are) multiple Cartan subalgebras. Cartan subalgebras are conjugate under inner automorphisms.

\begin{example}
    The set of all diagonal matrices is a Cartan subalgebra of $\glmn$.
\end{example}

Note that the scope of this paper is limited to the Lie superalgebra $\glmn$. That being said, all Lie superalgebras with a reductive even part and an invariant symmetric form have Cartan subalgebras. 

Furthermore, all definitions and statements presented in this section can be stated for any Lie superalgebra equipped with a grading $\fg = \fg(-1)\oplus\fg(0)\oplus \fg(1)$ where $\fg(0)=\even{\fg}$ is a reductive Lie algebra and $\fg(-1)$ and $\fg(1)$ are irreducible $\fg_0$-modules.

\begin{definition}
    Let $\mathfrak{h}$ be a Cartan subalgebra of $\glmn$.
    Consider the actions of the elements of $\mathfrak{h}$ on $\mathfrak{g}$ by the adjoint action (as in Example \ref{adjoint action}). These can be simultaneously diagonalized, so there exists a decomposition of $\mathfrak{g}$ into eigenspaces
        $$\mathfrak{g}=\mathfrak{h}\oplus \underset{\alpha\in\mathfrak{h}^*}{\bigoplus} \mathfrak{g}_\alpha\text{.}$$
    where $\alpha\in \mathfrak{h}^*$ and $\mathfrak{g}_\alpha=\{g\in\mathfrak{g}\,|\,\forall h\in\mathfrak{h}\,[h,g]=\alpha(h)g\}$. 
    We call $\alpha$ such that $\mathfrak{g}_\alpha \neq 0$ a \textbf{root} and $\mathfrak{g}_\alpha$ a \textbf{root space}. 
    The set of all roots $\Delta\subset \fh^*$ is called the \textbf{root system} of $\mathfrak{gl}(m|n)$. Roots can be assigned parity via the parity of their respective root spaces. As such $\Delta$ can be split into even and odd parts denoted by $\even{\Delta}$ and $\odd{\Delta}$ respectively. 
\end{definition}

\begin{theorem}
    Let $\mathfrak{h}$ be a Cartan subalgebra of $\glmn$. We have:
    \begin{enumerate}       
        \item $\dim\mathfrak{g}_\alpha=1$, for $\alpha\in\Delta$.
        \item $[\mathfrak{g}_\alpha,\mathfrak{g}_\beta]\subset \mathfrak{g}_{\alpha+\beta}$ where $\mathfrak{h}$ is considered to be $\mathfrak{g}_0$.
        \item $\Delta=-\Delta$, $\even{\Delta}=-\even{\Delta}$ and $\odd{\Delta}=-\odd{\Delta}$.
    \end{enumerate}
\end{theorem}

\begin{definition}
\label{ weight def}
    Fix a Cartan subalgebra $\fh$ of $\glmn$. Functionals $\lambda\in \fh^*$ are called \textbf{weights}. 
    Given a simple $\glmn$-module $M$, any subspace $V$ of dimension one is a weight space in the sense that there exists $\lambda\in \fh^*$ such that $\{v\in V|\,\forall h\in\fh\, h.v=\lambda(h)v\}$.

    Notice that the set of roots $\Delta$ is in fact a set of weights, that is to say $\Delta \subset \fh^*$.
\end{definition}

\begin{example}
\label{usual cartan}
    Let $\mathfrak{h}$ be the Cartan subalgebra  of $\glmn$ given by all diagonal matrices. For $1\le i \le m$ and $1\le j \le n$, denote by $\varepsilon_i$ and $\delta_j$ the weights which are dual to the matrices $E_{i,\,i}\in\mathfrak{h}$ and $E_{m+j,\,m+j}\in\mathfrak{h}$. 
    
    Then $\Delta_{\overline{0}}=\{\pm (\varepsilon_i-\varepsilon_j)\}_{1\leq i, j\leq m} \cup \{\pm (\delta_i-\delta_j)\}_{1\leq i, j\leq n}$, $\Delta_{\overline{1}} = \{\pm (\ror{i}{j})\}_{1\leq i\leq m, \, 1\leq j\leq n}$.
\end{example}

Defining $\varepsilon_i$ and $\delta_j$ as in Example \ref{usual cartan}, any weight $\nu\in\fh^*$ is given by $$\nu = \sum_{i=1}^{m}a_i\varepsilon_i + \sum_{j=1}^{n}b_j\delta_j.$$ For convenience we write $\gram$ as shorthand for $\sum_{i=1}^{m}a_i\varepsilon_i + \sum_{j=1}^{n}b_j\delta_j$ and write $\nu_k:=a_k$ and $\nu_{-k}:=b_k$.

Further, one may equip $\fh^*$ with a symmetric bilinear form given by $$(\eps_i|\eps_j):=\delta_{ij}, \; (\delta_j|\eps_i) :=(\eps_i|\delta_j) :=-\delta_{ij} \; (\delta_i|\delta_j):=-\delta_{ij}$$ for all $i, j$.

\begin{definition}
    A {\it base} of $\Delta$ is a linearly independent set $\Sigma \subset \Delta$ such that every root $\alpha\in \Delta$ can be written as a linear combination of the elements of $\Sigma$, with integer coefficients, so that the coefficients are either all non-positive or all non-negative.
\end{definition}
The following is a classical result appearing in \cite[Chapter 10]{intro1972liealgebras}.
\begin{lemma}
    Let $\Delta = \Delta^+\sqcup \Delta^-$ be a decomposition of the root system such that $$\alpha\in \Delta^+ \,\Longleftrightarrow \, -\alpha\in \Delta^-.$$ Consider the simple roots with respect to this decomposition: these are roots that cannot be written as a sum of two positive roots with respect to this decomposition. The set of all simple roots is a base. 

    Conversely, given a base $\Sigma$, denote by $\Delta^+$ the set of all roots which are linear combinations of elements in $\Sigma$ with non-negative coefficients. Then $\Delta = \Delta^+ \sqcup (\Delta\setminus \Delta^+)$ is a a decomposition of $\Delta$ into sets of positive and negative roots, satisfying the requirement above.
\end{lemma}

\begin{definition}
    Consider $\glmn$ with the Cartan subalgebra as described in \cref{usual cartan} and set of positive roots given by $$\Delta^+=\{\eps_i-\eps_j|1\le i < j \le m\}\cup \{\delta_i-\delta_j|1\le i < j \le n\}\cup\{\ror{i}{j}|i\in\{1,\dots,m\}\text{ and }j\in\{1,\dots,n\}\}.$$ We call the base corresponding to $\Delta^+$ the \textbf{distinguished base} and denote it by $\distinguished$. 
\end{definition}

\begin{example}
    For $\mathfrak{gl}(2|2)$, we have:
    $\Delta = \{\pm (\ror{i}{j})|1\le i,j \le 2\}$. The distinguished base is given by base $$\distinguished= \{\eps_1-\eps_2,\eps_2-\delta_1, \delta_1-\delta_2\}.$$
    The set of positive and negative roots is given by the set of roots in $\Delta$ which can we written as $\sum k_i \alpha_i$ where $\alpha_i\in \distinguished$ and $k_i$ are either all non-negative and all non-positive respectively. 

    As such, the set of positive roots with respect to $\distinguished$ is 
    $$\Delta^+=\{\eps_1-\eps_2, \eps_2-\delta_1, \delta_1-\delta_2, \eps_1-\delta_1, \eps_2-\delta_2,\eps_1-\delta_2\}$$
    and $\Delta^- = - \Delta^+$.
\end{example}

%\begin{definition}
  %  Let $\fh$ be a Cartan subalgebra of $\glmn$ and $\Delta^+$ be some fixed set of positive roots. Denote by $\fn^\pm=\underset{\alpha\in\Delta^\pm}{\bigoplus}\fg_\alpha$. 

 %   The decomposition $\fg=\fn^-\oplus \fh\oplus \fn^+$ is called the \textbf{triangular decomposition of $\fg$}. The subalgebra $\fb=\fh\oplus \fn^+$ is called the \textbf{Borel subalgebra associated with $\Delta^+$} or simply the \textbf{Borel subalgebra} when there is no ambiguity. 
%\end{definition}

\begin{definition}
    A $\fg$-module $M$ is a vector superspace with a Lie superalgebra homomorphism $\fg\to \mathfrak{gl}(M)$.
\end{definition}

\begin{example}
    \label{adjoint action}
    Let $\mathfrak{g}$ be a Lie superalgebra. The adjoint module of $\mathfrak{g}$ is the superspace $\mathfrak{g}$; the action of $g\in\mathfrak{g}$ on $h\in\mathfrak{g}$ is given by $g.h=[g,h]$.
\end{example}

In this paper we consider only $\glmn$-modules that are finite dimensional and semisimple over $\fh$, with $\fh$ acting integrably (that is to say, the action of $\fh$ integrates to an algebraic action of the corresponding torus group).

\begin{definition}
\label{partial order}
    \comment{Fix some set of positive roots $\Delta^+$ of $\fh^*$ and consider the base $\Sigma$ corresponding to $\Delta^+$.} We define a partial order on $\fh^*$ by $\lambda \ge \mu$ if $\lambda - \mu=\sum_{\alpha\in\distinguished}a_\alpha \alpha$ where $a_\alpha\in \mathbb{R}_{\ge 0}$.
\end{definition}

\begin{fact}
    Let $L$ be a simple $\glmn$-module. As an $\fh$-module, $L=\oplus_{\lambda\in \fh^*}V_\lambda$ where $V_\lambda$ is trivial for all except finitely many $\lambda$. Fixing any base $\Sigma$, there exists a single highest weight $\lambda$ with respect to the partial order described in \cref{partial order}. Such $\lambda$ is called the highest weight of $L$ with respect to $\Sigma$.

    Moreover, any pair of simple $\glmn$-modules $M\not\cong N$ have distinct highest weights. Therefore, we denote the simple $\glmn$-module of highest weight $\lambda$ with respect to a base $\Sigma$ by $L_\Sigma(\lambda)$.
    
    When $\Sigma=\distinguished$, we write $L(\lambda)$.

    The highest weight $\lambda$ of any integrable simple module $L$ with respect to $\distinguished$ satisfies the conditions: $\lambda_i\in \Z$ for any $i$, and $\lambda_i\geq \lambda_{i+1}$ for $i<m$ and $i\geq m+1$. 
    
    The weights satisfying the first of these conditions are called {\it integral}, and the weights satisfying the the second condition are called {\it dominant}. Any integral dominant weight of $\glmn$ is the highest weight of some integrable simple module $L$ with respect to $\distinguished$. 
\end{fact}

\begin{notation}
\mbox{}
    \begin{itemize}
        \item Denote by $\distHW$ the set of all $\lambda \in \fh^*$ such that $\lambda$ is integral dominant with respect to $\distinguished$.
        \item Denote  $\Lambda^+:=\bigsqcup_{m, n\geq 0} \distHW$.
    \end{itemize}
\end{notation}

\subsection{Reflections}
\mbox{}

We now consider $\glmn$.

\begin{definition}
    Let $\Sigma$ be some base of $\Delta$. We denote the Weyl vector associated with $\Sigma$ by $$\rho_\Sigma := \frac{1}{2}\sum_{\alpha\in \even{\Delta}^+}\alpha -\frac{1}{2}\sum_{\beta\in \odd{\Delta}^+}\beta.$$ 

        We will usually write $\rho := \rho_{\distinguished}$.
\end{definition}

\begin{fact}
    For any $\alpha \in \Sigma$, $(\rho_\Sigma| \alpha)=\frac{1}{2}(\alpha|\alpha)$. 
\end{fact}
This fact is \cite[Proposition 1.28]{cheng2012dualities}.

\begin{remark}
When discussing the Weyl vector $\rho_\Sigma$, we will usually consider the Weyl vector shifted by $r(1\, \dots \, 1|-1\, \dots \, -1)$ for some $r\in \{0, 1/2\}$, so that all the coordinates of $\rho_\Sigma$ are integers (which need not always be true as the Weyl vector is a half sum). This shifted Weyl vector satisfies the above property as well.
\end{remark}

\begin{definition}
An \textbf{even reflection by $\alpha\in \even{\Delta}$} is an automorphism $s_\alpha: \Delta \to \Delta$ of the root system given by
    $$
    s_\alpha(\beta)=\beta-2\frac{(\beta,\alpha)}{(\alpha,\alpha)}\alpha. 
    $$
\end{definition}
The even reflections generate the Weyl group of the Lie superalgebra $\glmn$, which is isomorphic to $S_m\times S_n$.

This definition is the same as for root systems of Lie algebras. That being said, with parity coming into play, one may need to consider another type of reflection. This is because not every pair of bases $\Sigma$ and $\Sigma'$ of $\glmn$ are conjugate with respect to the action of the Weyl group.

We now introduce the notion of an odd reflection, due to Serganova in \cite{LSS86}.

\begin{definition}
\label{odd ref def}
   An \textbf{odd reflection by $\alpha\in \odd{\Delta}$} is an automorphism $r_\alpha: \Delta \to \Delta$ given by  
    $$
    r_\alpha(\beta)=
    \begin{cases}
        -\alpha &\text{ if } \beta = \alpha \\
        \beta + \alpha &\text{ if } \beta \neq \alpha \text{ and } (\beta, \alpha) \neq 0 \\
        \beta & \text{ otherwise}
    \end{cases}
    $$
\end{definition}

\begin{fact}
    Let $\Sigma$ be a base. For any $\alpha\in \odd{\Sigma}$ and $\beta\in \even{\Sigma}$, $r_\alpha (\Sigma)$ and $s_\beta (\Sigma)$ are also bases of $\Delta$.
\end{fact}

\begin{remark}
\label{ref set rmk}
    It is not immediately obvious why the odd reflection is a natural extension of the notion of even reflection. One way to see this connection is by considering the way even and odd reflections change the set of positive roots.
    
    Given a base $\Sigma$ and $\beta\in \even{\Sigma}$, we have $\Delta^+_{s_{\beta}\Sigma} = \Delta^+_\Sigma \setminus \{\beta\} \cup \{-\beta\} $.

    Similarly, for $\alpha \in \odd{\Sigma}$ we have $\Delta^+_{r_{\alpha}\Sigma} = \Delta^+_\Sigma \setminus \{\alpha\} \cup \{-\alpha\} $.
\end{remark}

We now consider how an odd reflection of a base affect the the highest weight of a given module with respect to this base.

\begin{fact}
\label{odd reflection hw}
  Let $\Sigma$ be a base for $\glmn$. Let $L$ be a simple integrable (finite dimensional) $\glmn$-module, and let $\alpha \in \odd{\Sigma}$. Let $\lambda$ be the highest weight of the module $L$ with respect to the base $ \Sigma$ and let $\lambda'$ be the highest weight of the module $L$ with respect to the base $r_\alpha \Sigma$; that is,
    $L_\Sigma(\lambda) \cong L_{r_\alpha \Sigma}(\lambda')$. Then $\lambda'$ is given by 
    $$
    \lambda' =
    \begin{cases}
        \lambda - \alpha &\text{ if }  (\lambda,\alpha)\neq 0 \\
        \lambda &\text{ if }  (\lambda,\alpha)=0.
    \end{cases}
    $$
\end{fact}

\begin{notation}
    Let $\lambda\in \distHW$ and let $\Sigma$ be a base for $\glmn$. We denote by $\lambda_{\Sigma}\in \fh^*$ the highest weight of $L(\lambda)$ with respect to the base $\Sigma$; that is, $L_{\distinguished}(\lambda) \cong L_\Sigma(\lambda_{\Sigma})$.
\end{notation}

By \cref{odd reflection hw}, the weights $\lambda_{\Sigma}$ are always integral.

\begin{fact}
\label{reflection weyl vec}
    Let $\Sigma$ be a base of $\Delta$. Then for $\alpha \in \even{\Sigma}$ we have
    $$\rho_{s_\alpha \Sigma} = \rho_\Sigma -\alpha.$$

    Similarly, for $\alpha \in \odd{\Sigma}$ we have
    $$\rho_{r_\alpha \Sigma} = \rho_\Sigma + \alpha.$$
\end{fact}
This fact appears as part of the proof of \cite[Proposition 1.28]{cheng2012dualities}.

\begin{definition}
    We say that a base $\Sigma$ {\bf agrees with the even part of} $\distinguished$ if $\Delta^+ _{\Sigma, \bar{0}}= \Delta^+ _{\distinguished, \bar{0}}$. Denote by $\mathbb{S}$ the set of all bases that agree with the even part of $\distinguished$. 
\end{definition}

\begin{remark}
    As Remark \ref{ref set rmk} states, even reflections change the set of positive even roots and while odd reflections change the set of positive odd roots. As such, $\mathbb{S}$ is clearly the set of all bases that can be produced from $\distinguished$ by performing only odd reflections with respect to roots in $\Delta^+_{\overline{1}}$ (the set of positive odd roots for the distinguished base).
\end{remark}

\mbox{}

\mbox{}

\subsection{Weight diagrams}

\begin{notation}
    Let $\lambda\in \distHW$ and $\Sigma \in \mathbb{S}$.
    We denote $\overline{\lambda}_\Sigma :=\lambda_\Sigma + \rho_\Sigma$ and we write $\overline{\lambda}:=\overline{\lambda}_{\distinguished}$ for short.
    
\end{notation}

\begin{notation}
   Given $\lambda \in \distHW$, denote by $$\text{Hwt}(\lambda):=\{\overline{\lambda}_\Sigma |\Sigma\in \mathbb{S}\}$$ the set of $\rho$-shifted weights which are highest weights of $L_{\distinguished}(\lambda)$ with respect to different bases in $\mathbb{S}$.
\end{notation}

\begin{remark}
    Note that $\overline{\lambda}_\Sigma$ is itself an element of $\fh ^ *$ and an integral weight, so we maintain the indexing convention described in Definition \ref{ weight def}.    
\end{remark}

\begin{definition}
    Let $\lambda \in \distHW$ and $\Sigma \in \mathbb{S}$. \textbf{The weight diagram of $\lambda$ with respect to $\Sigma$} is defined as a function on the integers, whose value in $p\in \mathbb{Z}$ is a finite collection of symbols $\times, >, <, \circ$. This function is subject to the following rules:
\begin{itemize}
    \item Let $k:=\sharp \{i\in \{1,\dots, m\}|(\overline{\lambda}_\Sigma)_i=p\}$, $s:= \sharp \{j\in \{1,\dots, n\}|(\overline{\lambda}_\Sigma)_{-j}=-p\}$. Then $D^\Sigma_\lambda(p)$ contains exactly $\min(k, s)$ symbols $\times$.
    \item If $k>s$ then $D^\Sigma_\lambda(p)$ contains exactly $k-s$ symbols $>$.
      \item If $k<s$ then $D^\Sigma_\lambda(p)$ contains exactly $s-k$ symbols $<$.
      \item $D^\Sigma_\lambda(p)$ contains a (unique) symbol $\circ$ iff it does not contain any other symbol (in such a case, we day that $p$ is an {\it empty} position).
\end{itemize}

\end{definition}
\begin{remark}
   For $\Sigma\in \mathbb{S}$, we have: $-1\leq s-k \leq 1$ by \cref{lem:generalities_on_ror_changes}. That is to say, in diagrams $D_{\lambda}^\Sigma$ for $\Sigma\in \mathbb{S}$, the number of positive coordinates such that $(\overline{\lambda}_\Sigma)_i=p$ and the number of negative coordinates such that $(\overline{\lambda}_\Sigma)_{-j}=-p$, differ by 1 at most (for any $p$). 
\end{remark}

\begin{notation}
      When not specified otherwise, we denote by $D_\lambda := D^{\distinguished}_\lambda$ the weight diagram of $\lambda$ with respect to the distinguished base $\distinguished$.
\end{notation}

\begin{remark}
  
  The weight diagram $D_\lambda$ has at most one symbol corresponding to each $p$; that is to say, for any $p\in\mathbb{Z}$, $D_\lambda(p)\in \{\circ, \times, >, <\}$. 
\end{remark}

Given $\lambda\in \distHW$, one may visualize $D^\Sigma_\lambda$ by drawing the symbols corresponding to each integer at the integer points of the real line: that is, at every point $p\in\mathbb{Z}$ we draw the symbols $D^\Sigma _\lambda(p)$ (expressions of the form $\times^k >$ or $\times^k<$ are drawn with the symbol $>$ or $<$ on the line and $k$ symbols $\times$ stacked vertically above it). 

\begin{example}
Consider the weight $\lambda \in \Lambda_{3|5}$ such that $\overline{\lambda}_\Sigma=  (4\, 4\, 2\,|-1\, -2\, -4\, -4\, -5)$ for some $\Sigma\in \mathbb{S}$. The diagram $D^\Sigma_\lambda$ is drawn as follows: 
        
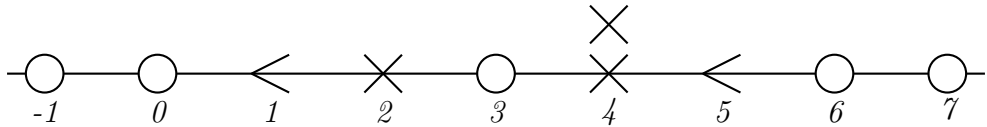
\begin{equation*}
    \begin{tikzpicture}
    
        \draw[thick, -] (-6.5 ,0)--(6.5, 0);
        
        \draw[thick, fill=white] (-6,0) circle [radius=0.25cm];
        \node[] at (-6,-0.5) {-1};
        
        % times
        \draw[thick, fill=white] (-4.5,0) circle [radius=0.25cm];
        \node[] at (-4.5,-0.5) {0};
        
        % le
        \draw[thick, -] (-3 -0.25, 0)--(-3 + 0.25, 0.25);
        \draw[thick, -] (-3 -0.25, 0)--(-3 + 0.25, -0.25);
        \node[] at (-3,-0.5) {1};

        %times
        \draw[thick, -] (-1.5,0)--(-1.5 + 0.25 ,0.25);
        \draw[thick, -] (-1.5,0)--(-1.5 - 0.25 ,0.25);
        \draw[thick, -] (-1.5,0)--(-1.5 + 0.25 ,-0.25);
        \draw[thick, -] (-1.5,0)--(-1.5 - 0.25 ,-0.25);
        \node[] at (-1.5,-0.5) {2};
            
        \draw[thick, fill=white] (0,0) circle [radius=0.25cm];
        \node[] at (0,-0.5) {3};

        % stacked times
        \draw[thick, -] (1.5, 0 + 0.65)--(1.5 + 0.25 ,0.25 + 0.65);
        \draw[thick, -] (1.5, 0 + 0.65)--(1.5 - 0.25 ,0.25 + 0.65);
        \draw[thick, -] (1.5, 0 + 0.65)--(1.5 + 0.25 ,-0.25 + 0.65);
        \draw[thick, -] (1.5, 0 + 0.65)--(1.5 - 0.25 ,-0.25 + 0.65);
        
        % times
        \draw[thick, -] (1.5,0)--(1.5 + 0.25 ,0.25);
        \draw[thick, -] (1.5,0)--(1.5 - 0.25 ,0.25);
        \draw[thick, -] (1.5,0)--(1.5 + 0.25 ,-0.25);
        \draw[thick, -] (1.5,0)--(1.5 - 0.25 ,-0.25);
        \node[] at (1.5,-0.5) {4};

        % le
        \draw[thick, -] (3 -0.25, 0)--(3 + 0.25, 0.25);
        \draw[thick, -] (3 -0.25, 0)--(3 + 0.25, -0.25);
        \node[] at (3,-0.5) {5};
        
        \draw[thick, fill=white] (4.5,0) circle [radius=0.25cm];
        \node[] at (4.5,-0.5) {6};
        
        \draw[thick, fill=white] (6,0) circle [radius=0.25cm];
        \node[] at (6,-0.5) {7};
        
        % cross
        %\draw[thick, -] (-6,0)--(-6 + 0.25 ,0.25);
        %\draw[thick, -] (-6,0)--(-6 - 0.25 ,0.25);
        %\draw[thick, -] (-6,0)--(-6 + 0.25 ,-0.25);
        %\draw[thick, -] (-6,0)--(-6 - 0.25 ,-0.25);
        
        % less than 
        %\draw[thick, -] (-3 -0.25, 0)--(-3 + 0.25, 0.25);
        %\draw[thick, -] (-3 -0.25, 0)--(-3 + 0.25, -0.25);
        
        % greater than 
        %\draw[thick, -] (-3 +0.25, 0)--(-3 - 0.25, 0.25);
        %\draw[thick, -] (-3 +0.25, 0)--(-3 - 0.25, -0.25);
        
    \end{tikzpicture}
\end{equation*}

\end{example}

Consider $\lambda\in \distHW$ and $\Sigma\in \mathbb{S}$. We may consider how an odd reflection by $\alpha\in \odd{\Sigma}$ changes the $\rho_{\Sigma}$-shifted weight $\overline{\lambda}_{\Sigma}$ and the diagram $D_{\lambda}^{\Sigma}$:
\begin{prop}
\label{odd reflection diagram}
    Let $\lambda \in \distHW$ and $\Sigma\in \mathbb{S}$.
    Let $\alpha = \pm (\ror{i}{j})\in \odd{\Sigma}$. Then 
    $$
    \overline{\lambda}_{r_\alpha \Sigma}=
    \begin{cases}
        \overline{\lambda}_\Sigma &\text{ if } (\overline{\lambda}_\Sigma)_i \neq -(\overline{\lambda}_\Sigma)_{-j} \\
        \overline{\lambda}_\Sigma + \alpha &\text{ if }  (\overline{\lambda}_\Sigma)_i = -(\overline{\lambda}_\Sigma)_{-j}.
    \end{cases}
    $$ 
\end{prop}
\begin{proof}
    Observe that $\pm ((\overline{\lambda}_\Sigma)_i + (\overline{\lambda}_\Sigma)_{-j}) =(\overline{\lambda}_\Sigma| \alpha) = (\lambda_\Sigma| \alpha) + (\rho_\Sigma| \alpha)= (\lambda_\Sigma| \alpha) + (\alpha|\alpha)$. 
    
    For any $\alpha\in \odd{\Delta}$ we have: $(\alpha| \alpha)=0$. So we have: $$(\overline{\lambda}_\Sigma)_i + (\overline{\lambda}_\Sigma)_{-j} = 0 \Longleftrightarrow (\lambda_\Sigma| \alpha) = 0.$$ By Fact \ref{reflection weyl vec} we find that $\rho_{r_\alpha \Sigma} = \rho_\Sigma + \alpha$, and by Fact \ref{odd reflection hw},
     $$\lambda_{r_\alpha \Sigma} =
     \begin{cases}
        \lambda_\Sigma - \alpha &\text{ if }  (\overline{\lambda}_\Sigma)_i \neq -(\overline{\lambda}_\Sigma)_{-j} \\
        \lambda_\Sigma & \text{ if } (\overline{\lambda}_\Sigma)_i = -(\overline{\lambda}_\Sigma)_{-j}
    \end{cases}
     $$

    As $\overline{\lambda_\Sigma} = \lambda_\Sigma + \rho_\Sigma$, we deduce the statement of the proposition.
\end{proof}

\begin{remark}\label{rmk:odd_reflection_on_diagrams}
    Let $\alpha = \pm (\ror{i}{j})$ for some $i, j$. Recall that  $\overline{\lambda}_{r_\alpha \Sigma}= \overline{\lambda}_\Sigma + \alpha$ only if $(\overline{\lambda}_\Sigma)_i = -(\overline{\lambda}_\Sigma)_{-j}$; the latter means that position $(\overline{\lambda}_\Sigma)_i$ in $D_{\Sigma}^{\lambda}$ has a $\times$ symbol.
    
    We deduce that if the diagram $D^{r_\alpha \Sigma}_{\lambda}$ differs from $D^{\Sigma}_{\lambda}$, then it differs by a single $\times$ symbol moved 1 position to the right if $\alpha = \ror{i}{j}$  for some $i, j$ or 1 position to the left if $\alpha = \delta_j - \eps_i$ for some $i, j$. 
    
\end{remark}

\begin{notation}
    Let $\lambda\in \distHW$, let $\alpha \in \odd{\Sigma}$. We write $r_\alpha D^\Sigma _\lambda := D^{r_\alpha \Sigma}_\lambda$.
\end{notation}

\begin{notation}
    We denote by $\times_i$ the position of the $i$-th $\times$ symbol in the diagram $D_{\lambda}$, counted from left to right. 
\end{notation}
In particular, $\times_1$ is the position of the leftmost $\times$ symbol.

\subsection{Cap diagrams}
The cap diagrams are a combinatorial tool developed to study representations of $\glmn$ and similar representation theories. They were introduced by Brundan and Stroppel, see for example \cite{weightdiagraminvetion}.
\begin{definition}
\label{cap diagram def}
    The \textbf{cap diagram} of the weight diagram $D_{\lambda}$ is defined as follows: for every symbol $\times$ in position $p$ in, we draw a cap from position $p$ to position $q>p$ if $D_\lambda (q)=\circ$ and we have: 
    
\begin{align*}
    \forall \, p\leq r < q, \;\;\;\;  &\sharp\{ p\leq s \leq r\,|\, D_{\lambda}(s) = \times \}> \sharp\{ p\leq s \leq r\,|\, D_{\lambda}(s) = \circ \} \\
    &\sharp\{ p\leq s \leq q\,|\, D_{\lambda}(s) = \times \}= \sharp\{ p\leq s \leq q\,|\, D_{\lambda}(s) = \circ \}. \\  
\end{align*}
\end{definition}

In other words, we draw a cap from position $p$ to position $q$ if for every $p\leq r< q$, the number of $\times$ symbols in the segment $[p, r]$ is greater than the number of $\circ$ symbols in the same segment, and these numbers are equal in the segment $[p, q]$.

\begin{remark}
    Note that the function
$$ g^{(p)}_{\times}(r):=\sharp\{ p\leq s \leq r\,|\, D_{\lambda}(s) = \times \} -  \sharp\{ p\leq s \leq r\,|\, D_{\lambda}(s) = \circ \}$$
satisfies: 
\begin{itemize}
    \item $g^{(p)}_{\times}(p)=1$ (since $D_{\lambda}(p)=\times$), 
    \item $g^{(p)}_{\times}(r)<0$ for $r>>l$ (due to the fact that there are finitely many $\times$ symbols and infinitely many $\circ$ symbols in $D_{\lambda}$), and 
    \item $|g^{(p)}_{\times}(r+1) -g^{(p)}_{\times}(r)| \leq 1$ for any $r\geq p$.
\end{itemize} 

So there exists $q\geq p$ such that $g^{(p)}_{\times}(q)=0$ and by the definition above, the minimal such $q$ is the endpoint of the cap starting from $l$.

Moreover, the fact that $g^{(p)}_{\times}(q)=0$ and $g^{(p)}_{\times}(q-1)>0$ implies that $D_{\lambda}(i)=\circ$.

\end{remark}

\begin{lemma}
    A position $q\in \Z$ is the endpoint of a cap iff there exists $p\in \Z, p\leq q$ such that 
\begin{align*}
    \forall \, p< r \leq q, \;\;\;\;  &\sharp\{ r\leq s \leq q\,|\, D_{\lambda}(s) = \times \}< \sharp\{ r\leq s \leq q\,|\, D_{\lambda}(s) = \circ \} \\
    &\sharp\{ p\leq s \leq q\,|\, D_{\lambda}(s) = \times \}= \sharp\{ p\leq s \leq q\,|\, D_{\lambda}(s) = \circ \}. 
\end{align*}
\end{lemma}
\begin{proof}
    Define 
    $$g^{(q)}_\circ (r) := \sharp \{ r < s \le q|D_\lambda(s)=\times\} - \sharp \{ r < s \le q|D_\lambda(s)=\circ\}.$$

    Notice that for $p < r \le q$, $$\sharp\{ r\leq s \leq q\,|\, D_{\lambda}(s) = \times \}< \sharp\{ r\leq s \leq q\,|\, D_{\lambda}(s) = \circ \} \iff g^{(q)}_\circ (r) < 0$$
    Similarly, for $p \le r < q$, 
    $$\sharp\{ r\leq s \leq q\,|\, D_{\lambda}(s) = \times \}< \sharp\{ r\leq s \leq q\,|\, D_{\lambda}(s) = \circ \} \iff g^{(p)}_\times (r) > 0$$
    Furthermore, $g^{(p)}_\times (q) = g^{(q)}_\circ (p)$. So for the $\Rightarrow$ direction it is sufficient to show that 
    $$\forall p < r \le q, \;\;\;\; g^{(q)}_\circ (r) < 0.$$
    Likewise, for the $\Leftarrow$ direction it is sufficient to show that 
    $$\forall p \le r < q, \;\;\;\; g^{(p)}_\times (r) < 0.$$
    
   The value $g^{(p)}_\times (r) + g^{(q)}_\circ (r)$ is precisely the number $\times$ symbols in the segment $(p,q]$ plus the number of $\circ$ symbols in the same segment. Thus $g^{(p)}_\times (r) + g^{(q)}_\circ (r)$ doesn't depend on $\lambda$.

We now prove the implications in the statement of the lemma.

    $\Rightarrow:$ Let $q$ be the position of the endpoint of a cap starting at position $p$ (thus $D_\lambda (q)=\circ$). Then by definition $g^{(p)}_\times(r) > 0$ for $p < r < q$ and $g^{(p)}_\times(q)=0$.
    As such the number of $\times$ symbols in the segment $[p,r]$ is greater than the number of $\circ$ symbols. Since $g^{(p)}_\times(q)=0$ we deduce that $g^{(p)}_\times (r) + g^{(q)}_\circ (r) = 0$. As such, $g^{(q)}_\circ (r) < 0$. Additionally we have: $$D_\lambda(q)=\circ, \; g^{(q)}_\circ (r) = -1 < 0.$$

    $\Leftarrow:$ Let $q$ and $p\le q$ be such that for $p< r \le q$, we have: $g^{(q)}_\circ (r) < 0$. Recall that $g^{(q)}_\circ (p) = 0$, and for the same reason $g^{(p)}_\times (q)=0$. We deduce that $D_\lambda(p) = \times$ and $g^{(p)}_\times (p) > 0$. 

    As $g^{(p)}_\times (r) + g^{(q)}_\circ (r)$ is constant and equal to zero, and $g^{(q)}_\circ (r) < 0$ for $p < r < q$, we conclude that $g^{(p)}_\times (r) > 0$. 
    
\end{proof}

From the above proof, the function $g^{(q)}_{\circ}$ also satisfies properties similar to the ones of $g^{(p)}_{\times}$:

\begin{itemize}
    \item $g^{(q)}_{\circ}(q)=-1$,
    \item $g^{(q)}_{\circ}(r)<<0$ for $r<<q$ (due to the fact that there are finitely many $\times$ symbols and infinitely many $\circ$ symbols in $D_{\lambda}$), and 
    \item $|g^{(q)}_{\circ}(r+1) -g^{(q)}_{\circ}(r)| \leq 1$ for any $r\geq p$.
\end{itemize} 

This immediately implies:

\begin{cor}
\label{cap cor}
    Let $q\in \Z$ such that $D_{\lambda}(q) = \circ$. Assume that for some $p\in \Z, p<q$ we have:
\begin{align*}
\sharp\{ p\leq r\leq q \,|\, D_{\lambda}(r) = \times \}\geq \sharp\{ p\leq r\leq q \,|\, D_{\lambda}(r) = \circ \} 
\end{align*}
Then position $q$ is the end of some cap in $D_{\lambda}$.
\end{cor}
\begin{proof}
    For $p$ as in the statement of the Corollary, we have $g^{(q)}_{\circ}(p)>0$. Thus there exists $p\leq p'\leq q$ such that $g^{(q)}_{\circ}(p')=0$, and the largest such $p'$ would be the beginning of a cap which ends at $q$.
\end{proof}

\begin{remark}
    The cap diagram of a given weight diagram $D_\lambda$ may be described in a more intuitive manner. Denote by $\{c_1,\dots, c_k\}$ the set of right endpoints of the caps, where $c_i$ is the end of the cap starting at $\times_i$. The positions $c_k,\dots, c_1$ can be defined starting inductively by  
    $$c_l = \min \{ p \in \mathbb{Z}| p > \times_l,\, D_{\lambda}(p) = \circ \text{ and } p \neq c_{l'} \text{ for } l' > l\} $$
    This description tells us how to draw a cap diagram effectively. 
    This is done by going over all $\times$ symbols of the diagram from rightmost to leftmost and drawing a cap from the $\times$ to the first empty position to the right of this $\times$ symbol that is not the endpoint of a previously drawn cap.
\end{remark}

\section{Description of bases in \texorpdfstring{$\mathbb{S}$}{S}}
\label{b sigma section}
\subsection{Right odd roots}
\mbox{}

Let $m,n \geq 0$.

\begin{definition}
\mbox{}
\begin{itemize}
    \item  The odd positive roots for the choice of distinguished base $\distinguished$ will be called right odd roots. The set of right odd roots is denoted by $\mathcal{R}:=\odd{{\Delta^+_{dist}}}$. Explicitly, $\alpha\in \mathcal{R} \implies \alpha=\ror{i}{j} \text{ for some } 1\le i \le m \text{ and } 1\le j \le n$.
    \item The set $\mathcal{R}$ inherits the partial order on the weights induced by the base $\distinguished$: $\alpha \leq \beta $ iff $\beta - \alpha $ is a linear combination of elements of $\distinguished$ with non-negative integer coefficients. Explicitly, 
$$\varepsilon_i-\delta_j\le \varepsilon_k-\delta_l\iff i\ge k \text{ and } j\le l.$$
\item  Given a set $X \subset \mathcal{R}$ and $\alpha\in \mathcal{R}$, we write $\alpha \geq X$ when $\alpha \geq \beta$ for all $\beta \in X$.
\end{itemize}
  
\end{definition}

\begin{remark}
This order yields a lattice structure on $\mathcal{R}$. For example, for $\mathfrak{gl}(4|4)$ the lattice looks as follows:

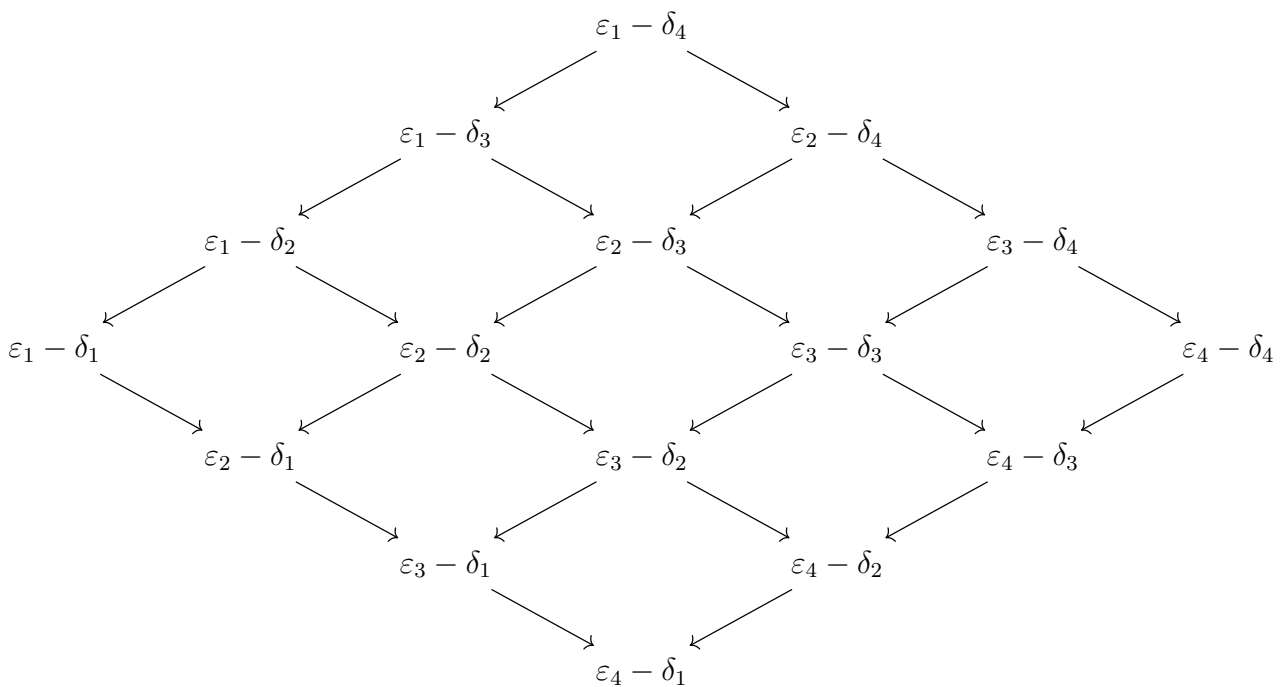
\begin{equation*}
\begin{tikzcd}
   &     &    
   & \varepsilon_1-\delta_4 \arrow[ld] \arrow[rd] &            &   & \\
 &  & \varepsilon_1-\delta_3 \arrow[ld] \arrow[rd] &                                              & \varepsilon_2-\delta_4 \arrow[ld] \arrow[rd] &                                              &  \\
 & \varepsilon_1-\delta_2 \arrow[ld] \arrow[rd] &               & \varepsilon_2-\delta_3 \arrow[ld] \arrow[rd] &           & \varepsilon_3-\delta_4 \arrow[ld] \arrow[rd] &  \\
\varepsilon_1-\delta_1 \arrow[rd] &                                              & \varepsilon_2-\delta_2 \arrow[ld] \arrow[rd] &    & \varepsilon_3-\delta_3 \arrow[ld] \arrow[rd] &                                              & \varepsilon_4-\delta_4 \arrow[ld] \\
                                  & \varepsilon_2-\delta_1 \arrow[rd]            &                                              & \varepsilon_3-\delta_2 \arrow[ld] \arrow[rd] &                                              & \varepsilon_4-\delta_3 \arrow[ld]            &                                   \\
                                  &                                              & \varepsilon_3-\delta_1 \arrow[rd]            &                                              & \varepsilon_4-\delta_2 \arrow[ld]            &                                              &                                   \\
                                  &                                              &                                              & \varepsilon_4-\delta_1                       &                                              &                                              &                                  
\end{tikzcd}
\end{equation*}
\end{remark}

\begin{definition}
    A subset $X \subset \mathcal{R}$ is called {\bf incomparable} if no two elements of $X$ can be compared in the order above: that is, $\forall \alpha, \beta \in X$, $\alpha \not\ge \beta$, $\alpha \not\le \beta$.
\end{definition}

\subsection{Describing bases in \texorpdfstring{$\mathbb{S}$}{S} using right odd roots}
\mbox{}

In what follows we will often use words over the alphabet $\{\varepsilon,\delta\}$ with $m$ copies of $\varepsilon$ and $n$ copies of $\delta$ to represent bases in $\mathbb{S}$, which is possible due to the well-known lemma below (cf. \cite[Example 1.20]{cheng2012dualities}). 

\begin{notation}
\mbox{}
\begin{itemize}
    \item For any word $w$ as above, we will write $\eps_i$ for the $i$-th letter $\eps$ in $w$, read from left to right (respectively, $\delta_j$ for the $j$-th letter $\delta$ in $w$).
    \item 
Given a word $w$, let $f_w(k)$ denote the number of letters $\delta$ in $w$ standing between the $\eps_k$ and $\eps_{k+1}$. We also set $f_w(0)$ to be the number of letters $\delta$ to the left of $\eps_1$. Similarly, let $g_w(k)$ denote the number of letters $\eps$ in $w$ standing between the $\delta_k$ and $\delta_{k+1}$. We also set $g_w(0)$ to be the number of letters $\eps$ to the left of $\delta_1$. When it is clear which word $w$ we are referring to, we will omit the subscript $w$ from this notation.
\end{itemize}
\end{notation}
\begin{lemma}
\label{base to word bij}
There exists a bijection between $\mathbb{S}$ and the set of words over the alphabet $\{\varepsilon,\delta\}$ with $m$ copies of $\varepsilon$ and $n$ copies of $\delta$. 
\end{lemma}

\begin{proof}
Fixing a word $w$, we number the letters $\varepsilon$ in $w$ from left to right $\varepsilon_1, \ldots, \varepsilon_m$ as before, and similarly for the letters $\delta$ in $w$. Taking the differences between each two adjacent symbols gives a base $\Sigma_w$, and it is obviously in $\mathbb{S}$ (due to the numbering of the symbols).

Vice versa, given a base $\Sigma \in \mathbb{S}$, we define the corresponding word $w$ by describing the function $f:\{0, \ldots,m-1\} \to \{0, \ldots, n\}$ corresponding to $w$: we set $$f(0):=\sharp\{1\leq i\leq n: \varepsilon_1 - \delta_i \in \Delta^-_{\Sigma}\}$$ and for $k>0$, we set
$$f(k):=\sharp\{1\leq i\leq n: \varepsilon_k - \delta_i \in \Delta^+_{\Sigma}\} - \sharp\{1\leq i\leq m: \varepsilon_{k+1} - \delta_i \in \Delta^+_{\Sigma}\}.$$

These maps are clearly inverse to one another.

\end{proof}

\begin{example}
 The base $\distinguished$ is given by the word $w_{dist}:=\varepsilon^m\delta^n= \delta^{f(0)}\overset{m}{\underset{k=1}{\prod}}\varepsilon\delta^{f(k)}$ where $f(k)=0$ for $k\in\{0,\dots,m-1\}$ and $f(m)=n$. 
 \end{example}
\begin{lemma}
\label{ror are incomparable}
    There exists a map between $\mathbb{S}$ and the set of incomparable subsets of $\mathcal{R}$ given by $\Sigma \longmapsto \Sigma \cap \mathcal{R}$, and this map is a bijection.
\end{lemma}
\begin{proof}
    Let $\Sigma\in \mathbb{S}$. Assume $\alpha, \beta\in \Sigma$ satisfy: $\alpha - \beta$ is a linear combination with positive coefficients of the elements of $\Delta_{\overline{0}, \distinguished}^+$. Then $\alpha - \beta$ is a combination with positive coefficients of the elements of $\Delta_{\overline{0}, \Sigma}^+$, since $\Sigma\in \mathbb{S}$. But that contradicts the assumption that $\Sigma$ is a base and its elements are linearly independent. Thus $\Sigma \cap \mathcal{R}$ is incomparable and the map is well-defined.

    We now define the inverse map. 
    Let $S$ be a set of incomparable roots. For any pair of roots $\ror{i}{j}$ and $\ror{i'}{j'}$ in $S$, we have: $i'\neq i$, and $i < i'\Rightarrow j < j'$.
    So we may write $S=\{\ror{i_1}{j_1},\dots ,\ror{i_k}{j_k}\}$ where $i_1 < i_2 < \dots < i_k$, and $j_1< j_2 < \dots < j_k$. We define the word $w_S$ by defining the function $d:\{0, \ldots, n+1\} \to \{0, \ldots, m\}$:
    $$
    d(j) := 
    \begin{cases}
        0 &\text{ if } \; 0\le j < j_1 \\
        i_l &\text{ if } \; j_{l-1} \le j < j_l \text{ for } l\in \{2,\dots, k-1\} \\
        i_k &\text{ if } \; j_k \le  j \le n\\
        m & \text{ if } \; j=n+1
    \end{cases}
    $$
    It is obvious that $d$ is an increasing function and we take the word $w_S= \varepsilon^{g(0)}\overset{n}{\underset{k=1}{\prod}}(\delta\varepsilon^{g(k)})$ 
    where $g(k):=d(k+1)-d(k)$, $k=0, \ldots, n$. Let $\Sigma_S \in \mathbb{S}$ be the base corresponding to the word $w_S$.
    
    All that remains is to verify that these maps are indeed inverse to each other. One direction is clear: if we set $S:= \Sigma \cap \mathcal{R}$ for some $\Sigma \in \mathbb{S}$, then the base we recover from $S$ by the construction above is clearly $\Sigma$ (it has the same word as $\Sigma$). 
    
    So we only need to check that for any incomparable set $S \subset \mathcal{R}$, we have: $\Sigma_S\cap \mathcal{R} =S$. 
    Consider the word $w_S$ constructed above.
    Observe that $\ror{i}{j}\in \Sigma_S\cap \mathcal{R}$ if and only if $j = \min\{j': i=d(j')\}$.
  Thus $\ror{i}{j}\in \Sigma_S\cap \mathcal{R}$ iff $(i, j) = (i_l, j_l)$ for some $l\in \{1, \ldots, k\}$, which implies $S \subset \Sigma_S\cap \mathcal{R}$.

\end{proof}

\begin{example}
   Consider the root system of $\fgl{4}{4}$ and $X=\{\ror{2}{1},\ror{3}{3}\}$, we know that $f(1)=f(4)=0$, $f(1)+f(2)=3 - 1 = 2$ and $f(1)+f(2)+f(3)+f(4)=4$ thus the word corresponding to the inverse of $X$ is $\varepsilon^2 \delta^2 \varepsilon \delta^2 \varepsilon$ and thus the base is 
   $$\Sigma=\{\varepsilon_1-\varepsilon_2, \ror{2}{1}, \delta_1-\delta_2, \delta_2-\varepsilon_3, \ror{3}{3}, \delta_3-\delta_4, \delta_4-\varepsilon_4\}.$$
   \end{example}
   
Recall that for any $\ror{i}{j}\in \Sigma$, then $r_{i,j}(\Sigma)\in \mathbb{S}$. Here is how the words describing $\Sigma, r_{i,j}(\Sigma)$ are related: take the word $w$ corresponding to $\Sigma$. The word $w$ corresponding to $ r_{i,j}(\Sigma)$ is obtained by swapping the $\eps_i$ (from the left) with the symbol $\delta_j$ standing to its right. 

In other words, if the word corresponding to $\Sigma$ is $w=\delta^{f(0)}\overset{m}{\underset{k=1}{\prod}}\varepsilon\delta^{f(k)}$ then the word corresponding to $r_{i,j}(\Sigma)$ will be $w'=\delta^{f'(0)}\overset{m}{\underset{k=1}{\prod}}\varepsilon\delta^{f'(k)}$ where $f'(i-1)= f(i-1)+1$, $f'(i)=f(i)-1$ and $f'$ agrees with $f$ everywhere else.

\begin{lemma}
    \label{reflection on set of ror}
    Let $\Sigma\in\mathbb{S}$ and $\ror{i}{j}\in\Sigma$. Then for any $\alpha\in (r_{i,j}(\Sigma)\cap \mathcal{R}) \setminus (\Sigma\cap \mathcal{R})$ has $\alpha > \ror{i}{j}$.
\end{lemma}
\begin{proof}
    Let $\Sigma\in \mathbb{S}$ and $\ror{i}{j}\in \Sigma$. 
    Consider the word $w$ corresponding to $\Sigma$, since $\ror{i}{j}\in \Sigma$ $\eps_i$ must be immediately to the left of $\delta_j$. 
    
    The letter immediately to the left of $\eps_i$ is either $\delta$ or $\eps$. Since $\delta$s and $\eps$s are numbered by order of appearance from left to right, the letter to left of $\eps_i$ is either $\delta_{j-1}$ or $\eps_{i-1}$.
    
    Similarly, the letter immediately to the right of $\delta_j$ is either $\delta_{j+1}$ or $\eps_{i+1}$.
    
    As the word corresponding to $r_{i,j}(\Sigma)$ is the same as $w$ excepts $\eps_i$ and $\delta_j$ swap places, we must simply consider the letters immediately to the left and right of $\eps_i \delta_j$. As such we consider 4 possible cases: $\eps_{i-1}\eps_i\delta_j\eps_{i+1}$, $\eps_{i-1}\eps_i\delta_j\delta_{j+1}$,
    $\delta_{j-1}\eps_i\delta_j\eps_{i+1}$ and $\delta_{j-1}\eps_i\delta_j\delta_{j+1}$.

    In the first case, $\eps_{i-1}\eps_i\delta_j\eps_{i+1} \mapsto \eps_{i-1}\delta_j\eps_i\eps_{i+1}$ thus $$r_{i,j}(\Sigma) \cap \mathcal{R} = (\Sigma \cap \mathcal{R}) \cup \{\ror{i-1}{j}\} \setminus \{\ror{i}{j}\}.$$ Notice that $\ror{i-1}{j} > \ror{i}{j}$.

    In the second case, $\eps_{i-1}\eps_i\delta_j\delta_{j+1} \mapsto \eps_{i-1}\delta_j\eps_i\delta_{j+1}$ thus 
    
$$r_{i,j}(\Sigma) \cap \mathcal{R} = (\Sigma \cap \mathcal{R}) \cup \{\ror{i-1}{j}, \ror{i}{j+1}\} \setminus \{\ror{i}{j}\}.$$ Notice that $\ror{i}{j+1}>\ror{i}{j}$.
    
    In the third case, $\delta_{j-1}\eps_i\delta_j\eps_{i+1} \mapsto \delta_{j-1}\delta_j\eps_i\eps_{i+1}$ thus $r_{i,j}(\Sigma) \cap \mathcal{R} = (\Sigma \cap \mathcal{R}) \setminus \{\ror{i}{j}\}$. 

    And finally in the fourth case $\delta_{j-1}\eps_i\delta_j\delta_{j+1} \mapsto \delta_{j-1}\delta_j\eps_i\delta_{j+1}$ thus 
$$r_{i,j}(\Sigma) \cap \mathcal{R} = (\Sigma \cap \mathcal{R}) \cup \{\ror{i}{j+1}\} \setminus \{\ror{i}{j}\}.$$ 

    In all cases $\alpha \in r_{i,j}(\Sigma)\cap \mathcal{R} \setminus \Sigma\cap \mathcal{R}$ has $\alpha > \ror{i}{j}$.
\end{proof}

\begin{definition}
    Given a sequence of right odd roots $(\alpha_1,\dots,\alpha_k)$, we say that the expression $r_{\alpha_k}\circ\dots\circ r_{\alpha_1}(\distinguished)$ is valid means that for any $i=0, \ldots, k-1$ we have:
$$\alpha_{i+1} \in r_{\alpha_i}\circ\dots\circ r_{\alpha_1}(\distinguished).$$
\end{definition}

\begin{lemma}
\label{sequence from distinguished}
    Let $\Sigma \in \mathbb{S}$. Then there exists a sequence of right odd roots $(\alpha_1,\dots,\alpha_k)$ such that the $r_{\alpha_k}\circ\dots\circ r_{\alpha_1}(\distinguished)$ is valid and equal to $\Sigma$.
\end{lemma}
Note that $\alpha_i\neq \alpha_j$ for all indices $i\neq j$.

\begin{proof}
    Consider $w$ the word corresponding to some $\Sigma\in \mathbb{S}$ and $f_w$ the function correspond to $w$, which is to say, $w = \delta^{f_w (0)}\overset{m}{\underset{k=1}{\prod}}\varepsilon\delta^{f_w(k)}$.
    Set $e(i):= \sum_{l=i}^{m}f_w(l)$ and $\mathbf{\Delta}_e(\Sigma) := \sum_{i=1}^m (n-e(i))$.
    
    We prove the Lemma by induction of $\mathbf{\Delta}_e(\Sigma)$.

    For the base case the only base that has $\mathbf{\Delta}_e(\Sigma)=0$ is $\distinguished$ and there is the set of roots is $\varnothing$.

    Assume the statement is correct for $\mathbf{\Delta}_e(\Sigma) = K$ and let $\Sigma\in \mathbb{S}$ such that $\mathbf{\Delta}_e(\Sigma) = K$.
    Let $w$ be the word corresponding to $\Sigma$. Let $1\le i \le m$ be the minimal index such that $f_w(i) > 0$. By minimality there must exists at least one $\delta$ between $\varepsilon_{m-i}$ and $\varepsilon_{m-(i+1)}$. denote the $\delta$ immediately to the left of $\varepsilon_{m-(i+1)}$ by $\delta_j$. Consider $w'$ obtained from $w$ by swapping $\delta_j$ with $\varepsilon_{m-(i+1)}$. $w'$ corresponds to some base $\Sigma'$ which clearly has $\mathbf{\Delta}_e(\Sigma')=\mathbf{\Delta}_e(\Sigma) - 1$ and we find that: 
    $$\ror{m-(i+1)}{j}\in \Sigma',\, r_{(m-(i+1)),j}(\Sigma')=\Sigma$$ 
    This completes the proof of the induction step.
\end{proof}

Note that there may be many other sequences of right odd roots $(\beta_1,\dots,\beta_k)$ such that $r_k \circ \dots \circ r_1(\distinguished)$ is valid and equal to $\Sigma$. We will now show that all such sequences are in fact simply permutations of one another. That is to say, the {\it set} $\{\alpha_1,\dots, \alpha_k\}$ is uniquely defined.

\begin{theorem}
\label{B Sigma}
    Consider the map 
$ B: \mathbb{S} \longrightarrow P(\mathcal{R})$ (subsets of $\mathcal{R}$)
given by
$$ B_{\Sigma} := \mathcal{R} \setminus \{\alpha \in \mathcal{R} | \; \exists \beta \in \Sigma \text{ so that } \alpha \geq \beta \}$$
(in other words, $B_{\Sigma}$ is a subset of $\mathcal{R}$ such that the minimal elements in the complement $\mathcal{R} \setminus B_{\Sigma}$ are precisely the elements of $\Sigma \cap \mathcal{R}$).

Then the map $B$ satisfies:
\begin{enumerate}
    \item There exists an ordering $(\alpha_1,\ldots,\alpha_k)$ of the elements of $ B_{\Sigma}$ such that the expression $r_{\alpha_k}\circ\ldots\circ r_{\alpha_1}(\Sigma_{dist} )$ is valid and equal to $\Sigma$;
    \item Given any sequence $(\alpha_1,\ldots,\alpha_k)$ of right odd roots such that $r_{\alpha_k}\circ\dots\circ r_{\alpha_1}(\Sigma_{dist})$ is valid and equal to $\Sigma$, we have: $\{\alpha_1,\ldots,\alpha_k\} = B_{\Sigma}$.
\end{enumerate}
\end{theorem}
\begin{proof}
    We may use Lemma \ref{sequence from distinguished} to stipulate that such a sequence $(\alpha_1,\ldots,\alpha_k)$ exists, and by Part (2) it will be an ordering of the elements of $B_{\Sigma}$, proving Part (1).    
    Moving on to Part (2), we need only verify that given a sequence of right odd roots $(\alpha_1,\dots,\alpha_k)$ such that $r_{\alpha_k}\circ \cdots \circ r_{\alpha_1}(\distinguished) = \Sigma$, that $k = \sharp B_\Sigma$ and that $\alpha_i\in B_\Sigma$ for $1\le i \le k$.

    Consider a base $\Sigma\in\mathbb{S}$ corresponding to a word $w= \eps^{g(0)}\overset{n}{\underset{k=1}{\prod}}\delta\varepsilon^{g(k)}$. As stated before, a reflection by $\alpha \in \Sigma \cap \mathcal{R}$ corresponds to a swap of a pair of letters $\varepsilon\delta \mapsto \delta\varepsilon$ in $w$. Therefore, the number of reflections with respect to right odd roots required to produce $\Sigma$ from $\distinguished$ is fixed and does not depend on the sequence of reflections: this number is 
    $ k=\sum_{j=1}^n d(j)$, where $d(j) := m - \sum_{l=0}^{j-1} g(l)$ is the number of letters $\eps$ to the right of the letter $\delta_j$ in $w$.

    We now construct a sequence $(\alpha_1,\dots, \alpha_k)$ explicitly,
    consider $\Sigma\in \mathbb{S}$ corresponding to a word $w= \eps^{g(0)}\overset{n}{\underset{k=1}{\prod}}\delta\varepsilon^{g(k)}$. 
    Recall that $w_{dist}$, the word corresponding to $\distinguished$ is given by $w_{dist}= \eps^m \delta^n =\eps^{g'(0)}\overset{n}{\underset{k=1}{\prod}}\delta\varepsilon^{g'(k)}$ where $g'(0)= m$ and $g'(k)\equiv 0$ for $1\le k \le n$.
    
    Consider that performing reflections $\ror{i}{1}$ for $i = m, m-1, \dots, g(0)$ on $\distinguished$ in that order is valid and yields a base i.e. $r_{m, 1}\circ \cdots \circ r_{g(0),1}=\Sigma_1$ corresponding to word $w_1=\eps^{g_1(0)}\overset{n}{\underset{k=1}{\prod}}\delta\varepsilon^{g_1(k)}$ where $g_1(0)= g(0)$ and $g_1(1)= m - g(0) = \sum_{k=1}^n g(k)$. Note here that $d(1)$ reflections have been performed.

    Repeating this process by performing reflections by $\ror{i}{2}$ for $i = m, \dots, g(1) + g(0)$ on $\Sigma_1$ is once again valid and yields a base  i.e. $r_{m,2}\circ \cdots \circ r_{g(0)+g(1), 2}(\Sigma_1)=\Sigma_2$ which is given by word $w_2=\eps^{g_2(0)}\overset{n}{\underset{k=1}{\prod}}\delta\varepsilon^{g_2(k)}$ where $g_2(0)= g_1(0) = g(0)$ and $g_2(1)= g_1(1) - (m - g(0) - g(1)) = g(1)$ and $g_2(2) = m - g(0) - g(1)$. Again note that $d(2)$ additional reflections have been performed to get from $\distinguished$ to $\Sigma_2$. 

    Further repeating these steps for every $1\le j \le n$ yields a base $\Sigma_n$ which has $w_n = w$ one may verify that the number of reflections performed is $k= \sum_{j=1}^n d(j)$.
    By Lemmas \ref{ror are incomparable}, \ref{reflection on set of ror} we obtain: for any $i=1, \ldots, k$, 
    $$\forall \beta \in r_{\alpha_i} \circ \ldots \circ r_{\alpha_1}(\distinguished) \cap \mathcal{R}, \;\;\;\;\text{ either } \; \beta> \alpha_i \;\;\text{ or }\;\; \beta, \alpha_i \; \text{ are incomparable}.$$
    This means that for any $\beta \in \Sigma$, and any $1\leq i \leq k$, we have $\beta \not<\alpha_i$. Thus $\alpha_1, \ldots, \alpha_k \in B_{\Sigma}$. This completes the proof of the theorem.
\end{proof}
\begin{cor}
    Let $\Sigma_1\in \mathbb{S}$ and let $\alpha\in\Sigma_1$. Let $\Sigma_2:=r_\alpha (\Sigma_1)$. Then $B_{\Sigma_2}=B_{\Sigma_1}\cup \{\alpha\}$.
\end{cor}
\begin{proof}
    By Part (1) of Theorem \ref{B Sigma}, we deduce that there is some ordering of the elements of $\{\beta_1,\dots,\beta_k\} := B_{\Sigma_1}$ such that $r_{\beta_k}\circ\cdots\circ r_{\beta_1}(\distinguished)$ is valid and equal to $\Sigma_1$. 
    Therefore $r_\alpha \circ r_{\beta_k}\circ\cdots\circ r_{\beta_1}(\distinguished)$ is valid and equal to $\Sigma_2$. 
    
    By Part (2) we deduce that $B_{\Sigma_2}=B_{\Sigma_1}\cup \{\alpha\}$.
\end{proof}

\subsection{Visual interpretation of the set \texorpdfstring{$B_\Sigma$}{B} and the set \texorpdfstring{$\Sigma \cap \mathcal{R}$}{Sigma in R}}
\mbox{}

The set $B_\Sigma$ gives us a way to describe a base in $\mathbb{S}$ in terms of a set of right odd roots. This in itself is quite convenient, but we may further develop this set by giving it a visual interpretation. 

Consider the set $\mathcal{R}$ as a lattice (as described earlier in this section) with labels removed. For each $\Sigma\in \mathbb{S}$, we may add to this lattice the relevant data defining $\Sigma$: either the set $\Sigma\cap \mathcal{R}$ (which is a set of incomparable right odd roots) or the set $B_{\Sigma}$.
\begin{example}
\label{first rectangle ex}
Consider $\fgl{4}{4}$ and consider the base $\Sigma$ given by the word $\varepsilon\delta\varepsilon\delta^2\varepsilon^2\delta$.
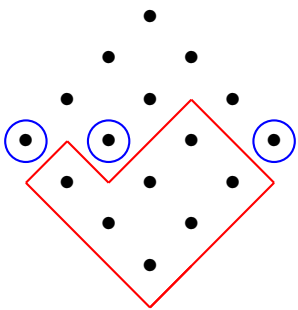
    \begin{equation*}
\begin{tikzpicture}[anchorbase,scale=1.1]
% row 1
\node at (0,1.5) {$\bullet $};

% row 2
\node at (0.5,1) {$\bullet $};
\node at (-0.5, 1) {$\bullet $};

% row 3
\node at (1, 0.5) {$\bullet $};
\node at (0, 0.5) {$\bullet $};
\node at (-1, 0.5) {$\bullet $};

% row 4
\node at (1.5, 0) {$\bullet $};
\node at (0.5, 0) {$\bullet $};
\node at (-0.5, 0) {$\bullet $};
\node at (-1.5, 0) {$\bullet $};
\draw[thick, blue] (-1.5,0) circle [radius=0.25cm];
\draw[thick, blue] (-0.5,0) circle [radius=0.25cm];
\draw[thick, blue] (1.5,0) circle [radius=0.25cm];
% row 5
\node at (1, -0.5) {$\bullet $};
\node at (0, -0.5) {$\bullet $};
\node at (-1, -0.5) {$\bullet $};

% row 6
\node at (0.5, -1) {$\bullet $};
\node at (-0.5, -1) {$\bullet $};

% row 7
\node at (0, -1.5) {$\bullet $};

\draw[thick, red, -] (0.5 , 0.5)--(-0.5 ,-0.5);
\draw[thick, red, -] (0.5 , 0.5)--(1.5 ,-0.5);
\draw[thick, red, -] (1.5 ,-0.5)--(0 ,-2);

\draw[thick, red, -] (-1 , 0)--(-1.5 ,-0.5);
\draw[thick, red, -] (-1 , 0)--(-0.5 , -0.5);
\draw[thick, red, -] (-1.5 , -0.5)--(0 ,-2);
\draw[thick, red, -] (0.5 , -1.5)--(0 ,-2);

\end{tikzpicture}
\end{equation*} 
In the above diagram the bullets circled in blue indicate the right odd roots $\Sigma\cap \mathcal{R}$ which in this case is the set $\{\ror{1}{1},\ror{2}{2},\ror{4}{4}\}$. 

Similarly, the area enclosed in a red polygon is used to indicate the set $B_\Sigma$. Depending on the application one may be inclined to display only $B_\Sigma$ or only $\Sigma\cap \mathcal{R}$.
\end{example}

\begin{remark}
    Any pair of right odd roots are incomparable if and only if one element is not less than the other, then it is important to understand what roots that are comparable look like on the diagram. This can be done by considering a single black dot (which of course corresponds to some right odd root), any other black dot which can be reached by downwards diagonal moves is smaller. 
Beneath is an example of all some fixed root $\alpha$ circled in blue) and all roots $\beta \lneq \alpha$ (enclosed by the blue polygon).

    
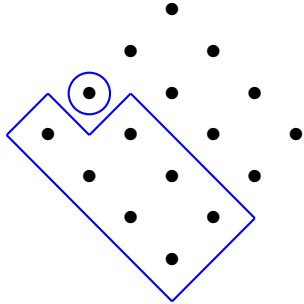
\begin{equation*}
\begin{tikzpicture}[anchorbase,scale=1.1]
% row 1
\node at (0,1.5) {$\bullet $};

% row 2
\node at (0.5,1) {$\bullet $};
\node at (-0.5, 1) {$\bullet $};

% row 3
\node at (1, 0.5) {$\bullet $};
\node at (0, 0.5) {$\bullet $};
\node at (-1, 0.5) {$\bullet $};
\draw[thick, blue] (-1,0.5) circle [radius=0.25cm];

% row 4
\node at (1.5, 0) {$\bullet $};
\node at (0.5, 0) {$\bullet $};
\node at (-0.5, 0) {$\bullet $};
\node at (-1.5, 0) {$\bullet $};

% row 5
\node at (1, -0.5) {$\bullet $};
\node at (0, -0.5) {$\bullet $};
\node at (-1, -0.5) {$\bullet $};

% row 6
\node at (0.5, -1) {$\bullet $};
\node at (-0.5, -1) {$\bullet $};

% row 7
\node at (0, -1.5) {$\bullet $};

\draw[thick, blue, -] (-1.5, 0.5)--(-1, 0);
\draw[thick, blue, -] (-1, 0)--(-0.5, 0.5);
\draw[thick, blue, -] (-0.5, 0.5)--(1, -1);
\draw[thick, blue, -] (1, -1)--(0, -2);
\draw[thick, blue, -] (0, -2)--(-2, 0);
\draw[thick, blue, -] (-2, 0)--(-1.5, 0.5);

\end{tikzpicture}
\end{equation*} 
\end{remark}

% beginning of arrow section
\section{Tracking changes in the weight diagrams}
\label{arrow diagram section}
\mbox{}
\subsection{Applying a single right odd reflection}
\begin{definition}
\mbox{}

\begin{itemize}
    \item Given $\Sigma \in \mathbb{S}$ and $\alpha\in \Sigma$, we say that $\alpha\in \mathcal{R}$ \textbf{changes $D_\lambda^{\Sigma}$} if $D_{\lambda}^{\Sigma} \neq D_{\lambda}^{r_{\alpha}\Sigma}$.
    \item We say that $\alpha\in \mathcal{R}$ \textbf{changes the weight diagram of $\lambda$} if there exists $\Sigma\in \mathbb{S}$ such that $\alpha\in\Sigma$ and $\overline{\lambda}_\Sigma\neq \overline{\lambda}_{r_\alpha\Sigma}$.

\end{itemize}

\end{definition}

In other words, $\alpha\in \mathcal{R}$ changes the weight diagram of $\lambda$ iff there exists $\Sigma\in \mathbb{S}$ such that $D_{\lambda}^{\Sigma} \neq D_{\lambda}^{r_{\alpha}\Sigma}$.

We will use the following straightforward observations, the first of them being just \cref{rmk:odd_reflection_on_diagrams}:
\begin{lemma}\label{lem:generalities_on_ror_changes}
Let $\Sigma\in \mathbb{S}$, and let $1\leq i\leq m$, $1\leq j\leq n$. 
    \begin{enumerate}
  
    \item\label{itm:gen_ror_change2} Assume that the application $r_{i, j}(\Sigma)$ is valid, i.e. $\ror{i}{j}\in \Sigma$. Denote: $p:=\left(\overline{\lambda}_{\Sigma}\right)_i$.

Then the odd reflection $r_{i, j}$ changes the diagram $D_{\lambda}^{\Sigma}$ iff $$p=\left(\overline{\lambda}_{\Sigma}\right)_i = -\left(\overline{\lambda}_{\Sigma}\right)_{-j}.$$

In that case, the diagram $D_{\lambda}^{r_{i,j}.\Sigma}$ is given by moving a $\times$ symbol in the diagram $D_{\lambda}^{\Sigma}$ from position $p$ to the position $p+1$.
\item\label{itm:gen_ror_change3} Assume the composition $r_{i, j+1}r_{i, j}(\Sigma)$ is valid and $r_{i, j}$ changes the diagram $D_{\lambda}^{\Sigma}$. 

Then the odd reflection $r_{i, j+1}$ changes the diagram $D_{\lambda}^{r_{i, j}(\Sigma)}$ if and only if $$-\left(\overline{\lambda}_{r_{i, j}.\Sigma}\right)_{-(j+1)} = p +1.$$

\end{enumerate}
\end{lemma}
\begin{proof}
The first two statements follow from the fact that an odd reflection $r_{\alpha}$ changes the diagram $D_{\lambda}^{\Sigma}$ for any $\Sigma \in \mathbb{S}$, $\alpha\in \Sigma$ iff $(\overline{\lambda}_{\Sigma}|\alpha)=0$, and in that case, $\overline{\lambda}_{r_{\alpha}.\Sigma} = \overline{\lambda}_{\Sigma}+\alpha$.

The last statement follows from the fact that 
$\overline{\lambda}_{r_{i, j}.\Sigma} = \overline{\lambda}_{\Sigma} + \eps_i - \delta_j$ so $\left(\overline{\lambda}_{r_{i, j}.\Sigma}\right)_i = p +1$.

\end{proof}

\begin{remark}\label{rmk:diagram_change}
In particular, in the setting of \eqref{itm:gen_ror_change3},  $r_{i, j+1}$ changes the diagram $D_{\lambda}^{r_{i, j}(\Sigma)}$ iff $D_{\lambda}^{\Sigma}(p+1)$ has either a $<$ or a $\times$ symbol. 
\end{remark}

Recall from Lemma \ref{sequence from distinguished} that any $\Sigma\in \mathbb{S}$ can be obtained from the base $\distinguished$ by applying several right odd reflections, and $\overline{\lambda}_{\Sigma}$ is obtained from $\overline{\lambda}$ by adding several right odd roots $\eps_i-\delta_j$ for some $i, j$. Together with \cref{lem:generalities_on_ror_changes}, this implies:

\begin{cor}
\label{diagram by base inequality}
The diagram $D_{\lambda}^{\Sigma}$ is obtained from $D_{\lambda}$ by moving several (perhaps zero) $\times$ symbols several positions to the right.

  Accordingly, $(\overline{\lambda}_{\Sigma})_i\geq \overline{\lambda}_i$ and $-(\overline{\lambda}_{\Sigma})_{-j}\geq -(\overline{\lambda}_{-j})$ for any $i, j$.
\end{cor}

\begin{prop}
\label{changes the weight diagram of lambda}
Let $\lambda\in \distHW$ and let $\alpha$ be a right odd root. For any pair $\Sigma', \Sigma''\in \mathbb{S}$ such that $\alpha\in\Sigma'\cap\Sigma''$, we have: $$\overline{\lambda}_{\Sigma'}\neq \overline{\lambda}_{r_{\alpha}\Sigma'}\implies 
\overline{\lambda}_{\Sigma''}\neq \overline{\lambda}_{r_{\alpha}\Sigma''}.$$
In other words, if the odd reflection $r_{\alpha}$ changes the weight diagram $D^{\Sigma}_\lambda$ for some $\Sigma \in \mathbb{S}$, then it does so for every $\Sigma\in \mathbb{S}$ with $\alpha \in \Sigma$.
\end{prop}
\begin{proof}

Let $\lambda\in \Lambda^+_{m|n}$. 

Recall that by \cref{lem:generalities_on_ror_changes}, for any $\Sigma \in \mathbb{S}$ and any $\beta\in \Sigma\cap\mathcal{R}$, we have: $$\overline{\lambda}_{r_{\beta}\Sigma} = \begin{cases}
\overline{\lambda}_{\Sigma} &\text{ if } (\overline{\lambda}_{\Sigma}|\beta) \neq 0\\ 
\overline{\lambda}_{\Sigma}+\beta &\text{ if } (\overline{\lambda}_{\Sigma}|\beta) = 0
\end{cases}$$
Thus $\overline{\lambda}_{\Sigma}\neq \overline{\lambda}_{r_{\alpha}\Sigma} $ iff $ (\overline{\lambda}_{\Sigma}|\beta) = 0$.
Now, let $\Sigma', \Sigma''\in \mathbb{S}$. We need to show that for any $ \alpha=\eps_i-\delta_j \in \Sigma'\cap \Sigma''\cap \mathcal{R}$, the scalars
$$(\overline{\lambda}_{\Sigma'}|\beta)=(\overline{\lambda}_{\Sigma'})_i + (\overline{\lambda}_{\Sigma'})_{-j}, \;\; (\overline{\lambda}_{\Sigma''}|\beta)=(\overline{\lambda}_{\Sigma''})_i +  (\overline{\lambda}_{\Sigma''})_{-j}$$
are either both zero or both non-zero.

Recall from Theorem \ref{B Sigma} that $\Sigma'$ and $\Sigma''$ can be obtained by performing a sequence of right odd reflections on $\distinguished$; we may write $\Sigma'=r_{\beta_{1}}\circ \dots \circ r_{\beta_{k}}(\distinguished)$ where $\{\beta_{1},\dots,\beta_{k}\}=B_{\Sigma'}$, and $\Sigma''=r_{\gamma_{1}}\circ \dots \circ r_{\gamma_{l}}(\distinguished)$ where $\{\gamma_{1},\dots,\gamma_{l}\}=B_{\Sigma''}$.

Thus the inequality $(\overline{\lambda}_{\Sigma'}|\alpha)\neq (\overline{\lambda}_{\Sigma''}|\alpha)$ may occur only if there was some $\beta \in B_{\Sigma'}\triangle B_{\Sigma''}$ (here $\triangle$ denotes the symmetric difference) such that $(\beta|\alpha)\neq 0$. Yet for $\beta = \eps_{i'}-\delta_{j'} \in \mathcal{R}$, to have $(\beta|\alpha)\neq 0$ we must have either $i=i'$ or $j=j'$ which means that either $\beta <\alpha$ or $\beta>\alpha$. 

By the definition of $B_\Sigma$ in Theorem \ref{B Sigma}, all the elements $\beta\in \mathcal{R}, \beta<\alpha$ must appear in both sets $B_{\Sigma'}$ and $B_{\Sigma''}$, while by the same theorem elements $\beta\in \mathcal{R}, \beta>\alpha$ appear neither in $B_{\Sigma'}$ nor in $B_{\Sigma''}$. This shows that $(\overline{\lambda}_{\Sigma'}|\alpha)\neq (\overline{\lambda}_{\Sigma''}|\alpha)$ and the statement is proved.

\end{proof}

From the proof of \cref{changes the weight diagram of lambda}, we also obtain:
\begin{cor}\label{cor:things_orthog_to_hwt_lam}
    Let $\lambda\in \distHW$ and let $\alpha\in \mathcal{R}$. Then for any $\Sigma', \Sigma''\in \mathbb{S}$ such that $\alpha \in \Sigma'\cap \Sigma''$, we have:
  $$(\overline{\lambda}_{\Sigma'}|\beta) =0 \;\; \Longleftrightarrow\;\; (\overline{\lambda}_{\Sigma''}|\beta)=0$$
\end{cor}

\subsection{The arrow diagram}
\begin{definition}
\label{Arrow diagram def}
Given $\lambda\in \distHW$ with $\overline{\lambda}:= \lambda + \rho_{\distinguished}$, we construct the \textbf{arrow diagram of $\lambda$} by considering the weight diagram $D_\lambda$ together with a set $\{k_1,\dots,k_m\}$. The number $k_i$ is defined recursively (from $k_m$ to $k_1$) by $$k_i:=\min\{c\in\mathbb{Z}\;|\;\forall 1\le j \le n, \forall  i < i' \le m,\;c\ge \overline{\lambda}_i,\, c\neq -\overline{\lambda}_{-j}\; \text{ and } \; c \neq k_{i'}\}.$$
\end{definition}
This set can be interpreted visually by iteratively drawing arrows above the diagram $D_\lambda$, each arrow beginning in a position with a symbol $>$ or $\times$, and ending in a symbol $\circ$ or $>$. At the $i$-th step of the process ($i=0, \ldots, m-1$), we draw an arrow from the position $\overline{\lambda}_{m-i}$ to the first position to the right of $\overline{\lambda}_{m-i}$ that does not contain a symbol $<$ nor a symbol $\times$, nor is the end of a previously drawn arrow. The endpoint of the arrow originating from $\overline{\lambda}_{m-i}$ is called $k_{m-i}$. We will refer to this arrow as the $i$-th arrow in the diagram $D_{\lambda}$. 

Thus the arrows are drawn from left to right and numbered from right to left.

Additionally, we will use the following notation:
\begin{definition}
\label{Mi def}
For $1\le i \le m$ we denote $$M_i:=\max \{j\in\{1,\dots, n\}|-\overline{\lambda}_{-j}<k_i\}.$$ 

The number $M_i$ is the index $j$ such that $-\overline{\lambda}_{-j}$ is the position closest to the left of $k_i$ such that $D_\lambda (\overline{\lambda}_{-j})\in \{\times, <\}$. 
Another way to interpret $M_i$ is as the number of the symbols $<$ and $\times$ that appear in $D_{\lambda}$ to the left of the position $k_i$.

\end{definition}

\begin{remark}
We remind the reader that the similar treatment of the symbols $\times$ and $<$ (respectively, $\times$ and $>$) is due to the fact that the symbol $\times$ is in fact a combination of the symbols $>$ and $<$.
\end{remark}

\begin{example}
\label{ex:first arrow diagram}
 Given $\lambda \in \Lambda^+_{4|4}$ such that $\overline{\lambda}=(6\, 4\, 3\, 0\,|0\, -1\, -4\, -5)$, we may draw the arrow diagram of $\lambda$ as $D_\lambda$ with arrows starting at $\overline{\lambda}_i$ and ending at $k_i$ for $1\le i \le 4$ as seen below-
 
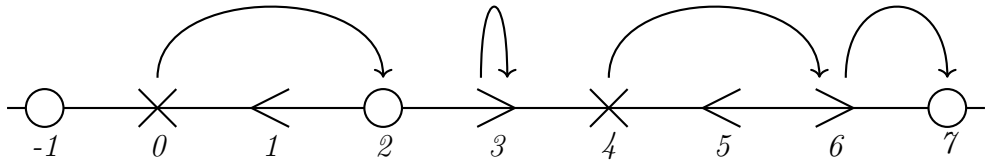
\begin{equation*}
    \begin{tikzpicture}
    
        \draw[thick, -] (-6.5 ,0)--(6.5, 0);
        
        \draw[thick, fill=white] (-6,0) circle [radius=0.25cm];
        \node[] at (-6,-0.5) {-1};
        
        % times
        \draw[thick, -] (-4.5,0)--(-4.5 + 0.25 ,0.25);
        \draw[thick, -] (-4.5,0)--(-4.5 - 0.25 ,0.25);
        \draw[thick, -] (-4.5,0)--(-4.5 + 0.25 ,-0.25);
        \draw[thick, -] (-4.5,0)--(-4.5 - 0.25 ,-0.25);
        \node[] at (-4.5,-0.5) {0};
        
        % le
        \draw[thick, -] (-3 -0.25, 0)--(-3 + 0.25, 0.25);
        \draw[thick, -] (-3 -0.25, 0)--(-3 + 0.25, -0.25);
        \node[] at (-3,-0.5) {1};
        
        \draw[thick, fill=white] (-1.5,0) circle [radius=0.25cm];
        \node[] at (-1.5,-0.5) {2};
        
        % ge
        \draw[thick, -] (0 +0.25, 0)--(0 - 0.25, 0.25);
        \draw[thick, -] (0 +0.25, 0)--(0 - 0.25, -0.25);
        \node[] at (0,-0.5) {3};
        
        % times
        \draw[thick, -] (1.5,0)--(1.5 + 0.25 ,0.25);
        \draw[thick, -] (1.5,0)--(1.5 - 0.25 ,0.25);
        \draw[thick, -] (1.5,0)--(1.5 + 0.25 ,-0.25);
        \draw[thick, -] (1.5,0)--(1.5 - 0.25 ,-0.25);
        \node[] at (1.5,-0.5) {4};
        
        % le
        \draw[thick, -] (3 -0.25, 0)--(3 + 0.25, 0.25);
        \draw[thick, -] (3 -0.25, 0)--(3 + 0.25, -0.25);
        \node[] at (3,-0.5) {5};
        
        % ge
        \draw[thick, -] (4.5 +0.25, 0)--(4.5 - 0.25, 0.25);
        \draw[thick, -] (4.5 +0.25, 0)--(4.5 - 0.25, -0.25);
        \node[] at (4.5,-0.5) {6};
        
        \draw[thick, fill=white] (6,0) circle [radius=0.25cm];
        \node[] at (6,-0.5) {7};
        
        % arrows
        \draw[thick, ->] (-4.5 ,0.4) .. controls (-4.5 ,1.65) and (-1.5 , 1.65) .. (-1.5 ,0.4);
        \draw[thick, ->] (-0.2 ,0.4) .. controls (-0.2 ,1.65) and (0.15 , 1.65) .. (0.15 ,0.4);
        \draw[thick, ->] (1.5 ,0.4) .. controls (1.5 ,1.65) and (4.5 -0.2 , 1.65) .. (4.5 -0.2 ,0.4);
        \draw[thick, ->] (4.5 +0.15 ,0.4) .. controls (4.5 +0.15 , 1.65) and (6 , 1.65) .. (6 ,0.4);
        
        % cross
        %\draw[thick, -] (-6,0)--(-6 + 0.25 ,0.25);
        %\draw[thick, -] (-6,0)--(-6 - 0.25 ,0.25);
        %\draw[thick, -] (-6,0)--(-6 + 0.25 ,-0.25);
        %\draw[thick, -] (-6,0)--(-6 - 0.25 ,-0.25);
        
        % less than 
        %\draw[thick, -] (-3 -0.25, 0)--(-3 + 0.25, 0.25);
        %\draw[thick, -] (-3 -0.25, 0)--(-3 + 0.25, -0.25);
        
        % greater than 
        %\draw[thick, -] (-3 +0.25, 0)--(-3 - 0.25, 0.25);
        %\draw[thick, -] (-3 +0.25, 0)--(-3 - 0.25, -0.25);
        
    \end{tikzpicture}
\end{equation*}

Here $k_4=2$, $k_3=3$, $k_2=6$ and $k_1=7$ and $M_4=2$, $M_3=2$, $M_2=4$ and $M_1=4$. 
\end{example} 

\begin{remark}
Note that there may exist $i\neq i'$ such that $M_i=M_{i'}$. On the other hand, $k_i=k_{i'}\implies i= i'$.
\end{remark}

In order to proceed with the description of the set of right odd roots that change the diagram of $\lambda$, we would like to consider a set of bases which will be helpful in the proof. 

For any $0 \le i \le m$, denote by $\Sigma^i$ the base in $\mathbb{S}$ described by the word $$\eps_1\eps_2\ldots \eps_{i}\delta_1\delta_2\ldots \delta_n\eps_{i+1}\ldots \eps_{m}.$$ 
That is, $\Sigma^m:=\distinguished$ and for every $i<m$,
$$\Sigma^i := r_{i+1,n}\circ \dots \circ r_{i+1,1}(\Sigma^{i+1})$$

Note that for $1 \le  i \le m$, we have: $\Sigma^i \cap \mathcal{R} = \{\ror{i}{1}\}$ and $\Sigma^0 \cap \mathcal{R} = \emptyset$. It is immediate that each of these sets is an incomparable set (as they all have at most one element).

Additionally one may note that $B_{\Sigma^i}=\{\ror{i'}{j}|\, i < i' \le m \text{ and } 1 \le j \le n\}$.

Denote $\overline{\lambda}^{i}:=\overline{\lambda_{\Sigma^i}}$. In the next theorem, we describe how the weight diagram of $\lambda$ is changed when we move from the base $\distinguished=\Sigma^{m}$ to $\Sigma^{m-1}, \ldots, \Sigma^0=\antidist$, the most important conclusion being \eqref{itm:ror_change_main} below.
\begin{theorem}
\label{theorem for ror change inequality}
Let $\lambda \in \distHW$. For $0\le i \leq m$ we have:
    \begin{enumerate}
        \item\label{itm:ror_change_1} For every $ 1\le j < n$, we have $ -\overline{\lambda}^i_{-j} <- \overline{\lambda}^i_{-(j+1)}$. 
        \item\label{itm:ror_change_2} For $\overline{\lambda}_i=\overline{\lambda}^i_{i} \le p < k_i$ we have: $D^{\Sigma^i}_\lambda(p)\in \{\times, <, \times >\}$.
        
        \item\label{itm:ror_change_4} For every $p> k_{i+1}$, $D_{\lambda}^{\Sigma^i}(p) = D_{\lambda}(p) $ (here we set $k_{m+1}:=-\infty$).
        \item\label{itm:ror_change_main} For every $i\geq 1$, the odd reflection $r_{i, j}$ changes the diagram of $\lambda$ if and only if $$M_i +1 - (k_i -\overline{\lambda}_i) \le j \le M_i.$$
    \end{enumerate}
\end{theorem}

Note that the first statement means that for all $p\in \Z$, we have: $D^{\Sigma^i}_\lambda(p) \neq \times^k <$ for $k\geq 1$. The last statement implies that $r_{i, j}$ changes the diagram of $\lambda$ only if $k_i>\overline{\lambda}_i$ (that is, only if the $i$-th arrow in $D_{\lambda}$ has distinct beginning and endpoint).

\begin{proof}
    We prove the statements \eqref{itm:ror_change_1}-\eqref{itm:ror_change_4} for $\Sigma^{i}$ by "reversed" induction on $i$ (which is to say we consider the bases $\Sigma^{m},\dots, \Sigma^0$ in that order). 
    
    The statement \eqref{itm:ror_change_main} will be obtained along the way for every $i$.

    For the base case consider $i=m$ ($\Sigma^m=\distinguished$). In that case, the statements \eqref{itm:ror_change_1}-\eqref{itm:ror_change_4} hold by the properties of the distinguished base and the definition of the arrow diagram.

{\bf Step:} Let $i<m$, and assume that the statement holds for $\Sigma^{i+1}$. Let us write $\mu:= {\lambda}^{i+1}$, and $D:=D^{\Sigma^{i+1}}_\lambda$ for short.

First of all, we notice that the weight $\overline{\mu}$ was obtained from $\overline{\lambda}$ by adding some right odd roots of the form $\eps_{i'}-\delta_{j}$ for $i'>i+1$ so $\overline{\mu}_{s}=\overline{\lambda}_{s} < k_s$ for any $1\leq s\leq i+1$.

    We will perform a sequence of reflections on $\Sigma^{i+1}$: $r_{i+1,n}\circ \dots \circ r_{i+1,1}(\Sigma^{i+1})$ (this sequence of reflections is clearly valid and the resulting base is $\Sigma^i$).

By the definition of arrow diagram, if $\overline{\lambda}_{i+1} = k_{i+1}$ then $D_{\lambda}(\overline{\lambda}_{i+1})$ is the symbol $>$. In that case, we can use the induction assumption, namely, that the statement \eqref{itm:ror_change_4} holds for $\Sigma^{i+1}$: we then have $$D_{\lambda}(\overline{\lambda}_{i+1}) =D(\overline{\lambda}_{i+1}) = D(\overline{\mu}_{i+1}) $$ since $\overline{\mu}_{i+1} = k_{i+1}>k_{i+2}$. Thus the symbol in $D(\overline{\lambda}_{i+1}) = D(\overline{\mu}_{i+1})$ is $>$. In that case, by \cref{lem:generalities_on_ror_changes}, the right odd reflections $r_{i+1,n}, \dots,  r_{i+1,1}$ will not change the diagram $D$. 

We conclude: if $\overline{\mu}_{i+1} = k_{i+1}$, then the right odd reflections $r_{i+1,n}, \dots,  r_{i+1,1}$ will not change the diagram of $\lambda$. In that case, the statement of the lemma clearly holds for $\Sigma^i$, and we are done.

From now on, we assume that $\overline{\mu}_{i+1} =  \overline{\lambda}_{i+1} < k_{i+1}$. In that case, $D_{\lambda}(\overline{\lambda}_{i+1}) = D(\overline{\mu}_{i+1}) =\times$.

 By the induction assumption, \eqref{itm:ror_change_2} holds for $\Sigma^{i+1}$: each $D(p)$ for $\overline{\mu}_{i+1}\le p < k_{i+1}$ is either $\times$, or $<$, or $\times >$. 
    
    As such, for every such $p$, there exists $1\le j_p \le n$ such that $-\overline{\mu}_{-j_p}= p$. 
    
    By induction assumption, \eqref{itm:ror_change_1} holds for $\Sigma^{i+1}$, and we deduce that $j_{p+1}=j_p +1$. So there exists $J:=j_{\overline{\mu}_{i+1}}$ such that $-\overline{\mu}_{-(J+c)}= p+c$ for $c=0, 1, \dots, k_{i+1} -1 -\overline{\mu}_{i+1}$.

By the induction assumption \eqref{itm:ror_change_4}, for all $ p\geq k_{i+1}>k_{i+2}$, we have: $D_{\lambda}(p) = D(p)$, so both diagrams $D_{\lambda}, D$ have $n-M_{i+1}$ symbols $<$ or $\times$ to the right of position $k_{i+1}$, and both diagrams have either $\circ$ or $>$ in position $k_{i+1}$.

By \cref{lem:generalities_on_ror_changes} \eqref{itm:gen_ror_change3}, the reflections $r_{i+1, 1}, \ldots, r_{i+1, J-1}$ will not change the diagram $D$, while the reflections $r_{i+1, J}, \ldots, r_{i+1, J+k_{i+1}-1-\overline{\mu}_{i+1}}$ will change $D$, each time moving the $\times$ symbol (originally in position $\overline{\mu}_{i+1}$ in $D$) one position to the right.

This symbol $\times$ reaches the position $k_{i+1}$ once the reflections $r_{i+1, 1}, \ldots, r_{i+1, (J+k_{i+1}-1-\overline{\mu}_{i+1})}$ are performed. 

Let $\Sigma:=r_{i+1, (J+k_{i+1}-1-\overline{\mu}_{i+1})}\circ \ldots\circ r_{i+1, 1}.\Sigma^{i+1}$. Denote by $\mu'$ the newly obtained weight such that $$\overline{\mu'} = \overline{\lambda}_{\Sigma}$$ and by $D'$ the obtained weight diagram: $D'$ is obtained from $D$ by moving the $\times$ symbol from position $\overline{\mu}_{i+1}$ to position $\overline{\mu'}_i = k_{i+1}$. The definition of $\mu'$ implies that $-\overline{\mu'}_{-(J+k_{i+1}-1-\overline{\mu}_{i+1})} = \overline{\mu'}_{i+1}=k_{i+1}$. 

Since $D(k_{i+1})\in \{\circ, >\}$, we conclude that $-\overline{\mu'}_{-j} \neq \overline{\mu'}_{i+1}$ for any $j> J+k_{i+1}-1-\overline{\mu}_{i+1}$. By \cref{lem:generalities_on_ror_changes}, this implies that the odd reflections $r_{i+1,j}$ for $ j>J+k_{i+1}-1-\overline{\mu}_{i+1}$ do not change the diagram $D'$, meaning that $\mu'=\lambda^i$ and $D' = D^{\Sigma^i}_{\lambda}$.

\mbox{}

Let us now check that the diagram $D'$ satisfies conditions \eqref{itm:ror_change_1}-\eqref{itm:ror_change_4}.

Recall that $D'$ was obtained from $D$ by moving the $\times$ symbol from position $\overline{\mu}_{i+1}$ to position $\overline{\mu'}_i = k_{i+1}$, and that the diagram $D$ satisfies the required conditions by the induction assumption.

Property \eqref{itm:ror_change_1} says that there is no position with both a $\times$ symbol and a $<$ symbol. This holds for $D$ by the induction assumption, and in all positions $p\neq k_{i+1}$, we have: the set of symbols in $D'(p)$ is a subset of the symbols in $D(p)$, differing by at most one symbol $\times$. 
So we only need to check that we don't have both a $\times$ symbol and a $<$ symbol in position $k_{i+1}$ of $D'$. But by definition, $D'(k_{i+1})$ is either $\times$ or $\times >$, so Property \eqref{itm:ror_change_1} for $D' = D^{\Sigma^i}_{\lambda}$ is proved.

To show Property \eqref{itm:ror_change_4}, recall that this property for $D$ means that $D_{\lambda}(p)= D(p)$ for any $p>k_{i+2}$. By the definition of $D'$, we also have: $D(p)=D'(p)$ for any $p>k_{i+1}$, so Property \eqref{itm:ror_change_4} for $D'$ is proved.

To show Property \eqref{itm:ror_change_2} for $D'$, recall that $\overline{\lambda}^i = \overline{\mu'}$ is obtained from both $\overline{\lambda}^{i+1} = \overline{\mu}$ and $\overline{\lambda}$ by adding some right odd roots of the form $\eps_{i'} - \delta_j$ for some $i'>i$ and some $j$; so $$\overline{\lambda}^i_i = \overline{\mu'}_i = \overline{\mu}_i = \overline{\lambda}_i.$$

Recall that $D_{\lambda}(p)= D'(p)$ for any $p>k_{i+1}$ by Property \eqref{itm:ror_change_4} for $D'$. So Property \eqref{itm:ror_change_2} for $D'(p)$ where $\max\{k_{i+1}, \overline{\lambda}^i_i\}\leq p<k_i$ follows from the same property for $D_{\lambda}$, which is part of the definition of an arrow diagram.

So we only need to establish that $D'(p)\in \{\times, <, \times >\}$ for all $\overline{\lambda}^i_i\leq p \leq k_{i+1}$ (if the set of such $p$'s is non-empty).

For $p=k_{i+1}$, we have already seen that $D'(k_{i+1})$ is either $\times$ or $\times >$.

For $\overline{\lambda}^i_i\leq p < k_{i+1}$, recall that
 $$\overline{\lambda}^i_i = \overline{\mu'}_i = \overline{\mu}_i = \overline{\lambda}_i > \overline{\lambda}_{i+1} = \overline{\mu}_{i+1}  $$
so any such $p$ satisfies $\overline{\lambda}_{i+1}\leq p\leq k_{i+1}$. But we established that $D(p)=D'(p)$ for any $\overline{\lambda}_{i+1}\leq p\leq k_{i+1}$, and thus by the induction assumption (Property \eqref{itm:ror_change_2} for $D$), we have:
$$ \overline{\lambda}^i_i\leq p < k_{i+1}, \;\;\; \;\;\; D(p)=D'(p)\in \{\times, <, \times >\} .$$
This proves Property \eqref{itm:ror_change_2} for $D'$. 

Finally, since the diagrams $D_{\lambda}, D, D'$ coincide in positions $k_{i+1}+1,k_{i+1}+2, \ldots$, the diagram $D'$ has exactly $n-M_{i+1}$ symbols $<$ in these positions. By definition of $\mu'=\lambda^i$, we have: $$-\overline{\mu'}_{-j} = \overline{\mu'}_{i+1}=k_{i+1}$$
only for $j=J+k_{i+1}-1-\overline{\mu}_{i+1}$. Thus $J+k_{i+1}-1-\overline{\mu}_{i+1} = M_{i+1}$ so $J=M_{i+1} - (k_{i+1}-1-\overline{\mu}_{i+1} )$. Thus the odd reflections $r_{i+1, j}$ changing the diagram of $\lambda$ are those where $$J=M_{i+1} - (k_{i+1}-1-\overline{\mu}_{i+1} )\leq j\leq M_{i+1}$$
(the last inequality follows from our computation that exactly $J+k_{i+1}-\overline{\mu}_{i+1}$ odd reflections of the form $r_{i+1, j}$ change the diagram $D$). This proves Property \eqref{itm:ror_change_main} for $i+1$.

This completes the proof of the theorem.
\end{proof}

\begin{remark}
    For a weight $\mu$ such that $\overline{\mu}_{\distinguished}= \overline{\lambda}_{\Sigma^i}$, consider the weight $\nu$ for $\mathfrak{gl}(i|n)$ given by $\overline{\nu}_{\distinguished_{(i|n)}}:=(\mu_1, \ldots, \mu_i|\mu_{-1}, \ldots, \mu_{-n})$ (here the base $\distinguished_{(i|n)}$ is the distinguished base for $\mathfrak{gl}(i|n)$).  Essentially, \cref{theorem for ror change inequality} shows that $\nu\in\Lambda^+_{i|n}$. 
   
\end{remark}
    Further, examining the proof of the theorem we find that:
    
    \begin{cor}
    \label{Sigma i diagarm}
        \mbox{}
        \begin{itemize}
            \item If $\overline{\lambda}_{i}=k_{i}$, the diagram $D^{\Sigma^i}_\lambda$ remains unchanged when we perform the odd reflections $r_{i, j}$ for all $j$.
            \item If $\overline{\lambda}_{i}\neq k_{i}$, the diagram $D^{\Sigma^{i-1}}_\lambda$ is obtained from the diagram $D^{\Sigma^i}_\lambda$ by moving the $\times$ symbol appearing in position $\overline{\lambda}_{i}$ of $D^{\Sigma^i}_\lambda$ to the position $k_{i}$.
        \end{itemize}
    \end{cor}
Note that $\overline{\lambda}_{i}\neq k_{i}$ implies that $D^{\Sigma^{i}}_\lambda(\overline{\lambda}_{i})$ must be either $\times$ or $\times >$, so we have a $\times$ symbol we can move.

\begin{cor}
\label{bounds on a i}
    Let $\lambda\in \distHW$. Let $\{k_1,\dots,k_m\}$ be the arrow endpoints in the arrow diagram of $\lambda$. For any $\Sigma\in \mathbb{S}$, we have: $\overline{\lambda}_i\le\left(\overline{\lambda_\Sigma}\right)_i\le k_i$. 
\end{cor}

\begin{remark}
\label{analogous arrow diagram}
    At this point it is prudent to note that for any proposition related to arrow diagrams one may show an analogous statement by replacing the arrows of the diagrams with ones that start at $<$ symbols. 

    This is done by defining $N_j$ as the number of $>$ symbols (where $\times$ symbols are counted as $>$) to the left of the $j$th (from the left) $<$ symbol and taking $l_j$ analagously to $k_i$ by 
    $$l_j:=\min\{c\in\mathbb{Z}\;|\;\forall 1\le i \le m, \forall  j > j' \ge 1,\;c\ge -\overline{\lambda}_{-j},\, c\neq \overline{\lambda}_{i}\; \text{ and } \; c \neq l_{j'}\}.$$
\end{remark}

\begin{cor}
\label{under arrow changes diagram}
    Let $\lambda\in \distHW$, consider the arrow diagram of $\lambda$ given by $k_1,\dots, k_m$, let $1\le i \le n$ and $1\le j \le n$. 
    $$\overline{\lambda}_i \le -\overline{\lambda}_{-j} < k_i \implies \ror{i}{j} \text{ changes the diagram of }\lambda$$
\end{cor}
\begin{proof}
    Observe that if  $\overline{\lambda}_i \le -\overline{\lambda}_{-j} < k_i $ then  $ \overline{\lambda}_i \le -\overline{\lambda}_{-j} \le -\overline{\lambda}_{-M_i} <  k_i$ thus $-\overline{\lambda}_{-M_i} -(-\overline{\lambda}_{-j}) \le k_i - \overline{\lambda}_i - 1 $ (RHS is the distance between 2 positions greater or equal to $\overline{\lambda}_i$ and strictly less than $k_i$). 

    Thus $M_i \ge j \underset{(1)}{\ge} M_i -(-\overline{\lambda}_{-M_i} -(-\overline{\lambda}_{-j})) \ge M_i - (k_i - \overline{\lambda}_i - 1) $ (where inequality $(1)$ is due to the fact that $-\overline{\lambda}_{-(j+1)} \ge -\overline{\lambda}_{-j} + 1$) thus by Theorem \ref{theorem for ror change inequality} $\ror{i}{j}$ changes the diagram of $\lambda$.
\end{proof}

\begin{example}
    Expanding on \cref{ex:first arrow diagram}, we draw below the diagrams $D^{\Sigma^i}_{\lambda}$ for $\lambda$ as in \cref{ex:first arrow diagram} and $i=4,3,2,1,0$. The solid arrows in the diagrams are the arrows remaining from the arrow diagram of $D^{\Sigma^4}_\lambda = D_\lambda$, and the dashed arrows drawn below indicate the changes in the diagrams when passing from base$\Sigma^{i+1}$ to base $\Sigma^i$. Recall that these changes, if they happen, are given by moving a $\times$ symbol. 
    
    We find that 
    the diagram of $D^{\Sigma^4}_\lambda = D_\lambda$ is given by

    
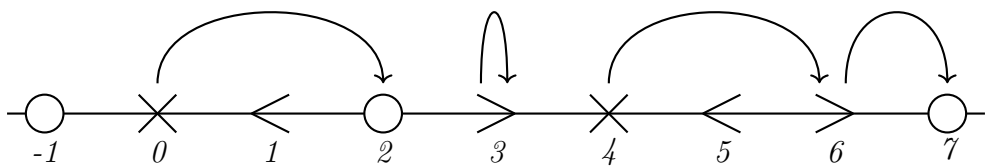
\begin{equation*}
    \begin{tikzpicture}
    
        \draw[thick, -] (-6.5 ,0)--(6.5, 0);
        
        \draw[thick, fill=white] (-6,0) circle [radius=0.25cm];
        \node[] at (-6,-0.5) {-1};
        
        % times
        \draw[thick, -] (-4.5,0)--(-4.5 + 0.25 ,0.25);
        \draw[thick, -] (-4.5,0)--(-4.5 - 0.25 ,0.25);
        \draw[thick, -] (-4.5,0)--(-4.5 + 0.25 ,-0.25);
        \draw[thick, -] (-4.5,0)--(-4.5 - 0.25 ,-0.25);
        \node[] at (-4.5,-0.5) {0};
        
        % le
        \draw[thick, -] (-3 -0.25, 0)--(-3 + 0.25, 0.25);
        \draw[thick, -] (-3 -0.25, 0)--(-3 + 0.25, -0.25);
        \node[] at (-3,-0.5) {1};
        
        \draw[thick, fill=white] (-1.5,0) circle [radius=0.25cm];
        \node[] at (-1.5,-0.5) {2};
        
        % ge
        \draw[thick, -] (0 +0.25, 0)--(0 - 0.25, 0.25);
        \draw[thick, -] (0 +0.25, 0)--(0 - 0.25, -0.25);
        \node[] at (0,-0.5) {3};
        
        % times
        \draw[thick, -] (1.5,0)--(1.5 + 0.25 ,0.25);
        \draw[thick, -] (1.5,0)--(1.5 - 0.25 ,0.25);
        \draw[thick, -] (1.5,0)--(1.5 + 0.25 ,-0.25);
        \draw[thick, -] (1.5,0)--(1.5 - 0.25 ,-0.25);
        \node[] at (1.5,-0.5) {4};
        
        % le
        \draw[thick, -] (3 -0.25, 0)--(3 + 0.25, 0.25);
        \draw[thick, -] (3 -0.25, 0)--(3 + 0.25, -0.25);
        \node[] at (3,-0.5) {5};
        
        % ge
        \draw[thick, -] (4.5 +0.25, 0)--(4.5 - 0.25, 0.25);
        \draw[thick, -] (4.5 +0.25, 0)--(4.5 - 0.25, -0.25);
        \node[] at (4.5,-0.5) {6};
        
        \draw[thick, fill=white] (6,0) circle [radius=0.25cm];
        \node[] at (6,-0.5) {7};
        
        % arrows
        \draw[thick, ->] (-4.5 ,0.4) .. controls (-4.5 ,1.65) and (-1.5 , 1.65) .. (-1.5 ,0.4);
        \draw[thick, ->] (-0.2 ,0.4) .. controls (-0.2 ,1.65) and (0.15 , 1.65) .. (0.15 ,0.4);
        \draw[thick, ->] (1.5 ,0.4) .. controls (1.5 ,1.65) and (4.5 -0.2 , 1.65) .. (4.5 -0.2 ,0.4);
        \draw[thick, ->] (4.5 +0.15 ,0.4) .. controls (4.5 +0.15 , 1.65) and (6 , 1.65) .. (6 ,0.4);
        
        % cross
        %\draw[thick, -] (-6,0)--(-6 + 0.25 ,0.25);
        %\draw[thick, -] (-6,0)--(-6 - 0.25 ,0.25);
        %\draw[thick, -] (-6,0)--(-6 + 0.25 ,-0.25);
        %\draw[thick, -] (-6,0)--(-6 - 0.25 ,-0.25);
        
        % less than 
        %\draw[thick, -] (-3 -0.25, 0)--(-3 + 0.25, 0.25);
        %\draw[thick, -] (-3 -0.25, 0)--(-3 + 0.25, -0.25);
        
        % greater than 
        %\draw[thick, -] (-3 +0.25, 0)--(-3 - 0.25, 0.25);
        %\draw[thick, -] (-3 +0.25, 0)--(-3 - 0.25, -0.25);
        
    \end{tikzpicture}
\end{equation*}    
    The diagram of $D^{\Sigma^3}_\lambda$ is given by
    
    
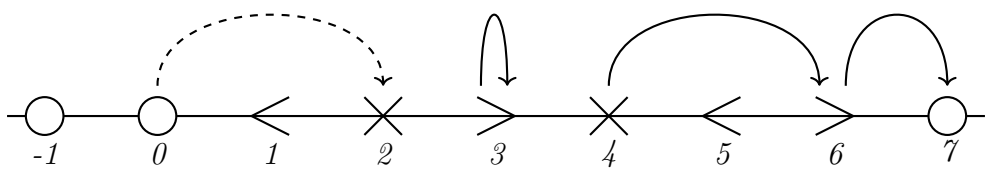
\begin{equation*}
    \begin{tikzpicture}
    
        \draw[thick, -] (-6.5 ,0)--(6.5, 0);
        
        \draw[thick, fill=white] (-6,0) circle [radius=0.25cm];
        \node[] at (-6,-0.5) {-1};
        
        \draw[thick, fill=white] (-4.5,0) circle [radius=0.25cm];
        \node[] at (-4.5,-0.5) {0};
        
        % le
        \draw[thick, -] (-3 -0.25, 0)--(-3 + 0.25, 0.25);
        \draw[thick, -] (-3 -0.25, 0)--(-3 + 0.25, -0.25);
        \node[] at (-3,-0.5) {1};
        
        % times
        \draw[thick, -] (-1.5,0)--(-1.5 + 0.25 ,0.25);
        \draw[thick, -] (-1.5,0)--(-1.5 - 0.25 ,0.25);
        \draw[thick, -] (-1.5,0)--(-1.5 + 0.25 ,-0.25);
        \draw[thick, -] (-1.5,0)--(-1.5 - 0.25 ,-0.25);
        \node[] at (-1.5,-0.5) {2};
        
        % ge
        \draw[thick, -] (0 +0.25, 0)--(0 - 0.25, 0.25);
        \draw[thick, -] (0 +0.25, 0)--(0 - 0.25, -0.25);
        \node[] at (0,-0.5) {3};
        
        % times
        \draw[thick, -] (1.5,0)--(1.5 + 0.25 ,0.25);
        \draw[thick, -] (1.5,0)--(1.5 - 0.25 ,0.25);
        \draw[thick, -] (1.5,0)--(1.5 + 0.25 ,-0.25);
        \draw[thick, -] (1.5,0)--(1.5 - 0.25 ,-0.25);
        \node[] at (1.5,-0.5) {4};
        
        % le
        \draw[thick, -] (3 -0.25, 0)--(3 + 0.25, 0.25);
        \draw[thick, -] (3 -0.25, 0)--(3 + 0.25, -0.25);
        \node[] at (3,-0.5) {5};
        
        % ge
        \draw[thick, -] (4.5 +0.25, 0)--(4.5 - 0.25, 0.25);
        \draw[thick, -] (4.5 +0.25, 0)--(4.5 - 0.25, -0.25);
        \node[] at (4.5,-0.5) {6};
        
        \draw[thick, fill=white] (6,0) circle [radius=0.25cm];
        \node[] at (6,-0.5) {7};
        
        % arrows
        \draw[thick, dashed, ->] (-4.5 ,0.4) .. controls (-4.5 ,1.65) and (-1.5 , 1.65) .. (-1.5 ,0.4);
        \draw[thick, ->] (-0.2 ,0.4) .. controls (-0.2 ,1.65) and (0.15 , 1.65) .. (0.15 ,0.4);
        \draw[thick, ->] (1.5 ,0.4) .. controls (1.5 ,1.65) and (4.5 -0.2 , 1.65) .. (4.5 -0.2 ,0.4);
        \draw[thick, ->] (4.5 +0.15 ,0.4) .. controls (4.5 +0.15 , 1.65) and (6 , 1.65) .. (6 ,0.4);
        
        % cross
        %\draw[thick, -] (-6,0)--(-6 + 0.25 ,0.25);
        %\draw[thick, -] (-6,0)--(-6 - 0.25 ,0.25);
        %\draw[thick, -] (-6,0)--(-6 + 0.25 ,-0.25);
        %\draw[thick, -] (-6,0)--(-6 - 0.25 ,-0.25);
        
        % less than 
        %\draw[thick, -] (-3 -0.25, 0)--(-3 + 0.25, 0.25);
        %\draw[thick, -] (-3 -0.25, 0)--(-3 + 0.25, -0.25);
        
        % greater than 
        %\draw[thick, -] (-3 +0.25, 0)--(-3 - 0.25, 0.25);
        %\draw[thick, -] (-3 +0.25, 0)--(-3 - 0.25, -0.25);
        
    \end{tikzpicture}
\end{equation*}

    Here we make the first arrow dashed to indicate that a $\times$ symbol has travelled along it. The solid arrows are left in the same start and endpoints as defined for $D_\lambda$ which are again left to illustrate the process.
    
    The diagram of $D^{\Sigma^2}_\lambda$ is given by

    
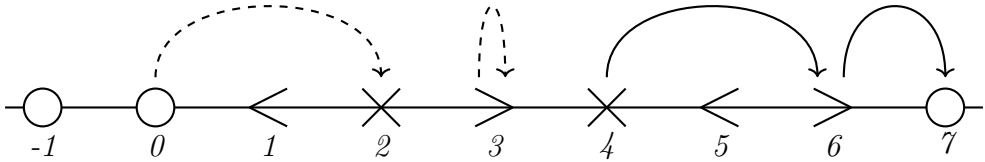
\begin{equation*}
    \begin{tikzpicture}
    
        \draw[thick, -] (-6.5 ,0)--(6.5, 0);
        
        \draw[thick, fill=white] (-6,0) circle [radius=0.25cm];
        \node[] at (-6,-0.5) {-1};
        
        \draw[thick, fill=white] (-4.5,0) circle [radius=0.25cm];
        \node[] at (-4.5,-0.5) {0};
        
        % le
        \draw[thick, -] (-3 -0.25, 0)--(-3 + 0.25, 0.25);
        \draw[thick, -] (-3 -0.25, 0)--(-3 + 0.25, -0.25);
        \node[] at (-3,-0.5) {1};
        
        % times
        \draw[thick, -] (-1.5,0)--(-1.5 + 0.25 ,0.25);
        \draw[thick, -] (-1.5,0)--(-1.5 - 0.25 ,0.25);
        \draw[thick, -] (-1.5,0)--(-1.5 + 0.25 ,-0.25);
        \draw[thick, -] (-1.5,0)--(-1.5 - 0.25 ,-0.25);
        \node[] at (-1.5,-0.5) {2};
        
        % ge
        \draw[thick, -] (0 +0.25, 0)--(0 - 0.25, 0.25);
        \draw[thick, -] (0 +0.25, 0)--(0 - 0.25, -0.25);
        \node[] at (0,-0.5) {3};
        
        % times
        \draw[thick, -] (1.5,0)--(1.5 + 0.25 ,0.25);
        \draw[thick, -] (1.5,0)--(1.5 - 0.25 ,0.25);
        \draw[thick, -] (1.5,0)--(1.5 + 0.25 ,-0.25);
        \draw[thick, -] (1.5,0)--(1.5 - 0.25 ,-0.25);
        \node[] at (1.5,-0.5) {4};
        
        % le
        \draw[thick, -] (3 -0.25, 0)--(3 + 0.25, 0.25);
        \draw[thick, -] (3 -0.25, 0)--(3 + 0.25, -0.25);
        \node[] at (3,-0.5) {5};
        
        % ge
        \draw[thick, -] (4.5 +0.25, 0)--(4.5 - 0.25, 0.25);
        \draw[thick, -] (4.5 +0.25, 0)--(4.5 - 0.25, -0.25);
        \node[] at (4.5,-0.5) {6};
        
        \draw[thick, fill=white] (6,0) circle [radius=0.25cm];
        \node[] at (6,-0.5) {7};
        
        % arrows
        \draw[thick, dashed, ->] (-4.5 ,0.4) .. controls (-4.5 ,1.65) and (-1.5 , 1.65) .. (-1.5 ,0.4);
        \draw[thick, dashed, ->] (-0.2 ,0.4) .. controls (-0.2 ,1.65) and (0.15 , 1.65) .. (0.15 ,0.4);
        \draw[thick, ->] (1.5 ,0.4) .. controls (1.5 ,1.65) and (4.5 -0.2 , 1.65) .. (4.5 -0.2 ,0.4);
        \draw[thick, ->] (4.5 +0.15 ,0.4) .. controls (4.5 +0.15 , 1.65) and (6 , 1.65) .. (6 ,0.4);
        
        % cross
        %\draw[thick, -] (-6,0)--(-6 + 0.25 ,0.25);
        %\draw[thick, -] (-6,0)--(-6 - 0.25 ,0.25);
        %\draw[thick, -] (-6,0)--(-6 + 0.25 ,-0.25);
        %\draw[thick, -] (-6,0)--(-6 - 0.25 ,-0.25);
        
        % less than 
        %\draw[thick, -] (-3 -0.25, 0)--(-3 + 0.25, 0.25);
        %\draw[thick, -] (-3 -0.25, 0)--(-3 + 0.25, -0.25);
        
        % greater than 
        %\draw[thick, -] (-3 +0.25, 0)--(-3 - 0.25, 0.25);
        %\draw[thick, -] (-3 +0.25, 0)--(-3 - 0.25, -0.25);
        
    \end{tikzpicture}
\end{equation*}

    Here we note that as the arrow starting at position $3$ ends at the same position, no change occurs in the diagram itself.

    The diagram of $D^{\Sigma^1}_\lambda$ is given by

    
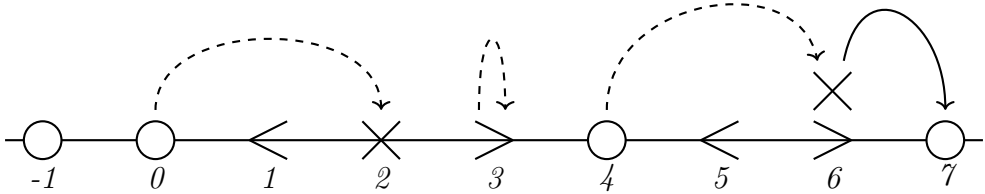
\begin{equation*}
    \begin{tikzpicture}
    
        \draw[thick, -] (-6.5 ,0)--(6.5, 0);
        
        \draw[thick, fill=white] (-6,0) circle [radius=0.25cm];
        \node[] at (-6,-0.5) {-1};
        
        \draw[thick, fill=white] (-4.5,0) circle [radius=0.25cm];
        \node[] at (-4.5,-0.5) {0};
        
        % le
        \draw[thick, -] (-3 -0.25, 0)--(-3 + 0.25, 0.25);
        \draw[thick, -] (-3 -0.25, 0)--(-3 + 0.25, -0.25);
        \node[] at (-3,-0.5) {1};
        
        % times
        \draw[thick, -] (-1.5,0)--(-1.5 + 0.25 ,0.25);
        \draw[thick, -] (-1.5,0)--(-1.5 - 0.25 ,0.25);
        \draw[thick, -] (-1.5,0)--(-1.5 + 0.25 ,-0.25);
        \draw[thick, -] (-1.5,0)--(-1.5 - 0.25 ,-0.25);
        \node[] at (-1.5,-0.5) {2};
        
        % ge
        \draw[thick, -] (0 +0.25, 0)--(0 - 0.25, 0.25);
        \draw[thick, -] (0 +0.25, 0)--(0 - 0.25, -0.25);
        \node[] at (0,-0.5) {3};
        
        \draw[thick, fill=white] (1.5,0) circle [radius=0.25cm];
        \node[] at (1.5,-0.5) {4};
        
        % le
        \draw[thick, -] (3 -0.25, 0)--(3 + 0.25, 0.25);
        \draw[thick, -] (3 -0.25, 0)--(3 + 0.25, -0.25);
        \node[] at (3,-0.5) {5};
        
        % stacked times
        \draw[thick, -] (4.5, 0 + 0.65)--(4.5 + 0.25 ,0.25 + 0.65);
        \draw[thick, -] (4.5, 0 + 0.65)--(4.5 - 0.25 ,0.25 + 0.65);
        \draw[thick, -] (4.5, 0 + 0.65)--(4.5 + 0.25 ,-0.25 + 0.65);
        \draw[thick, -] (4.5, 0 + 0.65)--(4.5 - 0.25 ,-0.25 + 0.65);
        
        % ge
        \draw[thick, -] (4.5 +0.25, 0)--(4.5 - 0.25, 0.25);
        \draw[thick, -] (4.5 +0.25, 0)--(4.5 - 0.25, -0.25);
        \node[] at (4.5,-0.5) {6};
        
        \draw[thick, fill=white] (6,0) circle [radius=0.25cm];
        \node[] at (6,-0.5) {7};
        
        % arrows
        \draw[thick, dashed, ->] (-4.5 ,0.4) .. controls (-4.5 ,1.65) and (-1.5 , 1.65) .. (-1.5 ,0.4);
        \draw[thick, dashed, ->] (-0.2 ,0.4) .. controls (-0.2 ,1.65) and (0.15 , 1.65) .. (0.15 ,0.4);
        \draw[thick, dashed, ->] (1.5 ,0.4) .. controls (1.5  ,1.65 + 0.25) and (4.5 -0.2 - 0.25 , 1.65 + 0.65) .. (4.5 -0.2 ,0.4 + 0.65);
        \draw[thick, ->] (4.5 +0.15 ,0.4 +0.65) .. controls (4.5 +0.15 + 0.25, 1.65 + 0.65) and (6 , 1.65) .. (6 ,0.4);
        
        % cross
        %\draw[thick, -] (-6,0)--(-6 + 0.25 ,0.25);
        %\draw[thick, -] (-6,0)--(-6 - 0.25 ,0.25);
        %\draw[thick, -] (-6,0)--(-6 + 0.25 ,-0.25);
        %\draw[thick, -] (-6,0)--(-6 - 0.25 ,-0.25);
        
        % less than 
        %\draw[thick, -] (-3 -0.25, 0)--(-3 + 0.25, 0.25);
        %\draw[thick, -] (-3 -0.25, 0)--(-3 + 0.25, -0.25);
        
        % greater than 
        %\draw[thick, -] (-3 +0.25, 0)--(-3 - 0.25, 0.25);
        %\draw[thick, -] (-3 +0.25, 0)--(-3 - 0.25, -0.25);
        
    \end{tikzpicture}
\end{equation*}

    Here we notice that as the arrow starting at position $4$ ends at a position with symbol $>$, the symbols become stacked and as was noted after Corollary \ref{Sigma i diagarm}, we find that it is possible to have $\times <$ in a position in the diagram, for some base $\Sigma^i$.
    Finally, the diagram $D^{\Sigma^0}_\lambda= D^{\antidist}_\lambda$ is given by 

    \begin{equation*}
    \begin{tikzpicture}
    
        \draw[thick, -] (-6.5 ,0)--(6.5, 0);
        
        \draw[thick, fill=white] (-6,0) circle [radius=0.25cm];
        \node[] at (-6,-0.5) {-1};
        
        \draw[thick, fill=white] (-4.5,0) circle [radius=0.25cm];
        \node[] at (-4.5,-0.5) {0};
        
        % le
        \draw[thick, -] (-3 -0.25, 0)--(-3 + 0.25, 0.25);
        \draw[thick, -] (-3 -0.25, 0)--(-3 + 0.25, -0.25);
        \node[] at (-3,-0.5) {1};
        
        % times
        \draw[thick, -] (-1.5,0)--(-1.5 + 0.25 ,0.25);
        \draw[thick, -] (-1.5,0)--(-1.5 - 0.25 ,0.25);
        \draw[thick, -] (-1.5,0)--(-1.5 + 0.25 ,-0.25);
        \draw[thick, -] (-1.5,0)--(-1.5 - 0.25 ,-0.25);
        \node[] at (-1.5,-0.5) {2};
        
        % ge
        \draw[thick, -] (0 +0.25, 0)--(0 - 0.25, 0.25);
        \draw[thick, -] (0 +0.25, 0)--(0 - 0.25, -0.25);
        \node[] at (0,-0.5) {3};
        
        \draw[thick, fill=white] (1.5,0) circle [radius=0.25cm];
        \node[] at (1.5,-0.5) {4};
        
        % le
        \draw[thick, -] (3 -0.25, 0)--(3 + 0.25, 0.25);
        \draw[thick, -] (3 -0.25, 0)--(3 + 0.25, -0.25);
        \node[] at (3,-0.5) {5};
        
        % ge
        \draw[thick, -] (4.5 +0.25, 0)--(4.5 - 0.25, 0.25);
        \draw[thick, -] (4.5 +0.25, 0)--(4.5 - 0.25, -0.25);
        \node[] at (4.5,-0.5) {6};
        
        % times
        \draw[thick, -] (6,0)--(6 + 0.25 ,0.25);
        \draw[thick, -] (6,0)--(6 - 0.25 ,0.25);
        \draw[thick, -] (6,0)--(6 + 0.25 ,-0.25);
        \draw[thick, -] (6,0)--(6 - 0.25 ,-0.25);
        \node[] at (6,-0.5) {7};
        
        % arrows
        \draw[thick, dashed, ->] (-4.5 ,0.4) .. controls (-4.5 ,1.65) and (-1.5 , 1.65) .. (-1.5 ,0.4);
        \draw[thick, dashed, ->] (-0.2 ,0.4) .. controls (-0.2 ,1.65) and (0.15 , 1.65) .. (0.15 ,0.4);
        \draw[thick, dashed, ->] (1.5 ,0.4) .. controls (1.5 ,1.65) and (4.5 -0.2 , 1.65) .. (4.5 -0.2 ,0.4);
        \draw[thick, dashed, ->] (4.5 +0.15 ,0.4) .. controls (4.5 +0.15 , 1.65) and (6 , 1.65) .. (6 ,0.4);
        
        % cross
        %\draw[thick, -] (-6,0)--(-6 + 0.25 ,0.25);
        %\draw[thick, -] (-6,0)--(-6 - 0.25 ,0.25);
        %\draw[thick, -] (-6,0)--(-6 + 0.25 ,-0.25);
        %\draw[thick, -] (-6,0)--(-6 - 0.25 ,-0.25);
        
        % less than 
        %\draw[thick, -] (-3 -0.25, 0)--(-3 + 0.25, 0.25);
        %\draw[thick, -] (-3 -0.25, 0)--(-3 + 0.25, -0.25);
        
        % greater than 
        %\draw[thick, -] (-3 +0.25, 0)--(-3 - 0.25, 0.25);
        %\draw[thick, -] (-3 +0.25, 0)--(-3 - 0.25, -0.25);
        
    \end{tikzpicture}
\end{equation*}

\end{example}

\begin{lemma}
\label{ctd inequality lower bound ie}
    Let $\lambda \in \distHW$, and let $1\le i \leq i'\leq m$. Then
    $$M_{i'} \leq M_i, \;\; M_{i'} - (k_{i'} - \overline{\lambda}_{i'} -1) \le M_i - (k_i - \overline{\lambda}_i - 1).$$
\end{lemma}
\begin{proof}
Recall from the Definition \ref{Mi def} which defines $M_i$  that $$M_i =\sharp\{j | -\overline{\lambda}_{-j}<k_i\}$$ (the total number of symbols $<$ and $\times$ in $D_{\lambda}$ to the left of position $k_i$). 
By the construction of the arrow diagram, $k_{i'}<k_i$ so $$M_{i'}  =\sharp\{j | -\overline{\lambda}_{-j}<k_{i'}\} \leq \sharp\{j | -\overline{\lambda}_{-j}\leq k_i\} =M_i.$$

Denote again $\lambda^i:=\lambda_{\Sigma^i}$. Then by Theorem \ref{theorem for ror change inequality} we have:\begin{itemize}
    \item $\overline{\lambda}^i_i=\overline{\lambda}_i$,
    \item $D^{\Sigma^{i}}_\lambda(p) = D_\lambda(p)$ for $p>k_i$,
\end{itemize} This implies that $$M_i=\sharp\{j | -\overline{\lambda}^i_{-j}< k_{i}\}$$ (the total number of symbols $<$ and $\times$ in $D^{\Sigma^{i}}_{\lambda}$ to the left of position $k_i$).

The same theorem gives:
$D^{\Sigma^{i}}_\lambda(p) \in \{<, \times, \times >\}$ for $\overline{\lambda}_i \le p \le k_i$, so
$$ \sharp\{j | -\overline{\lambda}^i_{-j}< k_{i}\}-\sharp\{j | -\overline{\lambda}^i_{-j}\leq \overline{\lambda}_{i}\} = k_{i} - \overline{\lambda}_{i} -1 $$ and thus
$$M_{i} - (k_{i} - \overline{\lambda}_{i} -1) =\sharp\{j | -\overline{\lambda}^i_{-j}\leq \overline{\lambda}_{i}\}.$$

We saw that $\overline{\lambda}^i_i=\overline{\lambda}_i$ and $D^{\Sigma^{i}}_\lambda(\overline{\lambda}_i) \in \{<, \times, \times >\}$, and thus $D^{\Sigma^{i}}_\lambda(\overline{\lambda}_i)$ has a $\times$ symbol.

Thus the ray $(-\infty; \overline{\lambda}_{i} ]$ in $D^{\Sigma^{i}}_\lambda$ contains at least one more symbol $\times$ that the ray $(-\infty; \overline{\lambda}_{i+1} ]$ in the same diagram, and the same number of $<$ symbols. On the other hand, as we saw in Corollary \ref{Sigma i diagarm} that, $D^{\Sigma^{i}}_\lambda$ was obtained from $D^{\Sigma^{i+1}}_\lambda $ by moving at most one $\times$ symbol from position $\overline{\lambda}_{i+1}$ to the right.

So the total number of symbols $\times$ or $<$ in the ray $(-\infty; \overline{\lambda}_{i+1} ]$ in $D^{\Sigma^{i}}_\lambda$ is at most one less than the number of theses symbols in the ray $(-\infty; \overline{\lambda}_{i+1} ]$ in the $D^{\Sigma^{i+1}}_\lambda$.

In other words,
\begin{align*}
    M_{i+1} - (k_{i+1} - \overline{\lambda}_{i+1} -1) -1 &=\sharp\{j | -\overline{\lambda}^{i+1}_{-j}\leq \overline{\lambda}_{i+1}\} -1 \leq 
    \sharp\{j | -\overline{\lambda}^{i}_{-j}\leq \overline{\lambda}_{i+1}\} \\&< 
    \sharp\{j | -\overline{\lambda}^{i}_{-j}\leq \overline{\lambda}_{i}\}= M_{i} - (k_{i} - \overline{\lambda}_{i} -1).
\end{align*}
Thus $M_{i+1} - (k_{i+1} - \overline{\lambda}_{i+1} -1) -1 < M_{i} - (k_{i} - \overline{\lambda}_{i} -1)$ and so 
$$ M_{i+1} - (k_{i+1} - \overline{\lambda}_{i+1} -1) -1\leq M_{i} - (k_{i} - \overline{\lambda}_{i} -1).$$
Writing these inequalities for $i, i+1, \ldots, i'-1$ we obtain the required result.
\end{proof}

Following this result, we would like to devise some easier methods to compute the weight diagram $D_{\lambda}^{\Sigma}$ for other $\Sigma \in \mathbb{S}$, not just for $\Sigma=\distinguished$.

One such method can be used to describe the weight diagram with respect to the anti-distinguished base $\antidist$ (reminder: this base is given by word $\delta^n\varepsilon^m$). This is done by tying in the cap diagram of $\lambda$. 

\begin{definition}
\label{cross o def}
    Let $\lambda\in\distHW$. Consider the arrow diagram of $\lambda$ 
    and denote by $\{k_1,\dots,k_m\}$ the endpoints of the arrows. 
    
    Given a symbol $\times$ in $D_{\lambda}$, denote its position by $p$ and consider a sequence of coordinates as follows: 
    
    Let $1\le i_1 \le m$ such that $\overline{\lambda}_{i_1}=p$ (such $i_1$ exists because we have an $\times$ symbol in this position). Then $D_{\lambda}(k_{i_1})\in \{>, \circ\}$ (the endpoint of the arrow starting in $\overline{\lambda}_{i_1}$).
    
    If $D_{\lambda}(k_{i_1})$ is $>$ then there exists $i_2 < i_1$ such that $\overline{\lambda}_{i_2} = k_{i_1}$. By repeating this process we must eventually reach $i_l$ such that, $D_{\lambda}(k_{i_l}) = \circ$. This yields a sequence $1 \le i_l < \dots < i_1 \le m$ such that $D_{\lambda}(\overline{\lambda}_{i_1})=\times$, $D_{\lambda}(k_{i_l})= \circ$ and for every $1\leq s<l$,
    \begin{enumerate}
        \item $\overline{\lambda}_{i_{s + 1}} = k_{i_{s}}$,
        \item $D_{\overline{\lambda}}(k_{i_s}) $ is $>$ for $s < l$.
    \end{enumerate}
    We call $(i_1,\dots, i_l)$ the \textbf{$(\times-\circ)$ sequence of the symbol $\times$ at position $p$}. Position $p$ is called the {\bf beginning of the $(\times-\circ)$ sequence} and position $k_{i_l}$ is called the {\bf end of the $(\times-\circ)$ sequence}. We will sometimes write "the $(\times-\circ)$ sequence $[p, k_{i_l}]$" as shorthand. 
\end{definition}
In other words, the sequence $(i_1,\dots, i_l)$ describes a sequence of arrows from left to right, starting at a symbol $\times$, so that each arrow ending where the next begins, and the rightmost arrow ends with $\circ$ (making this the maximal sequence beginning with the given symbols $\times$).

\begin{remark}
\label{cross o seq order}
Let $k_i$ be the endpoint of an arrow in the arrow diagram $D_{\lambda}$, such that $D_{\lambda}(k_i)=\circ$, one may reconstruct a $(\times-\circ)$ sequence ending in position $k_i$. Recall that $D_{\lambda}(\overline{\lambda}_i)\in \{\times, >\}$. 
    
    In the case when $D_{\lambda}(\overline{\lambda}_i)$ is $>$, there exists $i' > i$ such that $k_{i'}=\overline{\lambda}_{i}$. 
    
    In the case when $D_{\lambda}(\overline{\lambda}_i)= \times$ there is no $i'$ such that $k_{i'}=\overline{\lambda}_i$.

    Therefore we can reconstruct the sequence $1 \le i_l < \dots < i_1 \le m$ in reversed order.
\end{remark}

\begin{remark}\label{rmk:circ_under_sequence_busy}
Given a $(\times-\circ)$ sequence beginning at position $p$ and ending at position $k$, note that every $\circ$ symbol in positions $p+1, \ldots, k$ must be the endpoint of some arrow in $D_{\lambda}$. Therefore, by \cref{cross o seq order}, it is the end of some $(\times-\circ)$ sequence.
\end{remark}

\begin{definition}
\label{x o sequence}
    Let $\lambda\in \distHW$, Two $(\times-\circ)$ sequences $[p, k]$ and $[p', k']$ in $D_{\lambda}$ are called {\bf intersecting} if $[p, k] \cap [p', k']\neq \emptyset$ (as an intersection of segments in $\Z$). Consider the equivalence relation on the set of all $(\times-\circ)$ sequences in $D_{\lambda}$ generated by the pairs of intersecting segments.
    
    For each equivalence class $E$ of $(\times-\circ)$ sequences, we will define the corresponding {\bf interval} $[p_E, k_E]$ as the segment in $\Z$ where $p_E$ is the smallest (leftmost) beginning of an $(\times-\circ)$ sequence in $E$, and $k_E$ is the largest (rightmost) end of an $(\times-\circ)$ sequence in $E$.
\end{definition}

\begin{remark}
    Note that the relation given by whether or not a pair $(\times-\circ)$ sequences are intersecting is not an equivalence relation. Further, given an equivalence class $E$ with respect to the relation, $p_E$ and $k_E$ need not belong to the same $(\times-\circ)$ sequence.
\end{remark}

Recall that for $\lambda\in \distHW$ we have defined a cap diagram, drawn on top of its weight diagram $D_{\lambda}$ (See Definition \ref{cap diagram def}). We denote by $a$ the atypicality of $\lambda$ (the number of $\times$ symbols in $D_{\lambda}$).
\begin{lemma}
\label{x o seq cap end}
    Let $\lambda\in \distHW$ of atypicality $a\ge 0$. Consider the arrow diagram of $\lambda$ 
    with $\{k_1,\dots,k_m\}$ denoting the endpoints of the arrows, and consider the cap diagram of $\lambda$ with caps $\{c_1,\dots, c_a\}$ ($c_i$ being the endpoint of the cap beginning at $\times_i$).

    Then $\{c_1,\dots,c_a\} = \{k_i\,|\,D_{\lambda}(k_i)=\circ \}$.
    That is to say, $\{c_1,\dots,c_a\}$ is precisely the set of arrow endpoints with a $\circ$ symbol in the weight diagram.
    
\end{lemma}
\label{end points of caps}
\begin{proof}
We denote $K_\times(\lambda) := \{k_i\,|\,D_{\lambda}(k_i)=\circ \}$ the set of ``endpoints'' of $(\times-\circ)$ sequences. When there is no ambiguity, we will write $K_{\times} := K_{\times}(\lambda)$.

    First we show that $\sharp K_\times =a$. Recall that $a$ is the number of $\times$ symbols in the weight diagram of $\lambda$. For every $1\le i \le a$, there exists a $(\times-\circ)$ sequence corresponding to $\times_i$ which gives a 1-to-1 correspondence between elements $k\in K_\times$ and $\times$ symbols in the weight diagram $D_{\lambda}$.

    Since we know that $\sharp K_\times =a$ it is sufficient to show that $\{c_1, \dots, c_a\} \supseteq K_\times$. 

    let $k \in K_{\times}(\lambda)$ be the end of some $(\times-\circ)$ sequence. Denote by $E$ the equivalence class of this $(\times-\circ)$ sequence, and consider the corresponding interval $[p_E, k_E]$.

By the definition of an interval, every symbol $\circ$ in the interval $[p_E, k_E]$ is between the beginning and the end of some $(\times-\circ)$ sequence in $E$. Therefore, by \cref{rmk:circ_under_sequence_busy} such a symbol is the end of some $(\times-\circ)$ sequence, and by definition of an interval, such a sequence must begin in a point in the same interval.

This correspondence between $\times$ symbols in $[p_E, k_E]$ and $\circ$ symbols in the same interval implies that 
So for every position $r\in [p_E, k_E]$ such that $D_{\lambda}(r)=\circ$, we have: 
\begin{align*}
\sharp\{ p\leq r\leq q \,|\, D_{\lambda}(r) = \times \}\geq \sharp\{ p\leq r\leq q \,|\, D_{\lambda}(r) = \circ \} 
\end{align*}
By Corollary \ref{cap cor}, every $\circ$ symbol in the interval $[p_E, k_E]$ (in particular, the $\circ$ in position $k$) is the endpoint of some cap in $D_{\lambda}$.

In other words, $k\in \{c_1, \ldots, c_a\}$. Thus we conclude that $K_{\times}(\lambda) \subseteq  \{c_1, \ldots, c_a\}$ and we are done.
    
\end{proof}

The following proposition is immediate from a well-known well known result appearing as  \cite[Exercise 2.56(2)]{cheng2012dualities}
 and is related to the computation of the socle of the Kac module with highest weight $\lambda$ (with respect to the distinguished base $\distinguished$).

\begin{prop}
    Let $\lambda\in \distHW$. Denote by $\{c_1,\dots, c_a\}$ the caps in the cap diagram of $\lambda$. Then $D_{\lambda}^{\antidist}$ is given by 
    $$
    D_{\lambda}^{\antidist}(p)=
    \begin{cases}
        \times &\text{ if } \, p=c_i \, \text{ for some } \, i\in \{1, \ldots, a\}\\
        \circ &\text{ if } \, p=\times_i \, \text{ for some } \, i\in \{1, \ldots, a\}\\
        D_{\lambda}(p) & \text{otherwise.}
    \end{cases}
    $$

\end{prop}
   In other words, the diagram $D_{\lambda}^{\antidist}$ of the highest weight of $L_{\Sigma}(\lambda)$ with respect to the antidistinguished basis $\antidist$ is given by moving all the $\times$ symbols in the diagram $D_{\lambda}$ via the caps in this diagram.

\begin{proof}

 Let $i\geq 1$. Recall that by \cref{theorem for ror change inequality} \eqref{itm:ror_change_main}, if $\overline{\lambda}_i \neq k_i$ then $D^{\Sigma^i}_\lambda$ must have a $\times$ symbol in position $\overline{\lambda}_i = \overline{\lambda}^i_i$ (despite the fact that the diagram $D_\lambda$ need not have a $\times$ symbol in position $\overline{\lambda}_i$ even if $\overline{\lambda}_i \neq k_i$).

By Corollary \ref{Sigma i diagarm}, the diagram $D^{\Sigma^{i-1}}_\lambda$ is either the same as $D^{\Sigma^i}_\lambda$ (which occurs if and only if $\overline{\lambda}^{i}_{i} = k_{i}$), or $D^{\Sigma^{i-1}}_\lambda$ is obtained from $D^{\Sigma^{i}}_\lambda$ by moving a $\times$ symbol from the position $\overline{\lambda}^{i}_{i}$ to the position $k_{i}$.

 If we have $k_i = \overline{\lambda}_{i'}$ for some $1\leq i'<i$ (that is, if $D_{\lambda}(k_i)$ was $>$), then the $\times$ symbol in position $k_i$ of $D^{\Sigma^{i-1}}_\lambda$ will be moved again when we pass from $\Sigma^{i'}$ to $\Sigma^{i'-1}$.

    Thus by moving through all bases $\Sigma^m,\Sigma^{m-1},\dots, \Sigma^{0}$ we find that every symbol $\times$ in $D_\lambda$ must move to the end of its respective $(\times-\circ)$ sequence. 

    We conclude that in the diagram $D^{\antidist}_\lambda =  D^{ \Sigma^0}_{\lambda}$ the positions of the $\times $ symbols are precisely in the ends of the $\times-\circ$ sequences of $D_{\lambda}$.
    
    By Lemma \ref{x o seq cap end}, there is a correspondence between endpoints of caps and ends of $(\times-\circ)$ sequences, so the positions of the $\times$ symbols in $D^{\antidist}_\lambda$ are $c_1, \ldots, c_a$ as required.
\end{proof}

\begin{remark}
    Observe that for any $\lambda\in \distHW$,  $\overline{\lambda}_{\antidist}$ is a shifted highest weight with respect to the anti-distinguished base $\antidist$ and therefore must satisfy $\left(\overline{\lambda}_{\antidist}\right)_i > \left(\overline{\lambda}_{\antidist}\right)_{i+1}$ and $\left(\overline{\lambda}_{\antidist}\right)_{-j} > \left(\overline{\lambda}_{\antidist}\right)_{-(j+1)}$ for all $i, j$. As such, given $D^{\antidist}_{\lambda}$ one may recover $\overline{\lambda}_{\antidist}$.
\end{remark}
% end of arrow diagram section
%beginning of CTD section

\section{The change tracking diagram}
\label{ctd section}

\subsection{The change tracking diagram}
Combining the result shown in Theorem \ref{theorem for ror change inequality} with $B_\Sigma$ shown in Theorem \ref{B Sigma}, we may produce a useful tool to describe the weight diagram of $\lambda$ with respect to any $\Sigma \in \mathbb{S}$.

\begin{definition}
\label{CTD}
Let $\lambda \in \distHW$. The {\bf change tracking diagram (CTD) of $\lambda$} is given by the map $c_\lambda:\mathcal{R}\to \{0,1\}$ where 
$$c_\lambda(\ror{i}{j})=\begin{cases}
			0 &\text{ if } \; \overline{\lambda}_i=k_i \text{ or } j\notin \{M_i-(k_i-a_i-1), \dots, M_i\}\\
            1 & \text{otherwise}
		 \end{cases}$$
\end{definition}

    By \cref{theorem for ror change inequality}, $c_\lambda(\ror{i}{j})=1$ if and only if $\ror{i}{j}$ changes the weight diagram of $\lambda$.

\begin{cor}
    \label{CTD formula}
    Let $\lambda\in \distHW$, and $\Sigma\in \mathbb{S}$. We have:
    $$\overline{\lambda}_\Sigma= \overline{\lambda} + \sum_{\alpha\in B_\Sigma}c_\lambda (\alpha) \cdot \alpha$$
\end{cor}
\begin{proof}
    Observe that by Theorem \ref{B Sigma}, $\Sigma$ is the result of performing reflection by every $\alpha \in B_\Sigma$ in some valid order. 
    Furthermore, for every $\alpha\in B_\Sigma$, $c_\lambda(\alpha)=1$ if and only if $\alpha$ changed the diagram of $\lambda$. In that case, performing the odd reflection $r_{\alpha}$ resulted in adding $\alpha$ to our highest weight, by \cref{lem:generalities_on_ror_changes}. Thus the formula is immediate.
\end{proof}

To visualize $c_\lambda$, we may draw the rectangle describing $\mathcal{R}$ as shown in Example \ref{first rectangle ex}, and color the right odd root $\ror{i}{j}$ in black if $c_\lambda(\ror{i}{j})=1$ and in white otherwise. Next, when desire to compute $\overline{\lambda}_\Sigma$ for some $\Sigma\in \mathbb{S}$, we simply overlay the colored rectangle with $B_\Sigma$, as was done in Example \ref{first rectangle ex}, and add the right odd roots corresponding to black circles to $\overline{\lambda}$.

\begin{example}
Consider $\lambda\in \Lambda^+_{(3|5)}$ such that $\overline{\lambda}= (4\, 3\, 0| 0\, -1\, -3\, -4 \, -5)$. Then the diagram $D_{\overline{\lambda}}$ is 
    
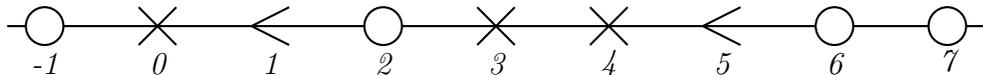
\begin{equation*}
    \begin{tikzpicture}
    
        \draw[thick, -] (-6.5 ,0)--(6.5, 0);
        
        \draw[thick, fill=white] (-6,0) circle [radius=0.25cm];
        \node[] at (-6,-0.5) {-1};
        
        % times
        \draw[thick, -] (-4.5,0)--(-4.5 + 0.25 ,0.25);
        \draw[thick, -] (-4.5,0)--(-4.5 - 0.25 ,0.25);
        \draw[thick, -] (-4.5,0)--(-4.5 + 0.25 ,-0.25);
        \draw[thick, -] (-4.5,0)--(-4.5 - 0.25 ,-0.25);
        \node[] at (-4.5,-0.5) {0};
        
        % le
        \draw[thick, -] (-3 -0.25, 0)--(-3 + 0.25, 0.25);
        \draw[thick, -] (-3 -0.25, 0)--(-3 + 0.25, -0.25);
        \node[] at (-3,-0.5) {1};
        
        \draw[thick, fill=white] (-1.5,0) circle [radius=0.25cm];
        \node[] at (-1.5,-0.5) {2};
            
        % times
        \draw[thick, -] (0,0)--(0 + 0.25 ,0.25);
        \draw[thick, -] (0,0)--(0 - 0.25 ,0.25);
        \draw[thick, -] (0,0)--(0 + 0.25 ,-0.25);
        \draw[thick, -] (0,0)--(0 - 0.25 ,-0.25);
        \node[] at (0,-0.5) {3};
        
        % times
        \draw[thick, -] (1.5,0)--(1.5 + 0.25 ,0.25);
        \draw[thick, -] (1.5,0)--(1.5 - 0.25 ,0.25);
        \draw[thick, -] (1.5,0)--(1.5 + 0.25 ,-0.25);
        \draw[thick, -] (1.5,0)--(1.5 - 0.25 ,-0.25);
        \node[] at (1.5,-0.5) {4};
        
        % le
        \draw[thick, -] (3 -0.25, 0)--(3 + 0.25, 0.25);
        \draw[thick, -] (3 -0.25, 0)--(3 + 0.25, -0.25);
        \node[] at (3,-0.5) {5};
        
        \draw[thick, fill=white] (4.5,0) circle [radius=0.25cm];
        \node[] at (4.5,-0.5) {6};
        
        \draw[thick, fill=white] (6,0) circle [radius=0.25cm];
        \node[] at (6,-0.5) {7};
        
        % cross
        %\draw[thick, -] (-6,0)--(-6 + 0.25 ,0.25);
        %\draw[thick, -] (-6,0)--(-6 - 0.25 ,0.25);
        %\draw[thick, -] (-6,0)--(-6 + 0.25 ,-0.25);
        %\draw[thick, -] (-6,0)--(-6 - 0.25 ,-0.25);
        
        % less than 
        %\draw[thick, -] (-3 -0.25, 0)--(-3 + 0.25, 0.25);
        %\draw[thick, -] (-3 -0.25, 0)--(-3 + 0.25, -0.25);
        
        % greater than 
        %\draw[thick, -] (-3 +0.25, 0)--(-3 - 0.25, 0.25);
        %\draw[thick, -] (-3 +0.25, 0)--(-3 - 0.25, -0.25);
        
    \end{tikzpicture}
\end{equation*}

Drawing the arrow diagram of $\lambda$, we obtain:
    
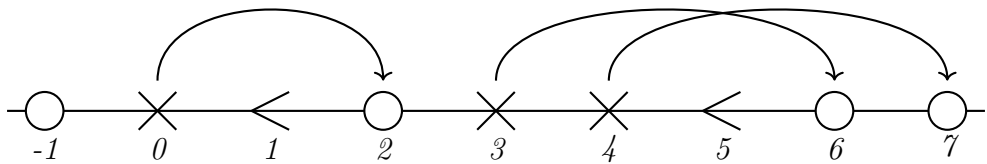
\begin{equation*}
    \begin{tikzpicture}
    
        \draw[thick, -] (-6.5 ,0)--(6.5, 0);
        
        \draw[thick, fill=white] (-6,0) circle [radius=0.25cm];
        \node[] at (-6,-0.5) {-1};
        
        % times
        \draw[thick, -] (-4.5,0)--(-4.5 + 0.25 ,0.25);
        \draw[thick, -] (-4.5,0)--(-4.5 - 0.25 ,0.25);
        \draw[thick, -] (-4.5,0)--(-4.5 + 0.25 ,-0.25);
        \draw[thick, -] (-4.5,0)--(-4.5 - 0.25 ,-0.25);
        \node[] at (-4.5,-0.5) {0};
        
        % le
        \draw[thick, -] (-3 -0.25, 0)--(-3 + 0.25, 0.25);
        \draw[thick, -] (-3 -0.25, 0)--(-3 + 0.25, -0.25);
        \node[] at (-3,-0.5) {1};
        
        \draw[thick, fill=white] (-1.5,0) circle [radius=0.25cm];
        \node[] at (-1.5,-0.5) {2};
            
        % times
        \draw[thick, -] (0,0)--(0 + 0.25 ,0.25);
        \draw[thick, -] (0,0)--(0 - 0.25 ,0.25);
        \draw[thick, -] (0,0)--(0 + 0.25 ,-0.25);
        \draw[thick, -] (0,0)--(0 - 0.25 ,-0.25);
        \node[] at (0,-0.5) {3};
        
        % times
        \draw[thick, -] (1.5,0)--(1.5 + 0.25 ,0.25);
        \draw[thick, -] (1.5,0)--(1.5 - 0.25 ,0.25);
        \draw[thick, -] (1.5,0)--(1.5 + 0.25 ,-0.25);
        \draw[thick, -] (1.5,0)--(1.5 - 0.25 ,-0.25);
        \node[] at (1.5,-0.5) {4};
        
        % le
        \draw[thick, -] (3 -0.25, 0)--(3 + 0.25, 0.25);
        \draw[thick, -] (3 -0.25, 0)--(3 + 0.25, -0.25);
        \node[] at (3,-0.5) {5};
        
        \draw[thick, fill=white] (4.5,0) circle [radius=0.25cm];
        \node[] at (4.5,-0.5) {6};
        
        \draw[thick, fill=white] (6,0) circle [radius=0.25cm];
        \node[] at (6,-0.5) {7};

        % arrows
        \draw[thick, ->] (-4.5 ,0.4) .. controls (-4.5 ,1.65) and (-1.5 , 1.65) .. (-1.5 ,0.4);
        \draw[thick, ->] (0 ,0.4) .. controls (0 ,1.65) and (4.5 , 1.65) .. (4.5 ,0.4);
        \draw[thick, ->] (1.5 ,0.4) .. controls (1.5 ,1.65) and (6 , 1.65) .. (6 ,0.4);
        
        % cross
        %\draw[thick, -] (-6,0)--(-6 + 0.25 ,0.25);
        %\draw[thick, -] (-6,0)--(-6 - 0.25 ,0.25);
        %\draw[thick, -] (-6,0)--(-6 + 0.25 ,-0.25);
        %\draw[thick, -] (-6,0)--(-6 - 0.25 ,-0.25);
        
        % less than 
        %\draw[thick, -] (-3 -0.25, 0)--(-3 + 0.25, 0.25);
        %\draw[thick, -] (-3 -0.25, 0)--(-3 + 0.25, -0.25);
        
        % greater than 
        %\draw[thick, -] (-3 +0.25, 0)--(-3 - 0.25, 0.25);
        %\draw[thick, -] (-3 +0.25, 0)--(-3 - 0.25, -0.25);
        
    \end{tikzpicture}
\end{equation*}

Thus $k_3 = 2$, $k_2=6$ and $k_1=7$ and $M_3=2$, $M_2=5$ and $M_1=5$. 

We now visualize the CTD of $\lambda$: 

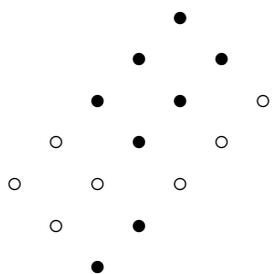
     \begin{equation*}
\begin{tikzpicture}[anchorbase,scale=1.1]
% row 1
\node at (0,1.5) {$\bullet $};

% row 2
\node at (0.5,1) {$\bullet $};
\node at (-0.5, 1) {$\bullet $};

% row 3
\node at (1, 0.5) {$\circ $};
\node at (0, 0.5) {$\bullet $};
\node at (-1, 0.5) {$\bullet $};

% row 4
\node at (0.5, 0) {$\circ $};
\node at (-0.5, 0) {$\bullet $};
\node at (-1.5, 0) {$\circ $};

% row 5
\node at (0, -0.5) {$\circ $};
\node at (-1, -0.5) {$\circ $};
\node at (-2, -0.5) {$\circ $};

% row 6
\node at (-0.5, -1) {$\bullet $};
\node at (-1.5, -1) {$\circ $};

% row 7
\node at (-1, -1.5) {$\bullet $};

%\draw[-] (1.5,0.5)--(1.5,-1);
\end{tikzpicture}
\end{equation*} 

Let us now compute $\overline{\lambda}_\Sigma$ for $\Sigma \in \mathbb{S}$ given by the word $\varepsilon\delta^3\varepsilon^2\delta^2$. Then $\Sigma\cap \mathcal{R} = \{\eps_1-\delta_1, \eps_3-\delta_4\}$, and $B_{\Sigma}=\{\ror{3}{1},\ror{3}{2},\ror{3}{3} , \ror{2}{1}, \ror{2}{2},\ror{2}{3}\}$.
Overlaying $B_\Sigma$ on our CTD we obtain 

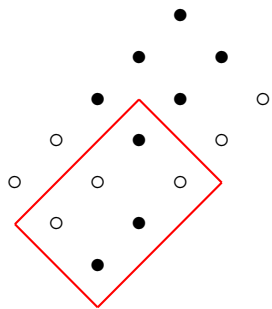
\begin{equation*}
\begin{tikzpicture}[anchorbase,scale=1.1]
% row 1
\node at (0,1.5) {$\bullet $};

% row 2
\node at (0.5,1) {$\bullet $};
\node at (-0.5, 1) {$\bullet $};

% row 3
\node at (1, 0.5) {$\circ $};
\node at (0, 0.5) {$\bullet $};
\node at (-1, 0.5) {$\bullet $};

% row 4
\node at (0.5, 0) {$\circ $};
\node at (-0.5, 0) {$\bullet $};
\node at (-1.5, 0) {$\circ $};

% row 5
\node at (0, -0.5) {$\circ $};
\node at (-1, -0.5) {$\circ $};
\node at (-2, -0.5) {$\circ $};

% row 6
\node at (-0.5, -1) {$\bullet $};
\node at (-1.5, -1) {$\circ $};

% row 7
\node at (-1, -1.5) {$\bullet $};

%\draw[-] (1.5,0.5)--(1.5,-1);
\draw[thick, red, -] (-1, -2)--(-2 ,-1);
\draw[thick, red, -] (-2 ,-1)--(-0.5 ,0.5);
\draw[thick, red, -] (-0.5 ,0.5)--(0.5 ,-0.5);
\draw[thick, red, -] (0.5 ,-0.5)--(-1, -2);

\end{tikzpicture}
\end{equation*}

We notice that the only black circles inside $B_\Sigma$ are given by the right odd roots $\ror{5}{1}$, $\ror{5}{2}$ and $\ror{4}{3}$. Hence
$$\overline{\lambda}_\Sigma=\overline{\lambda} + \ror{5}{1} + \ror{5}{2} +\ror{4}{3}= (4\, 4\, 2| -1\, -2\, -4\, -4 \, -5) $$
So the diagram of $\overline{\lambda}_\Sigma$ is 

    
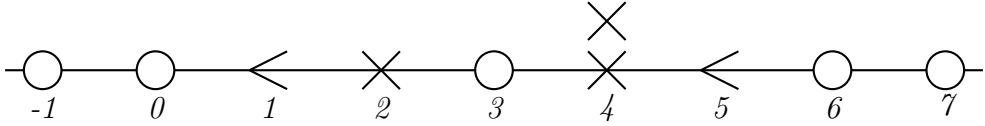
\begin{equation*}
    \begin{tikzpicture}
    
        \draw[thick, -] (-6.5 ,0)--(6.5, 0);
        
        \draw[thick, fill=white] (-6,0) circle [radius=0.25cm];
        \node[] at (-6,-0.5) {-1};
        
        % times
        \draw[thick, fill=white] (-4.5,0) circle [radius=0.25cm];
        \node[] at (-4.5,-0.5) {0};
        
        % le
        \draw[thick, -] (-3 -0.25, 0)--(-3 + 0.25, 0.25);
        \draw[thick, -] (-3 -0.25, 0)--(-3 + 0.25, -0.25);
        \node[] at (-3,-0.5) {1};

        %times
        \draw[thick, -] (-1.5,0)--(-1.5 + 0.25 ,0.25);
        \draw[thick, -] (-1.5,0)--(-1.5 - 0.25 ,0.25);
        \draw[thick, -] (-1.5,0)--(-1.5 + 0.25 ,-0.25);
        \draw[thick, -] (-1.5,0)--(-1.5 - 0.25 ,-0.25);
        \node[] at (-1.5,-0.5) {2};
            
        \draw[thick, fill=white] (0,0) circle [radius=0.25cm];
        \node[] at (0,-0.5) {3};

        % stacked times
        \draw[thick, -] (1.5, 0 + 0.65)--(1.5 + 0.25 ,0.25 + 0.65);
        \draw[thick, -] (1.5, 0 + 0.65)--(1.5 - 0.25 ,0.25 + 0.65);
        \draw[thick, -] (1.5, 0 + 0.65)--(1.5 + 0.25 ,-0.25 + 0.65);
        \draw[thick, -] (1.5, 0 + 0.65)--(1.5 - 0.25 ,-0.25 + 0.65);
        
        % times
        \draw[thick, -] (1.5,0)--(1.5 + 0.25 ,0.25);
        \draw[thick, -] (1.5,0)--(1.5 - 0.25 ,0.25);
        \draw[thick, -] (1.5,0)--(1.5 + 0.25 ,-0.25);
        \draw[thick, -] (1.5,0)--(1.5 - 0.25 ,-0.25);
        \node[] at (1.5,-0.5) {4};

        % le
        \draw[thick, -] (3 -0.25, 0)--(3 + 0.25, 0.25);
        \draw[thick, -] (3 -0.25, 0)--(3 + 0.25, -0.25);
        \node[] at (3,-0.5) {5};
        
        \draw[thick, fill=white] (4.5,0) circle [radius=0.25cm];
        \node[] at (4.5,-0.5) {6};
        
        \draw[thick, fill=white] (6,0) circle [radius=0.25cm];
        \node[] at (6,-0.5) {7};
        
        % cross
        %\draw[thick, -] (-6,0)--(-6 + 0.25 ,0.25);
        %\draw[thick, -] (-6,0)--(-6 - 0.25 ,0.25);
        %\draw[thick, -] (-6,0)--(-6 + 0.25 ,-0.25);
        %\draw[thick, -] (-6,0)--(-6 - 0.25 ,-0.25);
        
        % less than 
        %\draw[thick, -] (-3 -0.25, 0)--(-3 + 0.25, 0.25);
        %\draw[thick, -] (-3 -0.25, 0)--(-3 + 0.25, -0.25);
        
        % greater than 
        %\draw[thick, -] (-3 +0.25, 0)--(-3 - 0.25, 0.25);
        %\draw[thick, -] (-3 +0.25, 0)--(-3 - 0.25, -0.25);
        
    \end{tikzpicture}
\end{equation*}

\end{example}
Once we have computed $\overline{\lambda}_{\Sigma}$ in terms of $c_{\lambda}$ in \cref{CTD formula}, we can easily recover the non-shifted highest weight of $L(\lambda)$ with respect to the base $\Sigma$:
\begin{cor}
    \label{highest weight formula}
    Given $\lambda\in \distHW$, 
    $$\lambda_\Sigma = \lambda +\sum_{\alpha\in B_\Sigma}(c_\lambda-1) (\alpha) \cdot \alpha$$
\end{cor}
\begin{proof}
    Let $\lambda\in\distHW$ and $\Sigma\in \mathbb{S}$. 
    Recall that $\overline{\lambda}_\Sigma=\lambda_\Sigma + \rho_\Sigma$.
    In Fact \ref{reflection weyl vec}, we saw that $\rho_{r_\alpha(\Sigma)}=\rho_\Sigma + \alpha$, and in Fact \ref{odd reflection hw}, we saw that 
    $$
    \lambda_{r_\alpha(\Sigma)}=
    \begin{cases}
        \lambda_\Sigma & (\alpha | \lambda)=0 \\
        \lambda_\Sigma -\alpha & (\alpha| \lambda)\neq 0
    \end{cases}
    $$
    Thus we find that $\ror{i}{j}$ changes the diagram of $\lambda$ if and only if $\ror{i}{j}$ does not change the (non-shifted) highest weight. Combining this with \cref{CTD formula} we find that 
        $$\lambda_\Sigma = \lambda +\sum_{\alpha\in B_\Sigma}(c_\lambda-1) (\alpha) \cdot \alpha.$$
\end{proof}
\begin{remark}
    Similarly to the visual algorithm used to compute the shifted weight $\overline{\lambda}_{\Sigma}$ for some $\Sigma\in \mathbb{S}$, if we wish to compute the non-shifted weight ${\lambda}_{\Sigma}$ for $\Sigma\in \mathbb{S}$, one may again overlay $B_\Sigma$ on the CTD of $\lambda$. This time we subtract from $\lambda$ every root in $B_\Sigma$ that is white (rather than add every root that is black).
\end{remark}

\begin{prop}
\label{CTD 1 iff alpha perp nu}
    Let $\lambda \in \distHW$ and $\alpha\in\mathcal{R}$. We have
    $$\left[\exists \Sigma \in \mathbb{S}:\;(\overline{\lambda}_\Sigma|\alpha)=0\right] \;\; \Longleftrightarrow \;\; c_\lambda(\alpha)=1.$$
\end{prop}
\begin{proof}

\mbox{}

{$\boldsymbol{\Leftarrow}$:} Let $\lambda \in \distHW$ and $\ror{i}{j}\in\mathcal{R}$ such that $c_\lambda(\ror{i}{j})=1$. Then, as mentioned in Definition \ref{CTD}, $\ror{i}{j}$ changes the diagram of $\lambda$. Thus by definition \ref{changes the weight diagram of lambda} there exists some $\Sigma \in \mathbb{S}$ such that $\ror{i}{j}\in \Sigma$ and $\overline{\lambda}_\Sigma \neq \overline{\lambda}_{r_{i,j}\Sigma}$. But in that case, by Lemma \ref{lem:generalities_on_ror_changes} we have: $(\overline{\lambda}_\Sigma|\ror{i}{j})=0 $, as required.

{$\boldsymbol{\Rightarrow}$:} Let $\lambda \in \distHW$ and $\ror{i}{j}\in\mathcal{R}$ such that $c_\lambda(\ror{i}{j})=0$. Let us show that for any $\Sigma\in \mathbb{S}$, we have: $(\overline{\lambda}_\Sigma|\ror{i}{j})\neq0$.

    As $c_\lambda(\ror{i}{j})=0$ we deduce that $\overline{\lambda}_i \neq -\overline{\lambda}_{-j}$ (otherwise $j=M_i-(k_i-\overline{\lambda}_i-1)$ and so  $c_\lambda(\ror{i}{j})=1$ by Theorem \ref{theorem for ror change inequality}). 
    
    Assume $\overline{\lambda}_i < -\overline{\lambda}_{-j}$. 
    
    By our assumption, $c_\lambda(\ror{i}{j})=0$. By Theorem \ref{theorem for ror change inequality}, $c_{\lambda}(\ror{i}{j'})=1$ iff $M_i+1-(k_i-\overline{\lambda}_i)\le j' \le M_i$. Thus either $j>M_i$, or $j<M_i+1-(k_i-\overline{\lambda}_i)$.
    
    If $j<M_i+1-(k_i-\overline{\lambda}_i)$ then there were at least $j$ symbols $<$ in $D_{\lambda}$ to the left of position $\overline{\lambda}_i$. Thus $-\overline{\lambda}_{-j} <\overline{\lambda}_i$ contradicting our assumption.
    
    If $M_i< j$ then there were at most $j-1$ symbols $<$ in $D_{\lambda}$ to the left of position $k_i$. Thus $k_i<-\overline{\lambda}_{-j} $. But for any $\Sigma \in \mathbb{S}$, by Corollary \ref{diagram by base inequality} we have :
    $$(\overline{\lambda}_{\Sigma})_i \leq k_i <-\overline{\lambda}_{-j} \leq  -(\overline{\lambda}_{\Sigma})_{-j}$$ so $(\overline{\lambda}_\Sigma|\ror{i}{j})\neq0 $ for every $\Sigma \in \mathbb{S}$.

    The case where $-\overline{\lambda}_{-j} < \overline{\lambda}_i$ is analogous (as discussed in Remark \ref{analogous arrow diagram}).
\end{proof}
\subsection{The structure of a CTD}
We have now developed a language to speak of weight diagrams, using two primary methods: the arrow diagram and the CTD. We can now observe the general structure of a weight diagram in terms of these methods. 
\begin{example}
Let $\lambda\in \Lambda^+_{5|5}$ such that $\overline{\lambda} = (5\, 4\, 3\, 0\, -1| 1\, 0\, -3\, -4\, -5)$. The diagram $D_\lambda$ is
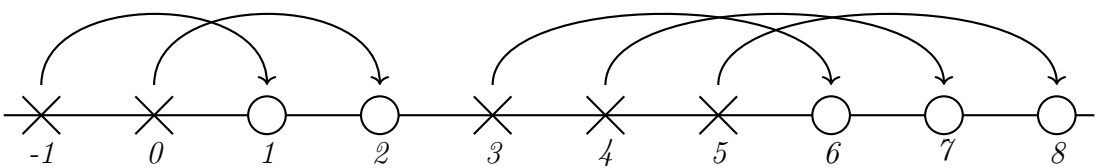
\begin{equation*}
    \begin{tikzpicture}
    
        \draw[thick, -] (-6.5 ,0)--(8, 0);
        
        % cross
        \draw[thick, -] (-6,0)--(-6 + 0.25 ,0.25);
        \draw[thick, -] (-6,0)--(-6 - 0.25 ,0.25);
        \draw[thick, -] (-6,0)--(-6 + 0.25 ,-0.25);
        \draw[thick, -] (-6,0)--(-6 - 0.25 ,-0.25);
        \node[] at (-6,-0.5) {-1};
        
        % cross
        \draw[thick, -] (-4.5,0)--(-4.5 + 0.25 ,0.25);
        \draw[thick, -] (-4.5,0)--(-4.5 - 0.25 ,0.25);
        \draw[thick, -] (-4.5,0)--(-4.5 + 0.25 ,-0.25);
        \draw[thick, -] (-4.5,0)--(-4.5 - 0.25 ,-0.25);
        \node[] at (-4.5,-0.5) {0};
        
        \draw[thick, fill=white] (-3,0) circle [radius=0.25cm];
        \node[] at (-3,-0.5) {1};
        
        \draw[thick, fill=white] (-1.5,0) circle [radius=0.25cm];
        \node[] at (-1.5,-0.5) {2};

        % cross
        \draw[thick, -] (0,0)--(0 + 0.25 ,0.25);
        \draw[thick, -] (0,0)--(0 - 0.25 ,0.25);
        \draw[thick, -] (0,0)--(0 + 0.25 ,-0.25);
        \draw[thick, -] (0,0)--(0 - 0.25 ,-0.25);
        \node[] at (0,-0.5) {3};
        
        % cross
        \draw[thick, -] (1.5,0)--(1.5 + 0.25 ,0.25);
        \draw[thick, -] (1.5,0)--(1.5 - 0.25 ,0.25);
        \draw[thick, -] (1.5,0)--(1.5 + 0.25 ,-0.25);
        \draw[thick, -] (1.5,0)--(1.5 - 0.25 ,-0.25);
        \node[] at (1.5,-0.5) {4};
        
        % cross
        \draw[thick, -] (3,0)--(3 + 0.25 ,0.25);
        \draw[thick, -] (3,0)--(3 - 0.25 ,0.25);
        \draw[thick, -] (3,0)--(3 + 0.25 ,-0.25);
        \draw[thick, -] (3,0)--(3 - 0.25 ,-0.25);
        \node[] at (3,-0.5) {5};
        
        % greater than 
        \draw[thick, fill=white] (4.5,0) circle [radius=0.25cm];
        \node[] at (4.5,-0.5) {6};
        
        \draw[thick, fill=white] (6,0) circle [radius=0.25cm];
        \node[] at (6,-0.5) {7};

        \draw[thick, fill=white] (7.5,0) circle [radius=0.25cm];
        \node[] at (7.5,-0.5) {8};

        % arrows
        \draw[thick, ->] (-6 ,0.4) .. controls (-6 ,1.65) and (-3 , 1.65) .. (-3 ,0.4);
        \draw[thick, ->] (-4.5 ,0.4) .. controls (-4.5 ,1.65) and (-1.5 , 1.65) .. (-1.5 ,0.4);

        \draw[thick, ->] (0 ,0.4) .. controls (0 ,1.65) and (4.5 , 1.65) .. (4.5 ,0.4);
        \draw[thick, ->] (1.5 ,0.4) .. controls (1.5 ,1.65) and (6 , 1.65) .. (6 ,0.4);
        \draw[thick, ->] (3 ,0.4) .. controls (3 ,1.65) and (7.5 , 1.65) .. (7.5 ,0.4);

        % cross
        %\draw[thick, -] (-6,0)--(-6 + 0.25 ,0.25);
        %\draw[thick, -] (-6,0)--(-6 - 0.25 ,0.25);
        %\draw[thick, -] (-6,0)--(-6 + 0.25 ,-0.25);
        %\draw[thick, -] (-6,0)--(-6 - 0.25 ,-0.25);
        
        % less than 
        %\draw[thick, -] (-3 -0.25, 0)--(-3 + 0.25, 0.25);
        %\draw[thick, -] (-3 -0.25, 0)--(-3 + 0.25, -0.25);
        
        % greater than 
        %\draw[thick, -] (-3 +0.25, 0)--(-3 - 0.25, 0.25);
        %\draw[thick, -] (-3 +0.25, 0)--(-3 - 0.25, -0.25);
        
    \end{tikzpicture}
\end{equation*}

One can see that the two parts of the weight diagram do not interact; that is to say, there are two sets of arrows that do not intersect. We can split the coordinates of $\lambda$ accordingly into two sets, $A_1:=\{4,5, -1,-2\}$ and $A_2:=\{1,2,3, -3,-4,-5\}$, so that $\overline{\lambda}_i, -\overline{\lambda}_{-j}$ are both greater than $ \overline{\lambda}_{i'},-\overline{\lambda}_{-j'}$ whenever $i, j\in A_2$, $i', j'\in A_1$.

It is easy to see (using CTD) that right odd roots of the form $\alpha = \ror{i}{j}$ where $i\in A_k$ and $-j\in A_{k'}$ for $k \neq k'$ have $c_\lambda (\alpha)=0$. 
\end{example}
Building on this observation we define a way to split $\lambda$ in to a smaller weight diagram which will be referred to as atoms of $\lambda$. 

\begin{definition}
\label{atom_def}
    Let $\lambda\in \distHW$. Consider the arrow diagram of $\lambda$ with arrow endpoints denoted by $k_1,\dots, k_m$. 
    
    Let $A\subseteq \{-n,\dots, -1, 1,\dots,m\}$, and let $I':=\min A\cap  \mathbb{Z}_{>0}$, $I'':= \max A\cap  \mathbb{Z}_{>0}$.

    Recall that for any $i\in A\cap  \mathbb{Z}_{>0}$, $[\overline{\lambda}_i, k_i] \subset [\overline{\lambda}_{I''},k_{I'}]$.
    Consider the following conditions on the set $A$:
    \begin{enumerate}
        \item We have $\bigcup_{i\in A, i>0}[\overline{\lambda}_i,k_i]=[\overline{\lambda}_{I''},k_{I'}]$ (as segments in $\mathbb{R}$). In other words, for any $r\in \mathbb{Z}$ such that $\overline{\lambda}_{I''}\leq r \leq k_{I'}$, there exists $i\in A, i>0$ such that $r\in[\overline{\lambda}_i,k_i]$ 
        (i.e. the $i$-th arrow in $D_{\lambda}$ goes over position $r$, perhaps beginning or ending at $r$).
        \item For any $-j\in A\cap \mathbb{Z}_{<0}$, $-\overline{\lambda}_{-j}\in [\overline{\lambda}_{I''},k_{I'}]$.
    \end{enumerate}
A maximal (with respect to inclusion) set $A$ satisfying the above properties is called \textbf{a $\lambda$-atom index set}. 

We denote by $\mathcal{A}_\lambda $ the set of all $\lambda$-atom index sets. 

\end{definition}

\begin{definition}
    Given a $\lambda$-atom set $A$.
    Let us write $A$ as $A=\{a_1,\dots, a_{m'},-b_1,\dots,-b_{n'}\}$ so that $a_i < a_{i+1}$, $-b_j < -b_{j+1}$. Define $\nu \in \Lambda^+_{m'|n'}$ to be given by $$\overline{\nu}_i := \lambda_{a_i}, \;\overline{\nu}_{-j}:=\lambda_{-b_j}.$$ 

    We then denote $\lambda^A:=\nu$ and call $\lambda^A$ \textbf{an atom of $\lambda$}.

\end{definition}
   \comment{ \Matan{As this is intended to appear before the equivallence classes of $\times-\circ$ sequence relation it shall remain defined like this. It may be worth adding a remark when this relation is described that atoms may be described using this relation}}
\begin{remark}\label{rmk:atom_index_sets_aer_intervals}
The set $\mathcal{A}_\lambda $ is in a natural bijection with the set of intervals as defined in Definition \ref{x o sequence}.

The correspondence is as follows: for a $\lambda$-atom index set $A$, the segment $[\overline{\lambda}_{I''}, k_{I'}]$ is an interval in the sense of Definition \ref{x o sequence}. 
Vice versa, given an interval $[p, k]$, we set $$A:=\{i, j\,| \,1\leq i\leq m, \;1\leq j\leq n, \; p\leq \overline{\lambda}_i \leq k_i\leq k, \;p\leq  -\overline{\lambda}_{-j} \leq k\}$$ and one can immediately see that $A$ is a $\lambda$-atom index set.
\end{remark}

\begin{remark}
    
   By the maximality of a $\lambda$-atom index set, any two $\lambda$-atom index sets are disjoint. 
    \end{remark}
    \begin{remark}\label{rmk:atom_index_set_conseq}
    Let $A$ be a $\lambda$-atom index set. For any $i,i'\in A \cap \{1,\dots,m\}$ where $i < i'$, any $i<a<i'$ must have $a\in A$. Similarly, if $-j,-j'\in A \cap \{-1,\dots,-n\}$ and $-j < -j'$, then any $-j<b<-j'$ must have $b\in A$.
    Thus any atom index set can be written as $$A=\{I', I'+1, I'+2, \dots, I'', -J',-J'-1,-J'-2,  \dots, -J''\}$$ for some $1\le I'\leq I''\le m$ and $1 \le J'\leq J'' \le n$.
    
    As such, $\overline{\lambda^A}=(\overline{\lambda}_{I'}\,\dots  \overline{\lambda}_{I''}\,| \overline{\lambda}_{-J'}\, \dots \, \overline{\lambda}_{-J''})$: that is, $\overline{\lambda}^A_{l+1} = \overline{\lambda}_{I'+l}$ and $\overline{\lambda}^A_{-l-1}= \overline{\lambda}_{-J'-l}$ for any $l$.

\end{remark}
\comment{
Given this definition and the above remark, we may define a map $\mathbb{S}\to \mathbb{S}_{k+1|l+1}$ given by $\Sigma \mapsto \Sigma^A$ where if the word corresponding to $\Sigma$ is given by $w$ then $\Sigma^A$ is given by the word whose relative positions of \Innas{what is the end of this sentence?}
}
The following statement follows directly from the definition of an atom of $\lambda$:
\begin{prop}\label{prop:weight_diagram_single_atom}
    Let $\lambda \in \distHW$ with $\lambda$-atom index set $A$. 
    The diagram of $\overline{\lambda^{A}}$ is given by 

    $$
    D_{\overline{\lambda^{A}}}(p) = 
    \begin{cases}
        D_{\overline{\lambda}}(p) &\text{ if } p = \overline{\lambda}_i \text{ for } i\in A \\
        \circ & \text{ otherwise}
    \end{cases}
    $$

    That is to say $D_{\overline{\lambda^{A}}}$ is given precisely by replacing all symbols that are given by $\overline{\lambda}_i$ for $i\notin A$ with $\circ$ symbols. 
\end{prop}

\begin{example}\label{ex:atoms2}
    Let $\lambda \in \Lambda^+_{5|4}$ be such that $\overline{\lambda} = (6\, 5\, 4\, \,3 \,0|\,0 \, -1\, -4\, -6)$.
    From the definition above, $\lambda$ has 2 distinct $\lambda$-atom index sets: $A= \{5,-1,-2\}$ and $B= \{1,2,3,-3,-4\}$.
    
    Thus $\lambda$ has 2 atoms, given by $\overline{\lambda^A}=(0| 0\,-1)$ and $\overline{\lambda^B}=(6\, 5\, 4| -4\, -6)$. These atoms can be seen as part of the diagram $D_{\lambda}$:
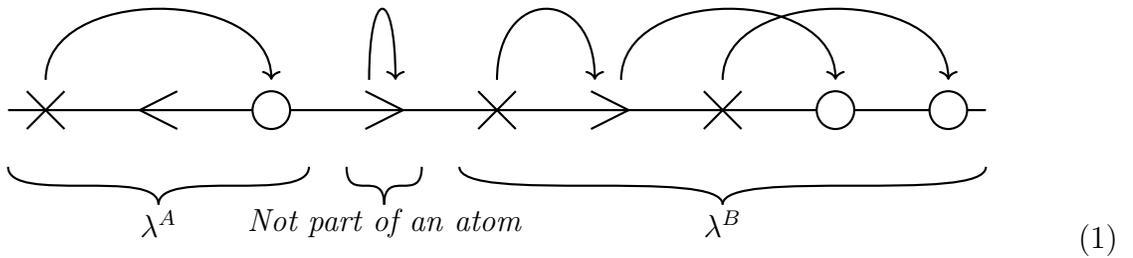
     \begin{equation}
\label{arrow atom decomp example}
    \begin{tikzpicture}
    
        \draw[thick, -] (-6.5 ,0)--(6.5, 0);
        
        % times
        \draw[thick, -] (-6,0)--(-6 + 0.25 ,0.25);
        \draw[thick, -] (-6,0)--(-6 - 0.25 ,0.25);
        \draw[thick, -] (-6,0)--(-6 + 0.25 ,-0.25);
        \draw[thick, -] (-6,0)--(-6 - 0.25 ,-0.25);

        % le
        \draw[thick, -] (-4.5 -0.25, 0)--(-4.5 + 0.25, 0.25);
        \draw[thick, -] (-4.5 -0.25, 0)--(-4.5 + 0.25, -0.25);
        
        \draw[thick, fill=white] (-3,0) circle [radius=0.25cm];
        
        % ge
        \draw[thick, -] (-1.5 +0.25, 0)--(-1.5 - 0.25, 0.25);
        \draw[thick, -] (-1.5 +0.25, 0)--(-1.5 - 0.25, -0.25);
        
        % times
        \draw[thick, -] (0,0)--(0 + 0.25 ,0.25);
        \draw[thick, -] (0,0)--(0 - 0.25 ,0.25);
        \draw[thick, -] (0,0)--(0 + 0.25 ,-0.25);
        \draw[thick, -] (0,0)--(0 - 0.25 ,-0.25);
        
        % ge
        \draw[thick, -] (1.5 +0.25, 0)--(1.5 - 0.25, 0.25);
        \draw[thick, -] (1.5 +0.25, 0)--(1.5 - 0.25, -0.25);
        
        % times
        \draw[thick, -] (3,0)--(3 + 0.25 ,0.25);
        \draw[thick, -] (3,0)--(3 - 0.25 ,0.25);
        \draw[thick, -] (3,0)--(3 + 0.25 ,-0.25);
        \draw[thick, -] (3,0)--(3 - 0.25 ,-0.25);
        
        \draw[thick, fill=white] (4.5,0) circle [radius=0.25cm];
        
        \draw[thick, fill=white] (6,0) circle [radius=0.25cm];
        
        % arrows
        \draw[thick, ->] (-6 ,0.4) .. controls (-6 ,1.65) and (-3 , 1.65) .. (-3 ,0.4);
        \draw[thick, ->] (-1.5 -0.2 ,0.4) .. controls (-1.5 -0.2 ,1.65) and (-1.5 + 0.15 , 1.65) .. (-1.5 + 0.15 ,0.4);
        \draw[thick, ->] (0 ,0.4) .. controls (0 ,1.65) and (1.5 -0.2 , 1.65) .. (1.5 -0.2 ,0.4);
        \draw[thick, ->] (1.5 +0.15 ,0.4) .. controls (1.5 +0.15 ,1.65) and (4.5 , 1.65) .. (4.5 ,0.4);
        \draw[thick, ->] (3 ,0.4) .. controls (3 , 1.65) and (6 , 1.65) .. (6 ,0.4);
        
        % brackets
        
        \draw[thick, -] (-6 -0.5,-0.75).. controls (-6 -0.5 ,-1.25) and (-4.5,-0.75) .. (-4.5,-1.25);
        \draw[thick, -] (-3 +0.5,-0.75).. controls (-3 +0.5 ,-1.25) and (-4.5,-0.75) .. (-4.5,-1.25);
        \node[] at (-4.5,-1.5) {$\lambda^A$};

        \draw[thick, -] (-1.5 -0.5,-0.75).. controls (-1.5 -0.5 ,-1.25) and (-1.5 ,-0.75) .. (-1.5 ,-1.25);
        \draw[thick, -] (-1.5 +0.5,-0.75).. controls (-1.5 +0.5 ,-1.25) and (-1.5 ,-0.75) .. (-1.5 ,-1.25);
        \node[] at (-1.5 ,-1.5) {Not part of an atom};
        
        \draw[thick, -] ( -0.5,-0.75).. controls ( -0.5 ,-1.25) and (3 ,-0.75) .. (3 ,-1.25);
        \draw[thick, -] (6 +0.5,-0.75).. controls (6 +0.5 ,-1.25) and (3 ,-0.75) .. (3 ,-1.25);
        \node[] at (3 ,-1.5) {$\lambda^B$};
        
        % cross
        %\draw[thick, -] (-6,0)--(-6 + 0.25 ,0.25);
        %\draw[thick, -] (-6,0)--(-6 - 0.25 ,0.25);
        %\draw[thick, -] (-6,0)--(-6 + 0.25 ,-0.25);
        %\draw[thick, -] (-6,0)--(-6 - 0.25 ,-0.25);
        
        % less than 
        %\draw[thick, -] (-3 -0.25, 0)--(-3 + 0.25, 0.25);
        %\draw[thick, -] (-3 -0.25, 0)--(-3 + 0.25, -0.25);
        
        % greater than 
        %\draw[thick, -] (-3 +0.25, 0)--(-3 - 0.25, 0.25);
        %\draw[thick, -] (-3 +0.25, 0)--(-3 - 0.25, -0.25);
        
    \end{tikzpicture}
\end{equation}

    One may further examine the CTD of $\lambda$ and find that the CTDs of $\lambda^A$ and $\lambda^B$ appear as sub diagrams of the CTD $c_{\lambda}$ with coordinates corresponding to $A$ and $B$ respectively. Below we show the CTD $c_\lambda$ and inside it the CTDs of the atoms of $\lambda$.
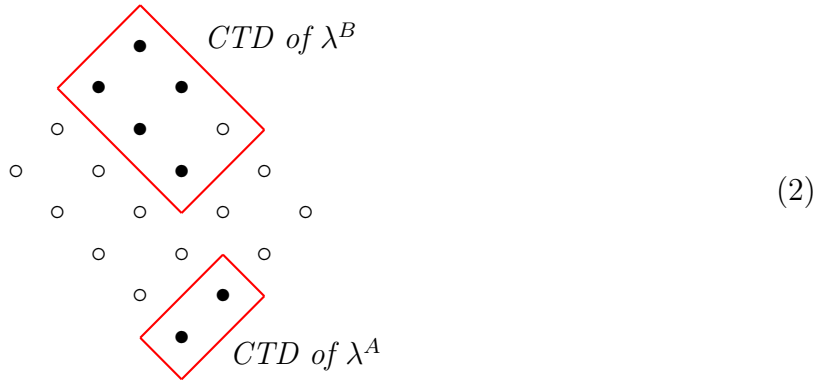
    \begin{equation}
\label{subCTD example}
\begin{tikzpicture}[anchorbase,scale=1.1]
% row 1
\node at (0,2) {$\bullet $};

% row 2
\node at (0.5,1.5) {$\bullet $};
\node at (-0.5, 1.5) {$\bullet $};

% row 3
\node at (1, 1) {$\circ $};
\node at (0,1) {$\bullet $};
\node at (-1, 1) {$\circ $};

% row 4
\node at (1.5,0.5) {$\circ $};
\node at (0.5, 0.5) {$\bullet $};
\node at (-0.5, 0.5) {$\circ $};
\node at (-1.5, 0.5) {$\circ $};

% row 5
\node at (2,0) {$\circ $};
\node at (1, 0) {$\circ $};
\node at (0, 0) {$\circ $};
\node at (-1, 0) {$\circ $};

% row 6
\node at (1.5,-0.5) {$\circ $};
\node at (0.5, -0.5) {$\circ $};
\node at (-0.5, -0.5) {$\circ $};

% row 7
\node at (1,-1) {$\bullet $};
\node at (0, -1) {$\circ $};

% row 8
\node at (0.5,-1.5) {$\bullet $};

\draw[thick, red] (0,2.5)--(-1, 1.5);
\draw[thick, red] (0.5, 0)--(-1, 1.5);
\draw[thick, red] (0.5, 0)--(1.5, 1);
\draw[thick, red] (0,2.5)--(1.5, 1);
\node[] at (1.7 ,2.1) {CTD of $\lambda^B$};

\draw[thick, red] (1,-0.5)--(0, -1.5);
\draw[thick, red] (0.5, -2)--(0, -1.5);
\draw[thick, red] (0.5, -2)--(1.5, -1);
\draw[thick, red] (1,-0.5)--(1.5, -1);
\node[] at (2 ,-1.7) {CTD of $\lambda^A$};

\end{tikzpicture}
\end{equation} 
\end{example}
Let us give the precise formulation for the ``embedding of CTDs'' mentioned in the above example.

\begin{definition}
    Let $\lambda\in \distHW$, and let $A$ be a $\lambda$-atom index set. We denote: $$\mathcal{R}^A:=\{\ror{i}{j}\,|\,i,-j\in A\}, \;\;\;\mathcal{R}_{\lambda} := \bigcup_{A\in \mathcal{A}_\lambda} \mathcal{R}^A.$$ 
\end{definition}

\begin{example}
    Expanding on Example \ref{ex:atoms2}, 
    $A=\{5,-1,-2\}$ so $$\mathcal{R}^A=\{\ror{5}{-1},\ror{5}{-2}\}.$$

    Similarly, $B=\{1,2,3,-3,-4\}$, so $$\mathcal{R}^B=\{\ror{1}{-3},\ror{2}{-3},\ror{2}{-3},\ror{1}{-4},\ror{2}{-4},\ror{3}{-4}\}.$$
    
    As those are the only two atoms of $\lambda$, 
    $$\mathcal{R}_\lambda = \mathcal{R}^A \cup \mathcal{R}^B = \{\ror{5}{-1},\ror{5}{-2}, \ror{1}{-3},\ror{2}{-3},\ror{2}{-3},\ror{1}{-4},\ror{2}{-4},\ror{3}{-4}\}.$$
\end{example}

\begin{prop}
    Let $\lambda\in \distHW$ and let $A$, $B$ be two $\lambda$-index sets. Let $a:=\min A\cap \mathbb{Z}_{>0}$ and $b:=\min B\cap \mathbb{Z}_{>0}$. If $a>b$ then for any $\alpha\in \mathcal{R}^A$, $\beta\in \mathcal{R}^B$, we have: $\alpha<\beta$, i.e. $\mathcal{R}^A < \mathcal{R}^B$.
\end{prop}
\begin{proof}
Assume $a>b$.

Then $a':=\max (A\cap \mathbb{Z}_{>0}) >b$ and by \cref{atom_def}, $b':=\max (B\cap \mathbb{Z}_{>0}) <a$. 

If $\ror{i}{j} \in\mathcal{R}^A$, $\ror{i'}{j'}\in \mathcal{R}^B$ then $$ b\leq i' \leq b'<a\leq i\leq a'.$$ Thus $i'<i$. 

By \cref{atom_def}, $-\overline{\lambda}_{-j'} \in [\overline{\lambda}_{b}, k_{b'}]$, $-\overline{\lambda}_{-j} \in [\overline{\lambda}_{a}, k_{a'}]$ and  $[\overline{\lambda}_{b}, k_{b'}] \cap [\overline{\lambda}_{a}, k_{a'}]=\emptyset$. Thus $k_{a'}<\overline{\lambda}_b$, and so $-\overline{\lambda}_{-j'}>-\overline{\lambda}_{-j} $ meaning that $\ror{i}{j}< \ror{i'}{j'}$ as required.
\end{proof}

The following two propositions explain that studying the CTD of $\lambda\in \distHW$ reduces to the study of the CTD of its atoms. This is extremely useful, making diagrams which are themselves a single atom into analogues of prime numbers.

\begin{prop}
\label{not same atom prop}
    Let $\lambda\in \distHW$. We have: $$\forall \, \alpha \in \mathcal{R}, \;\; c_{\lambda}(\alpha)= 1 \; \Longrightarrow \; \alpha \in  \mathcal{R}_{\lambda}. $$ 
    That is, if $\ror{i}{j}$ changes the diagram of $\lambda$ then $i,j$ must be part of the same $\lambda$-atom index set.
    
    \comment{
   % Then for $i\in A\cap \mathbb{Z}_{>0}$ and $1 \le j \le n$, $-j\notin A$ or $1\le i \le m$ $i\notin A$ and $-j \in A\cap \mathbb{Z}_{<0} $ implies $c_{\lambda}(\ror{i}{j})=0$.
    }
\end{prop}

\begin{proof}
   
Let $\ror{i}{j}\in \mathcal{R}$ be such that $c_{\lambda}(\ror{i}{j})=1$: by the proof of Theorem \ref{theorem for ror change inequality}, this means that $\overline{\lambda}_i\leq -\overline{\lambda}^i_{-j}<k_i$ and that $\overline{\lambda}_i <-\overline{\lambda}^{i-1}_{-j} = -\overline{\lambda}_{-j}+1 \leq k_i$. 

In other words, position $ -\overline{\lambda}^i_{-j}$ is between the endpoints of the $i$-th arrow, and so is $ -\overline{\lambda}^{i-1}_{-j}$. 

Unwinding this, we obtain that there exists $i'\geq i$ so that $\overline{\lambda}_{s}\leq k_{s+1}$ for every $s=i'-1, \ldots, i$, and $\overline{\lambda}_{i'}\leq -\overline{\lambda}_{-j}<k_i $. 

This immediately implies that $i,i', -j$ are in the same $\lambda$-atom index set, as required.
\end{proof}

\begin{prop}
\label{atom ctd prop}
    Let $\lambda\in \distHW$ and let $A$ be a $\lambda$-atom index set. 
    
    Denote the elements of $A$ by $ \{a_1,\dots, a_{m'}, -b_1,\dots, -b_{n'}\}$ so that $0<a_1<a_2<\ldots<a_{m'}$, $0<b_1<b_2<\ldots<b_{n'}$ for all $i, j$.
    Then the CTD of the atom $\lambda^A$ is given by 
    $$ c_{\lambda^A}: \{\ror{i}{j} \, |\, 1\leq i\leq m', \, 1\leq j\leq n'\} \to \{0, 1\},\;\;\;
        c_{\lambda^A}(\ror{i}{j})= c_{\lambda}(\ror{a_i}{b_j})
    $$
\end{prop}
\begin{proof}
    
    Recall that $\overline{\lambda^A}=(\overline{\lambda}_{a_1}\,\dots  \overline{\lambda}_{a_{m'}}\,| \overline{\lambda}_{-b_1}\, \dots \, \overline{\lambda}_{-b_{n'}})$, so the weight diagram of $\lambda^A$ is given by removing all the symbols whose coordinates lie outside the segment $[\overline{\lambda}_{a_{m'}}, k_{a_{1}}]$. By the definition of a $\lambda$-atom index set, if an arrow of the diagram $D_{\lambda}$ begins at, ends at or goes over any point in $[\overline{\lambda}_{a_{m'}}, k_{a_{1}}]$, then the said arrow must lie entirely in $[\overline{\lambda}_{a+m'-1}, k_a]$. Thus any such arrow of the diagram $D_{\lambda}$ remains intact as an arrow in $D_{\lambda^A}$ and any arrow in $D_{\lambda^A}$ was an arrow in $D_{\lambda}$ as well.
    
    This implies the required statement.
\end{proof}

\begin{remark}
    The corollary has interesting implication for weight diagrams. 
    
    Let $\lambda \in \distHW$ with $\mathcal{A}_{\lambda}=\{A_1,\dots, A_s\}$, so that $\max A_i\cap \mathbb{Z}_{>0} > \max A_{i+1}\cap \mathbb{Z}_{>0}$ for every $i=1, \ldots, s-1$.

Set $[d_i, d'_i]$ to be the endpoints of the segment in $\mathbb{R}$ corresponding to the $\lambda$-atom $\overline{\lambda}^{A_i}$ (that is, the smallest segment containing all the coordinates of $\overline{\lambda}^{A_i}$ and its arrow endpoints).
    
    For any $\Sigma\in \mathbb{S}$, we have: $D^{\Sigma}_\lambda(p)=D_\lambda(p)$ for $p\notin \bigcup_i A_i$, and there exists $0\leq k \leq s+1$ so that the subdiagrams $D^{\Sigma}_\lambda([d_i, d'_i])$ coincide with $D^{\distinguished}_{\lambda^{A_i}}$ for $i> k$ the subdiagrams $D^{\Sigma}_\lambda([d_i, d'_i])$ coincide with $D^{\antidist}_{\lambda^{A_i}}$ for $i< k$.

    In other words, in each $\lambda$-atom segment except (perhaps) for one, the diagram $D^{\Sigma}_\lambda$ looks either like the diagram of the corresponding atom in the distinguished base, or the anti-distinguished base.
\end{remark}

\section{Counterexample and Correction to Tail Conjecture}
\label{tail conj section}

The following section is intended as a response to \cite[Conjecture 6.3.1]{tailconjpaper}.
\subsection{Iso-sets}

\begin{definition}
We call $S\subset\Delta_1$ an \textit{iso-set} if $S$ forms a basis of an isotropic subspace in $\mathfrak{h}^*$, i.e. if $S$ is linearly independent and $(S|S)=0$. 
\end{definition}

\begin{remark}\label{rmk:iso-set_into_a_base}
      Let $S$ be a subset $S\subset \Delta_1$ such that $(\alpha|\beta)=0$ for any $\alpha, \beta\in S$. The following conditions are equivalent:
    \begin{itemize}
        \item The set $S$ is an iso-set.
        \item The set $S$ is linearly independent in $\mathfrak{h}^*$.
        \item $\forall \alpha\in S, \;\; -\alpha\notin S$.
        \item There exists a base $\Sigma$ (not necessarily in $\mathbb{S}$) such that $S\subset \Sigma$.
    \end{itemize}

\end{remark}

\begin{remark}
    Let $S$ be an iso-set. Then any $\alpha\in S$ may be replaced by $-\alpha$, and we again obtain an iso-set of the same size as $S$: indeed, every $\beta\in S$ satisfied: $(\beta|\alpha)=0$, so $(\pm \beta|\pm \alpha)=0$.
\end{remark}
    In particular, we may define:
    
    \begin{definition}
        Given an iso-set $S$, we may define the iso-set $||S|| \subset \mathcal{R}$ such that: 
    \begin{itemize}
        \item $\forall \alpha \in S$ either $\alpha \in ||S||$ or $-\alpha \in ||S||$,
        \item $\forall \beta\in ||S||$, exactly one of the roots $\pm \beta$ belongs to $S$.
    \end{itemize} 

    \end{definition}
The following lemma is obvious:
\begin{lemma}
   For any iso-set $S$, $\sharp S=\sharp||S||$. 
\end{lemma}

\begin{remark}
        In terms of the rectangular visualization of $\mathcal{R}$, an iso-set $S\subset \mathcal{R}$ is any subset of $\mathcal{R}$ where no two elements lie in the same "row" or "column" of the rectangle. 
        
        In particular, every incomparable set $S\subset \mathcal{R}$ is an iso-set.
\end{remark}

\begin{example}
    In this example we consider the superalgebra $\mathfrak{gl}(2|2)$. 
    
    Let $S_1=\{\ror{1}{1},\ror{2}{2}\}$. The set $S_1$ is an iso-set, and $S_1\subset \mathcal{R}$. Further, we may define a base $\Sigma\in \mathbb{S}$ such that $S_1\subset \Sigma$: we take $\Sigma$ to be the base given by the word $\varepsilon \delta \varepsilon \delta$. In fact, we have: $S_1= \Sigma\cap \mathcal{R}$.
     
    Iso-sets need not be subsets of $\mathcal{R}$. 
    Consider $S_2 = \{\delta_1-\varepsilon_2\}\subset \Sigma$ which is again a subset of the same base $\Sigma$ as above. This is obviously an iso-set. Notice that $\delta_1-\varepsilon_2$ is not a right odd root, as such $S_1\not\subset \mathcal{R}$.

    Finally, we may consider iso-sets that can not be embedded into any $\Sigma\in \mathbb{S}$ (although this section focuses on iso-sets that can be embedded into such bases). For example if we consider the base 
    $\Sigma = \{\varepsilon_2-\delta_1, -\varepsilon_1+\delta_1, \varepsilon_1-\delta_2\}$ (here $\Sigma = r_{\varepsilon_1-\delta_1} \circ s_{\varepsilon_1-\varepsilon_2} (\distinguished)$). Take $S_3 = \{\ror{2}{1}, \ror{1}{2}\}$. Notice that $S_3$ is an iso-set but $\ror{2}{1} < \ror{1}{2}$ as such $S_3$ by Lemma \ref{ror are incomparable} does not correspond to a base in $\mathbb{S}$.
\end{example}

Let $\nu\in\Lambda_{m|n}$. 
\begin{definition}
    We define $s(\nu)$ to be the cardinality of the largest iso-set that can be embedded into some $\Sigma\in \mathbb{S}$ and is orthogonal to $\nu$ with respect to the pairing $(-|-)$. 
\end{definition} 

It is not immediately obvious how to compute $s(\nu)$ (as one may need to consider every single base). The following lemma simplifies this problem by implying one may consider only iso-sets composed of right odd roots. 

\begin{lemma}
\label{iso set absolute value lemma}
    Let $S$ be an iso-set. Then the following are equivalent:
    \begin{enumerate}
        \item There exists a base $\Sigma \in \mathbb{S}$ such that $S\subset \Sigma$. 
        \item The set $||S||$ can be embedded into some $\Sigma'\in \mathbb{S}$.
        \item The set $||S||$ is incomparable.
        
    \end{enumerate}
\end{lemma}

\begin{proof}
    We first note that (2) and (3) are equivalent by Lemma \ref{ror are incomparable}: by this Lemma, there is a bijection $\Sigma \longmapsto \Sigma\cap \mathcal{R}$ between bases $\Sigma \in \mathbb{S}$ and sets of incomparable right odd roots.

    We now prove (1) $\implies$ (2). Assume $S\subset \Sigma \in \mathbb{S}$. 
    If $S \subset \mathcal{R}$ then $||S|| = S$ and there is nothing to prove. Otherwise, let $\delta_j - \varepsilon_i \in \Sigma$. Consider the base $r_{\delta_j - \varepsilon_i}\Sigma$ (which is of course also in $\mathbb{S}$). As $(S|S)=0$ we know that every $\alpha \in S$ such that $\delta_j - \varepsilon_i \neq \alpha$ is an element of $r_{\delta_j - \varepsilon_i}\Sigma$. Additionally, $r_{\delta_j - \varepsilon_i}(\delta_j - \varepsilon_i)=\ror{i}{j}$. Thus we obtain that $$S' = (S\setminus \{\delta_j - \varepsilon_i\})\cup \{\ror{i}{j}\}$$ is an iso-set that can be embedded into $r_{\delta_j - \varepsilon_i}\Sigma$. 
    By performing this process for every element of $S$ that is not a right odd root, we construct an iso-set $S''\subset \mathcal{R}$ such that for every $\alpha\in\mathbb{S}$ either $\alpha$ or $-\alpha$ is in $S''$. Similarly, for every $\alpha \in S''$, either $\alpha\in S$ or $-\alpha \in  S$. Thus $S''=||S||$, and we may embed it into the base $\Sigma'\in \mathbb{S}$ which is obtained by performing odd reflections on $\Sigma$ by every root of the form $\delta_j-\varepsilon_i\in S$. 

    Showing (2) $\implies$ (1) is similar. Assume $||S|| \subset \Sigma \in \mathbb{S}$. We simply perform reflections by right odd roots $\alpha \in ||S|| \setminus S$ on $\Sigma$ to find a base $\Sigma'$ that contains $S$.
\end{proof}

\begin{lemma}
\label{iso-set of ror}
    Let $\nu\in \Lambda_{m|n}$, and let $S$ be the largest cardinality iso-set such that $(S|\nu)=0$ and $S$ can be embedded into some base $\Sigma \in \mathbb{S}$. Then there exists an iso-set $S' \subset \mathcal{R}$ of right odd roots such that $\# S'=\# S$, $(S'|\nu)=0$ and $S'$ can be embedded in some $\Sigma'\in \mathbb{S}$. 
\end{lemma}

\begin{proof}

    Taking $S':=||S||$, we have, by the previous Lemma \ref{iso set absolute value lemma}: $S'\subset \mathcal{R}$ is an incomparable set, so it can be embedded into some $\Sigma'\in \mathbb{S}$ (One such $\Sigma'$ is given by: $\Sigma'\cap \mathcal{R}=S$). Furthermore, for any $\beta\in S'=||S||$, one of the roots $\pm \beta$ belongs to $S$, so $(\pm \beta|\nu)=0$. Thus $(S'| \nu)=0$ as required.

\end{proof}

By Lemma \ref{iso-set of ror} we can always assume that any maximal cardinality iso-set perpendicular to $\nu$ is a set of right odd roots:

\begin{cor}\label{cor:s_nu_second_definition}
For any $\nu\in \Lambda_{m|n}$, $s(\nu)$ is the size of the largest incomparable set in $\mathcal{R}$ which is orthogonal to $\nu$.
\end{cor}

\begin{remark}
   Recall that the atypicality of $\nu \in \Lambda_{m|n}$ is defined as the maximal cardinality of a iso-set in $\Delta$ which is orthogonal to $\nu$. Such an iso-set $S$ corresponds to an iso-set $||S||\subset\mathcal{R}$ which is again orthogonal to $\nu$, but might not be incomparable.

    This immediately implies:
    $$atyp(\nu)\geq s(\nu)$$
    and this inequality is often strict.
\end{remark}
\cref{cor:s_nu_second_definition} means we may apply the tools developed in Section \ref{b sigma section} to solve questions relating to $s(\nu)$.

In order to simplify the search for a base in $\mathbb{S}$ with the most right odd roots perpendicular to $\nu$, we may consider a tool similar to the CTD shown in Section \ref{arrow diagram section}. 

\begin{definition}\label{def:f_nu}
    Let $\nu\in \Lambda_{m|n}$. We define a map $f_\nu:\mathcal{R}\to \{0,1\}$ such that $f(\ror{i}{j})=1$ if $\nu_i=-\nu_j$ (that is, if $(\nu|\eps_i-\delta_j)=0$) and $f_\nu(\ror{i}{j})=0$ otherwise. 
\end{definition}

\begin{cor}
For any $\nu\in\Lambda_{m|n}$, 
$$s(\nu) = \max_{\Sigma \in \mathbb{S}}\sharp \{ \alpha\in\Sigma \cap \mathcal{R}|f_\nu(\alpha)=1\}.
$$
\end{cor}

This function $f_{\nu}$ can be visualized by drawing a diagram similar to the CTD. The right odd roots $\nu$ such that $f(\nu)=1$ are marked with a $\star$.

\begin{example}
    Consider the weight $\nu=(5\,5\,4\,1\,0|0\,-1\,-3\,-5\,-5)$ for $\mathfrak{gl}(5|5)$. Note that $\nu=\overline{\lambda}_\Sigma$ for $\lambda\in \distHW$ with $\overline{\lambda}=(5\, 4\, 2\, -1\, -2| 2\, 1\, -2\, -3\, -5)$ and $\Sigma$ given by $\eps \delta \eps \delta^4 \eps^3$. The diagram for $f_\nu$ is given by 

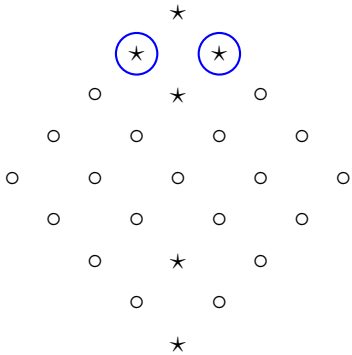
\begin{equation*}
\label{f nu}
\begin{tikzpicture}[anchorbase,scale=1.1]
% row 1
\node at (0,1.5) {$\star$};

% row 2
\node at (0.5,1) {$\star $};
\draw[thick, blue] (0.5,1) circle [radius=0.25cm];
\node at (-0.5, 1) {$\star $};
\draw[thick, blue] (-0.5,1) circle [radius=0.25cm];

% row 3
\node at (1, 0.5) {$\circ $};
\node at (0, 0.5) {$\star $};
\node at (-1, 0.5) {$\circ $};

% row 4
\node at (1.5, 0) {$\circ $};
\node at (0.5, 0) {$\circ $};
\node at (-0.5, 0) {$\circ $};
\node at (-1.5, 0) {$\circ $};

% row 5
\node at (2, -0.5) {$\circ $};
\node at (1, -0.5) {$\circ $};
\node at (0, -0.5) {$\circ $};
\node at (-1, -0.5) {$\circ $};
\node at (-2, -0.5) {$\circ $};

% row 6
\node at (1.5, -1) {$\circ $};
\node at (0.5, -1) {$\circ $};
\node at (-0.5, -1) {$\circ $};
\node at (-1.5, -1) {$\circ $};

% row 7
\node at (1, -1.5) {$\circ $};
\node at (0, -1.5) {$\star $};
\node at (-1, -1.5) {$\circ $};

% row 8
\node at (0.5, -2) {$\circ $};
\node at (-0.5, -2) {$\circ $};

% row 9
\node at (0, -2.5) {$\star $};

%\draw[-] (1.5,0.5)--(1.5,-1);

\end{tikzpicture}
\end{equation*} 

The set $S:=\{\ror{1}{4},\ror{2}{5}\}$ (circled in blue in diagram above) of right odd roots is clearly incomparable, and thus defines an iso-set which is orthogonal to $\nu$ and can be embedded into a base in $\mathbb{S}$. It is easy to see that there is no larger incomparable set consisting of odd roots marked by $\star$ (that is, orthogonal to $\nu$), so $s(\nu)=2$. Notice that $atyp(\lambda)=4> s(\nu)$ Since we may take an iso-set $S'=S \cup \{\ror{5}{1},\ror{4}{2}\}$ which is orthogonal to $\nu$ but cannot be embedded into any $\Sigma\in \mathbb{S}$.

\end{example}

\subsection{Longtail}
\mbox{}

\begin{definition}\label{def:longtail}
    Let $\lambda\in \distHW$. Denote:
    $$longtail(\lambda) := \max_{\nu\in Hwt(\lambda)} s(\nu) $$
\end{definition}

We now give our first bound on $longtail(\lambda)$.

Let $\lambda\in \distHW$, and consider the maximal number of $\times$ symbols stacked on top of one another in a single position in $D_{\lambda}^{\Sigma}$ when $\Sigma$ runs over all the bases in $\mathbb{S}$. The following lemma shows that this number is at most $longtail(\lambda)$.

\begin{lemma}\label{lem:longtail_ineq_number_x_stacked}
    Let $\lambda\in \distHW$, and let $\Sigma\in \mathbb{S}$. Consider the weight diagram $D_{\lambda}^{\Sigma}$ and denote by $k$ the maximal number of $\times$ symbols stacked on top of one another in a single position in $D_{\lambda}^{\Sigma}$. Then $$s(\overline{\lambda}_{\Sigma}) \geq k,$$
    and thus $longtail(\lambda)\geq k$.
\end{lemma}
\begin{proof}
   Let $p\in \mathbb{Z}$ such that $D^\Sigma _\lambda(p)\in \{\times^k, \times^k >, \times^k <\}$ (which is to say $p$ is the position with the most $\times$ symbols stack on top of each other). 

    Then there exist indices $1\le i_1 < \dots < i_k \le m$ and $1\le j_1 < \dots < j_k \le n$ such that $$-(\overline{\lambda}_\Sigma)_{-j_l}=(\overline{\lambda}_\Sigma)_{i_l} = p$$ for $1\le l \le k$.  

    Define $S = \{\ror{i_1}{j_1},\dots, \ror{i_k}{j_k}\}$. This is an incomparable set, since for $l< l'$ for $i_{l} < i_{l'}$ and $j_{l} < j_{l'}$, the right odd roots $\ror{i_l}{j_l}$ and $\ror{i_{l'}}{j_{l'}}$ are incomparable. Moreover, any $\ror{i_l}{j_l}\in S$ is orthogonal to $\overline{\lambda}_{\Sigma}$.

    Thus $S$ is an iso-set orthogonal to $\overline{\lambda}_{\Sigma}$ and so $$s(\overline{\lambda}_\Sigma) \ge \sharp S = k.$$
\end{proof}

In \cite[6.3]{tailconjpaper}, it is stated that that the above is an equality. This is not the case, as can be seen in the following example:

\begin{example}
    To show that the equality need not hold, consider $\lambda\in \distHW$ for $\fgl{4}{4}$ given by $\overline{\lambda} = (4\,3\,1\,0|0\, -1\,-3\,-4)$. Let $\Sigma\in \mathbb{S}$ be the base given by the word $\varepsilon \delta^4 \varepsilon^3$. Then $$\overline{\lambda}_\Sigma = (4\, 6\, 5\, 2|-2\, -4\,-5 \, -6).$$ Then $D^{\Sigma}_\lambda$ is given by  
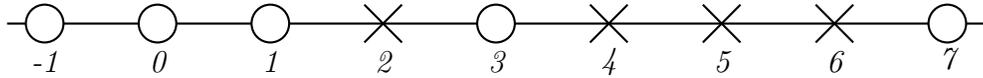
\begin{equation*}
    \begin{tikzpicture}
    
        \draw[thick, -] (-6.5 ,0)--(6.5, 0);
        
        \draw[thick, fill=white] (-6,0) circle [radius=0.25cm];
        \node[] at (-6,-0.5) {-1};
        
        \draw[thick, fill=white] (-4.5,0) circle [radius=0.25cm];
        \node[] at (-4.5,-0.5) {0};
        
        \draw[thick, fill=white] (-3,0) circle [radius=0.25cm];
        \node[] at (-3,-0.5) {1};

        % times
        \draw[thick, -] (-1.5,0)--(-1.5 + 0.25 ,0.25);
        \draw[thick, -] (-1.5,0)--(-1.5 - 0.25 ,0.25);
        \draw[thick, -] (-1.5,0)--(-1.5 + 0.25 ,-0.25);
        \draw[thick, -] (-1.5,0)--(-1.5 - 0.25 ,-0.25);
        \node[] at (-1.5,-0.5) {2};
        
        \draw[thick, fill=white] (0,0) circle [radius=0.25cm];
        \node[] at (0,-0.5) {3};
        
        % times
        \draw[thick, -] (1.5,0)--(1.5 + 0.25 ,0.25);
        \draw[thick, -] (1.5,0)--(1.5 - 0.25 ,0.25);
        \draw[thick, -] (1.5,0)--(1.5 + 0.25 ,-0.25);
        \draw[thick, -] (1.5,0)--(1.5 - 0.25 ,-0.25);
        \node[] at (1.5,-0.5) {4};
        
        % times
        \draw[thick, -] (3,0)--(3 + 0.25 ,0.25);
        \draw[thick, -] (3,0)--(3 - 0.25 ,0.25);
        \draw[thick, -] (3,0)--(3 + 0.25 ,-0.25);
        \draw[thick, -] (3,0)--(3 - 0.25 ,-0.25);
        \node[] at (3,-0.5) {5};
        
        % times
        \draw[thick, -] (4.5,0)--(4.5 + 0.25 ,0.25);
        \draw[thick, -] (4.5,0)--(4.5 - 0.25 ,0.25);
        \draw[thick, -] (4.5,0)--(4.5 + 0.25 ,-0.25);
        \draw[thick, -] (4.5,0)--(4.5 - 0.25 ,-0.25);
        \node[] at (4.5,-0.5) {6};
        
        \draw[thick, fill=white] (6,0) circle [radius=0.25cm];
        \node[] at (6,-0.5) {7};
        
        % cross
        %\draw[thick, -] (-6,0)--(-6 + 0.25 ,0.25);
        %\draw[thick, -] (-6,0)--(-6 - 0.25 ,0.25);
        %\draw[thick, -] (-6,0)--(-6 + 0.25 ,-0.25);
        %\draw[thick, -] (-6,0)--(-6 - 0.25 ,-0.25);
        
        % less than 
        %\draw[thick, -] (-3 -0.25, 0)--(-3 + 0.25, 0.25);
        %\draw[thick, -] (-3 -0.25, 0)--(-3 + 0.25, -0.25);
        
        % greater than 
        %\draw[thick, -] (-3 +0.25, 0)--(-3 - 0.25, 0.25);
        %\draw[thick, -] (-3 +0.25, 0)--(-3 - 0.25, -0.25);
        
    \end{tikzpicture}
\end{equation*}
  and $f_{\overline{\lambda}_\Sigma}$ is given by the diagram 
    
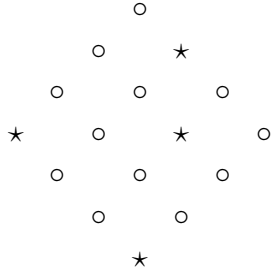
\begin{equation*}
    \begin{tikzpicture}[anchorbase,scale=1.1]
% row 1
\node at (0,1.5) {$\circ$};

% row 2
\node at (0.5,1) {$\star $};
\node at (-0.5, 1) {$\circ $};

% row 3
\node at (1, 0.5) {$\circ $};
\node at (0, 0.5) {$\circ $};
\node at (-1, 0.5) {$\circ $};

% row 4
\node at (1.5, 0) {$\circ $};
\node at (0.5, 0) {$\star $};
\node at (-0.5, 0) {$\circ $};
\node at (-1.5, 0) {$\star $};

% row 5
\node at (1, -0.5) {$\circ $};
\node at (0, -0.5) {$\circ $};
\node at (-1, -0.5) {$\circ $};

% row 6
\node at (0.5, -1) {$\circ $};
\node at (-0.5, -1) {$\circ $};

% row 7
\node at (0, -1.5) {$\star $};

%\draw[-] (1.5,0.5)--(1.5,-1);

\end{tikzpicture}
\end{equation*} 

Notice that the most $\times$ symbols in any given position of $D^{\Sigma}_\lambda$ is $1$. On the other hand, the iso-set $\Sigma \cap \mathcal{R}=\{\ror{1}{1},\ror{2}{4}\}$ is orthogonal to $\overline{\lambda}_\Sigma$ and $s(\overline{\lambda}_{\Sigma})\geq \sharp \Sigma \cap \mathcal{R} = 2$.
\end{example}

\subsection{The Tail Conjecture}

\begin{definition}
\label{dagger space}
We denote by $\Lambda^\dagger$ the set of weight diagrams with the following properties: 
\begin{itemize} 
    
    \item At most one position in the diagram contains more than one of the symbols $>,<,\times$. If such a position exists, it is given by $\times^k$ for some $k\ge 2$, and there are no $\times$ symbols to the left of this position.
    
    \item The symbol $\circ$ appears between the two leftmost positions containing symbols $\times$.
\end{itemize}
\end{definition}

\begin{remark}
    Notice that for every $D\in \Lambda^\dagger$ there exist $\lambda \in \distHW$ and $\Sigma \in \mathbb{S}$ such that $D= D^\Sigma _\lambda$.
\end{remark}
To clarify, here are some examples of diagrams contained in $\Lambda^\dagger$

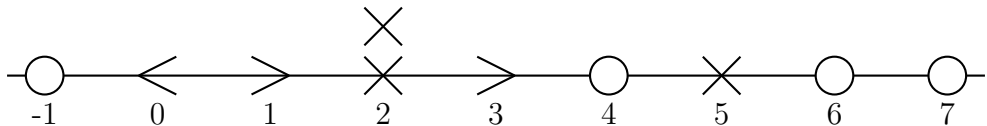
    \begin{equation*}
    \begin{tikzpicture}
    
        \draw[thick, -] (-6.5 ,0)--(6.5, 0);
        
        \draw[thick, fill=white] (-6,0) circle [radius=0.25cm];
        \node[] at (-6,-0.5) {-1};
        
        % le
        \draw[thick, -] (-4.5 -0.25, 0)--(-4.5 + 0.25, 0.25);
        \draw[thick, -] (-4.5 -0.25, 0)--(-4.5 + 0.25, -0.25);
        \node[] at (-4.5,-0.5) {0};
        
        % ge
        \draw[thick, -] (-3 +0.25, 0)--(-3 - 0.25, 0.25);
        \draw[thick, -] (-3 +0.25, 0)--(-3 - 0.25, -0.25);
        \node[] at (-3,-0.5) {1};
        
        % times
        \draw[thick, -] (-1.5,0)--(-1.5 + 0.25 ,0.25);
        \draw[thick, -] (-1.5,0)--(-1.5 - 0.25 ,0.25);
        \draw[thick, -] (-1.5,0)--(-1.5 + 0.25 ,-0.25);
        \draw[thick, -] (-1.5,0)--(-1.5 - 0.25 ,-0.25);
        \node[] at (-1.5,-0.5) {2};

        % stacked times
        \draw[thick, -] (-1.5, 0 + 0.65)--(-1.5 + 0.25 ,0.25 + 0.65);
        \draw[thick, -] (-1.5, 0 + 0.65)--(-1.5 - 0.25 ,0.25 + 0.65);
        \draw[thick, -] (-1.5, 0 + 0.65)--(-1.5 + 0.25 ,-0.25 + 0.65);
        \draw[thick, -] (-1.5, 0 + 0.65)--(-1.5 - 0.25 ,-0.25 + 0.65);
        
        % ge
        \draw[thick, -] (0 +0.25, 0)--(0 - 0.25, 0.25);
        \draw[thick, -] (0 +0.25, 0)--(0 - 0.25, -0.25);
        \node[] at (0,-0.5) {3};
        
        \draw[thick, fill=white] (1.5,0) circle [radius=0.25cm];
        \node[] at (1.5,-0.5) {4};
        
        % times
        \draw[thick, -] (3,0)--(3 + 0.25 ,0.25);
        \draw[thick, -] (3,0)--(3 - 0.25 ,0.25);
        \draw[thick, -] (3,0)--(3 + 0.25 ,-0.25);
        \draw[thick, -] (3,0)--(3 - 0.25 ,-0.25);
        \node[] at (3,-0.5) {5};

        \draw[thick, fill=white] (4.5,0) circle [radius=0.25cm];
        \node[] at (4.5,-0.5) {6};
        
        \draw[thick, fill=white] (6,0) circle [radius=0.25cm];
        \node[] at (6,-0.5) {7};
        
        % cross
        %\draw[thick, -] (-6,0)--(-6 + 0.25 ,0.25);
        %\draw[thick, -] (-6,0)--(-6 - 0.25 ,0.25);
        %\draw[thick, -] (-6,0)--(-6 + 0.25 ,-0.25);
        %\draw[thick, -] (-6,0)--(-6 - 0.25 ,-0.25);
        
        % less than 
        %\draw[thick, -] (-3 -0.25, 0)--(-3 + 0.25, 0.25);
        %\draw[thick, -] (-3 -0.25, 0)--(-3 + 0.25, -0.25);
        
        % greater than 
        %\draw[thick, -] (-3 +0.25, 0)--(-3 - 0.25, 0.25);
        %\draw[thick, -] (-3 +0.25, 0)--(-3 - 0.25, -0.25);
        
    \end{tikzpicture}
\end{equation*}

    
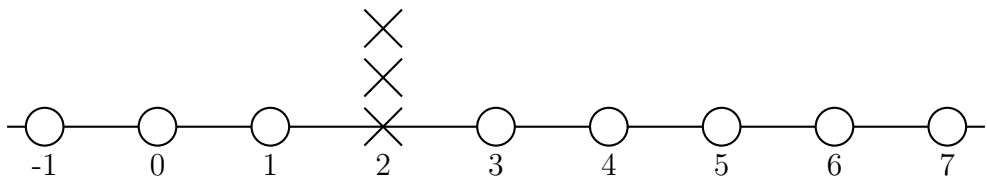
\begin{equation*}
    \begin{tikzpicture}
    
        \draw[thick, -] (-6.5 ,0)--(6.5, 0);
        
        \draw[thick, fill=white] (-6,0) circle [radius=0.25cm];
        \node[] at (-6,-0.5) {-1};
        
        \draw[thick, fill=white] (-4.5,0) circle [radius=0.25cm];
        \node[] at (-4.5,-0.5) {0};
        
        \draw[thick, fill=white] (-3,0) circle [radius=0.25cm];
        \node[] at (-3,-0.5) {1};
        
        % times
        \draw[thick, -] (-1.5,0)--(-1.5 + 0.25 ,0.25);
        \draw[thick, -] (-1.5,0)--(-1.5 - 0.25 ,0.25);
        \draw[thick, -] (-1.5,0)--(-1.5 + 0.25 ,-0.25);
        \draw[thick, -] (-1.5,0)--(-1.5 - 0.25 ,-0.25);
        \node[] at (-1.5,-0.5) {2};

        % stacked times
        \draw[thick, -] (-1.5, 0 + 0.65)--(-1.5 + 0.25 ,0.25 + 0.65);
        \draw[thick, -] (-1.5, 0 + 0.65)--(-1.5 - 0.25 ,0.25 + 0.65);
        \draw[thick, -] (-1.5, 0 + 0.65)--(-1.5 + 0.25 ,-0.25 + 0.65);
        \draw[thick, -] (-1.5, 0 + 0.65)--(-1.5 - 0.25 ,-0.25 + 0.65);
        
        % stacked times
        \draw[thick, -] (-1.5, 0 + 2*0.65)--(-1.5 + 0.25 ,0.25 + 2*0.65);
        \draw[thick, -] (-1.5, 0 + 2*0.65)--(-1.5 - 0.25 ,0.25 + 2*0.65);
        \draw[thick, -] (-1.5, 0 + 2*0.65)--(-1.5 + 0.25 ,-0.25 + 2*0.65);
        \draw[thick, -] (-1.5, 0 + 2*0.65)--(-1.5 - 0.25 ,-0.25 + 2*0.65);
        
        \draw[thick, fill=white] (0,0) circle [radius=0.25cm];
        \node[] at (0,-0.5) {3};
        
        \draw[thick, fill=white] (1.5,0) circle [radius=0.25cm];
        \node[] at (1.5,-0.5) {4};
        
        \draw[thick, fill=white] (3,0) circle [radius=0.25cm];
        \node[] at (3,-0.5) {5};

        \draw[thick, fill=white] (4.5,0) circle [radius=0.25cm];
        \node[] at (4.5,-0.5) {6};
        
        \draw[thick, fill=white] (6,0) circle [radius=0.25cm];
        \node[] at (6,-0.5) {7};
        
        % cross
        %\draw[thick, -] (-6,0)--(-6 + 0.25 ,0.25);
        %\draw[thick, -] (-6,0)--(-6 - 0.25 ,0.25);
        %\draw[thick, -] (-6,0)--(-6 + 0.25 ,-0.25);
        %\draw[thick, -] (-6,0)--(-6 - 0.25 ,-0.25);
        
        % less than 
        %\draw[thick, -] (-3 -0.25, 0)--(-3 + 0.25, 0.25);
        %\draw[thick, -] (-3 -0.25, 0)--(-3 + 0.25, -0.25);
        
        % greater than 
        %\draw[thick, -] (-3 +0.25, 0)--(-3 - 0.25, 0.25);
        %\draw[thick, -] (-3 +0.25, 0)--(-3 - 0.25, -0.25);
        
    \end{tikzpicture}
\end{equation*}

And some examples of diagrams not contained in $\Lambda^\dagger$ 

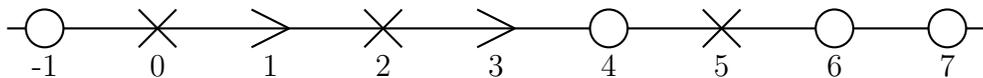
\begin{equation*}
    \begin{tikzpicture}
    
        \draw[thick, -] (-6.5 ,0)--(6.5, 0);
        
        \draw[thick, fill=white] (-6,0) circle [radius=0.25cm];
        \node[] at (-6,-0.5) {-1};
        
        % times
        \draw[thick, -] (-4.5,0)--(-4.5 + 0.25 ,0.25);
        \draw[thick, -] (-4.5,0)--(-4.5 - 0.25 ,0.25);
        \draw[thick, -] (-4.5,0)--(-4.5 + 0.25 ,-0.25);
        \draw[thick, -] (-4.5,0)--(-4.5 - 0.25 ,-0.25);
        \node[] at (-4.5,-0.5) {0};
        
        % ge
        \draw[thick, -] (-3 +0.25, 0)--(-3 - 0.25, 0.25);
        \draw[thick, -] (-3 +0.25, 0)--(-3 - 0.25, -0.25);
        \node[] at (-3,-0.5) {1};
        
        % times
        \draw[thick, -] (-1.5,0)--(-1.5 + 0.25 ,0.25);
        \draw[thick, -] (-1.5,0)--(-1.5 - 0.25 ,0.25);
        \draw[thick, -] (-1.5,0)--(-1.5 + 0.25 ,-0.25);
        \draw[thick, -] (-1.5,0)--(-1.5 - 0.25 ,-0.25);
        \node[] at (-1.5,-0.5) {2};
        
        % ge
        \draw[thick, -] (0 +0.25, 0)--(0 - 0.25, 0.25);
        \draw[thick, -] (0 +0.25, 0)--(0 - 0.25, -0.25);
        \node[] at (0,-0.5) {3};
        
        \draw[thick, fill=white] (1.5,0) circle [radius=0.25cm];
        \node[] at (1.5,-0.5) {4};
        
        % times
        \draw[thick, -] (3,0)--(3 + 0.25 ,0.25);
        \draw[thick, -] (3,0)--(3 - 0.25 ,0.25);
        \draw[thick, -] (3,0)--(3 + 0.25 ,-0.25);
        \draw[thick, -] (3,0)--(3 - 0.25 ,-0.25);
        \node[] at (3,-0.5) {5};

        \draw[thick, fill=white] (4.5,0) circle [radius=0.25cm];
        \node[] at (4.5,-0.5) {6};
        
        \draw[thick, fill=white] (6,0) circle [radius=0.25cm];
        \node[] at (6,-0.5) {7};
        
        % cross
        %\draw[thick, -] (-6,0)--(-6 + 0.25 ,0.25);
        %\draw[thick, -] (-6,0)--(-6 - 0.25 ,0.25);
        %\draw[thick, -] (-6,0)--(-6 + 0.25 ,-0.25);
        %\draw[thick, -] (-6,0)--(-6 - 0.25 ,-0.25);
        
        % less than 
        %\draw[thick, -] (-3 -0.25, 0)--(-3 + 0.25, 0.25);
        %\draw[thick, -] (-3 -0.25, 0)--(-3 + 0.25, -0.25);
        
        % greater than 
        %\draw[thick, -] (-3 +0.25, 0)--(-3 - 0.25, 0.25);
        %\draw[thick, -] (-3 +0.25, 0)--(-3 - 0.25, -0.25);
        
    \end{tikzpicture}
\end{equation*}
(In the diagram above the two leftmost positions containing a $\times$ do not have a $\circ$ between them)

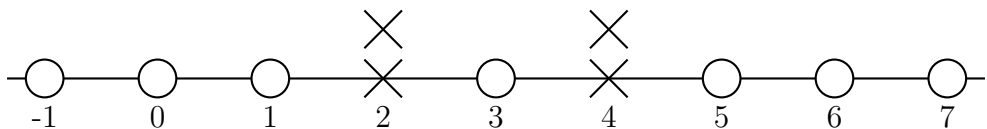
\begin{equation*}
    \begin{tikzpicture}
    
        \draw[thick, -] (-6.5 ,0)--(6.5, 0);
        
        \draw[thick, fill=white] (-6,0) circle [radius=0.25cm];
        \node[] at (-6,-0.5) {-1};
        
        \draw[thick, fill=white] (-4.5,0) circle [radius=0.25cm];
        \node[] at (-4.5,-0.5) {0};
        
        \draw[thick, fill=white] (-3,0) circle [radius=0.25cm];
        \node[] at (-3,-0.5) {1};
        
        % times
        \draw[thick, -] (-1.5,0)--(-1.5 + 0.25 ,0.25);
        \draw[thick, -] (-1.5,0)--(-1.5 - 0.25 ,0.25);
        \draw[thick, -] (-1.5,0)--(-1.5 + 0.25 ,-0.25);
        \draw[thick, -] (-1.5,0)--(-1.5 - 0.25 ,-0.25);
        \node[] at (-1.5,-0.5) {2};

        % stacked times
        \draw[thick, -] (-1.5, 0 + 0.65)--(-1.5 + 0.25 ,0.25 + 0.65);
        \draw[thick, -] (-1.5, 0 + 0.65)--(-1.5 - 0.25 ,0.25 + 0.65);
        \draw[thick, -] (-1.5, 0 + 0.65)--(-1.5 + 0.25 ,-0.25 + 0.65);
        \draw[thick, -] (-1.5, 0 + 0.65)--(-1.5 - 0.25 ,-0.25 + 0.65);

        \draw[thick, fill=white] (0,0) circle [radius=0.25cm];
        \node[] at (0,-0.5) {3};
        
        % times
        \draw[thick, -] (1.5,0)--(1.5 + 0.25 ,0.25);
        \draw[thick, -] (1.5,0)--(1.5 - 0.25 ,0.25);
        \draw[thick, -] (1.5,0)--(1.5 + 0.25 ,-0.25);
        \draw[thick, -] (1.5,0)--(1.5 - 0.25 ,-0.25);
        \node[] at (1.5,-0.5) {4};

        % stacked times
        \draw[thick, -] (1.5, 0 + 0.65)--(1.5 + 0.25 ,0.25 + 0.65);
        \draw[thick, -] (1.5, 0 + 0.65)--(1.5 - 0.25 ,0.25 + 0.65);
        \draw[thick, -] (1.5, 0 + 0.65)--(1.5 + 0.25 ,-0.25 + 0.65);
        \draw[thick, -] (1.5, 0 + 0.65)--(1.5 - 0.25 ,-0.25 + 0.65);

        \draw[thick, fill=white] (3,0) circle [radius=0.25cm];
        \node[] at (3,-0.5) {5};

        \draw[thick, fill=white] (4.5,0) circle [radius=0.25cm];
        \node[] at (4.5,-0.5) {6};
        
        \draw[thick, fill=white] (6,0) circle [radius=0.25cm];
        \node[] at (6,-0.5) {7};
        
        % cross
        %\draw[thick, -] (-6,0)--(-6 + 0.25 ,0.25);
        %\draw[thick, -] (-6,0)--(-6 - 0.25 ,0.25);
        %\draw[thick, -] (-6,0)--(-6 + 0.25 ,-0.25);
        %\draw[thick, -] (-6,0)--(-6 - 0.25 ,-0.25);
        
        % less than 
        %\draw[thick, -] (-3 -0.25, 0)--(-3 + 0.25, 0.25);
        %\draw[thick, -] (-3 -0.25, 0)--(-3 + 0.25, -0.25);
        
        % greater than 
        %\draw[thick, -] (-3 +0.25, 0)--(-3 - 0.25, 0.25);
        %\draw[thick, -] (-3 +0.25, 0)--(-3 - 0.25, -0.25);
        
    \end{tikzpicture}
\end{equation*}
(In the diagram above there is more than one position with more than one $\times$ symbol)

There exists a natural map $\Psi:\Lambda^\dagger \to \Lambda^{+}=\bigsqcup_{m, n\geq 0} \distHW $ (see  \cite[6.1.1]{tailconjpaper}).

Let us describe this map. Given $D\in \Lambda^\dagger$, we have at most one position with exactly $i$ symbols $\times$ where $i>1$. 

If such a position does not exist, then $D = D^{\distinguished}_{\lambda}$ for some $\lambda \in \distHW$, and we set $\lambda:=\Psi(D)$. 

If such a position $p$ in $D$ exists, let $D'$ be the diagram obtained from $D$ by moving $i-1$ symbols $\times$ from position $p$ to the previous $i-1$ empty positions in $D$ (to the left of $p$). We remind that an empty position is a position with the symbol $\circ$ inside. We then obtain a diagram $D'$ which is $D^{\distinguished}_{\lambda}$ for some $\lambda \in \distHW$, and we set $\lambda:=\Psi(D)$. 

We may also define an inverse map $\Phi: \Lambda^{+}=\bigsqcup_{m, n\geq 0} \distHW  \to  \Lambda^\dagger$. 

Given $\lambda\in \distHW$, let $D_{\overline{\lambda}}$ be its weight diagram with respect to the distinguished base.
Recall that since $\lambda\in\distHW$, every position $p$ has $D_{\lambda}(p)\in \{\times, \circ, >, <\}$ (no symbols are stacked). 

If $D_{\lambda}$ has no $\times$ symbols (i.e. $atyp(\lambda)=0$), then set $\Phi(\lambda):=D_\lambda$. 

Otherwise, let $\times_1$ be the position of the leftmost $\times$ symbol of the diagram $D_{\lambda}$.

Denote by $$d := \min \{p\in\mathbb{Z}|D_{\lambda}(p)=\circ \text{ and } p > \times_1\}$$ the first position to the right of $\times_1$ which has a $\circ$ symbol, and denote by $$s := \sharp \{p\in \{\times_1,\dots, d-1\}|D_{\lambda}(p)=\times\}$$ the number of $\times$ symbols between the positions $\times_1$ and $d$ (including the symbol in position $\times_1$). Next, we denote by
$$\times_\dagger := \max \{p\in \mathbb{Z}|D_{\overline{\lambda}}(p)=\times \text{ and } \times_1 \le p < d\}$$ the position of the rightmost $\times$ symbol appearing before $d$ and after $\times_1 - 1$ (note that we may have $\times_1=\times_\dagger$).

We then define the weight diagram $D'$
such that 
$$
D'(p):=
\begin{cases}
    \times^s & p = \times_\dagger \\
    \circ & D_{\overline{\lambda}}(p)=\times \text{ and } p\neq \times_\dagger \\
    D_{\lambda}(p) & \text{otherwise}
\end{cases}
$$
and clearly $D' \in \Lambda^{\dagger}$. We set: $D':=\Phi(\lambda)$.

The following lemma is straightforward:
\begin{lemma}
The maps $\Phi: \Lambda^{+}=\bigsqcup_{m, n\geq 0} \distHW  \rightleftarrows \Lambda^\dagger:\Psi$ are mutually inverse and define a bijection between the sets $\Lambda^+, \Lambda^{\dagger}$.
\end{lemma}

Here is an example of this correspondence. 
\begin{example}
Consider the following diagram $D$ in $\Lambda^{\dagger}$:

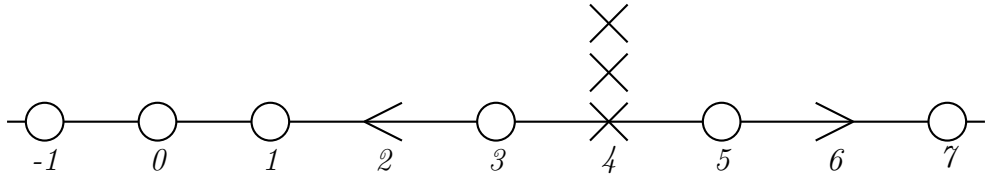
\begin{equation*}
    \label{first dagger}
    \begin{tikzpicture}
    
        \draw[thick, -] (-6.5 ,0)--(6.5, 0);
        
        \draw[thick, fill=white] (-6,0) circle [radius=0.25cm];
        \node[] at (-6,-0.5) {-1};
        
        \draw[thick, fill=white] (-4.5,0) circle [radius=0.25cm];
        \node[] at (-4.5,-0.5) {0};
        
        \draw[thick, fill=white] (-3,0) circle [radius=0.25cm];
        \node[] at (-3,-0.5) {1};
        
        % less than 
        \draw[thick, -] (-1.5 -0.25, 0)--(-1.5 + 0.25, 0.25);
        \draw[thick, -] (-1.5 -0.25, 0)--(-1.5 + 0.25, -0.25);
        \node[] at (-1.5,-0.5) {2};

        \draw[thick, fill=white] (0,0) circle [radius=0.25cm];
        \node[] at (0,-0.5) {3};
        
        % cross
        \draw[thick, -] (1.5,0)--(1.5 + 0.25 ,0.25);
        \draw[thick, -] (1.5,0)--(1.5 - 0.25 ,0.25);
        \draw[thick, -] (1.5,0)--(1.5 + 0.25 ,-0.25);
        \draw[thick, -] (1.5,0)--(1.5 - 0.25 ,-0.25);

        % stacked cross
        \draw[thick, -] (1.5,0 + 0.65)--(1.5 + 0.25 ,0.25 + 0.65);
        \draw[thick, -] (1.5,0 + 0.65)--(1.5 - 0.25 ,0.25 + 0.65);
        \draw[thick, -] (1.5,0 + 0.65)--(1.5 + 0.25 ,-0.25 + 0.65);
        \draw[thick, -] (1.5,0 + 0.65)--(1.5 - 0.25 ,-0.25 + 0.65);
        % stacked cross
        \draw[thick, -] (1.5,0 + 2* 0.65)--(1.5 + 0.25 ,0.25 + 2* 0.65);
        \draw[thick, -] (1.5,0 + 2* 0.65)--(1.5 - 0.25 ,0.25 + 2* 0.65);
        \draw[thick, -] (1.5,0 + 2* 0.65)--(1.5 + 0.25 ,-0.25 + 2* 0.65);
        \draw[thick, -] (1.5,0 + 2* 0.65)--(1.5 - 0.25 ,-0.25 + 2* 0.65);
        \node[] at (1.5,-0.5) {4};
        
        \draw[thick, fill=white] (3,0) circle [radius=0.25cm];
        \node[] at (3,-0.5) {5};
        
        % greater than 
        \draw[thick, -] (4.5 +0.25, 0)--(4.5 - 0.25, 0.25);
        \draw[thick, -] (4.5 +0.25, 0)--(4.5 - 0.25, -0.25);
        \node[] at (4.5,-0.5) {6};
        
        \draw[thick, fill=white] (6,0) circle [radius=0.25cm];
        \node[] at (6,-0.5) {7};
        
        % cross
        %\draw[thick, -] (-6,0)--(-6 + 0.25 ,0.25);
        %\draw[thick, -] (-6,0)--(-6 - 0.25 ,0.25);
        %\draw[thick, -] (-6,0)--(-6 + 0.25 ,-0.25);
        %\draw[thick, -] (-6,0)--(-6 - 0.25 ,-0.25);
        
        % less than 
        %\draw[thick, -] (-3 -0.25, 0)--(-3 + 0.25, 0.25);
        %\draw[thick, -] (-3 -0.25, 0)--(-3 + 0.25, -0.25);
        
        % greater than 
        %\draw[thick, -] (-3 +0.25, 0)--(-3 - 0.25, 0.25);
        %\draw[thick, -] (-3 +0.25, 0)--(-3 - 0.25, -0.25);
        
    \end{tikzpicture}
\end{equation*}

 Then $\Psi(D) = \lambda$ for $\lambda\in \Lambda^+_{4,4}$ with the weight diagram $D_{\overline{\lambda}}$ given by:
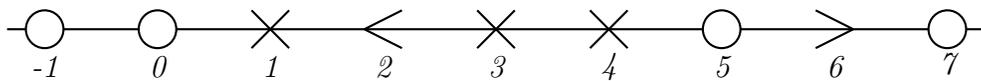
\begin{equation*}
    \label{reverse dagger}
    \begin{tikzpicture}
    
        \draw[thick, -] (-6.5 ,0)--(6.5, 0);
        
        \draw[thick, fill=white] (-6,0) circle [radius=0.25cm];
        \node[] at (-6,-0.5) {-1};
        
        \draw[thick, fill=white] (-4.5,0) circle [radius=0.25cm];
        \node[] at (-4.5,-0.5) {0};
        
        % cross
        \draw[thick, -] (-3,0)--(-3 + 0.25 ,0.25);
        \draw[thick, -] (-3,0)--(-3 - 0.25 ,0.25);
        \draw[thick, -] (-3,0)--(-3 + 0.25 ,-0.25);
        \draw[thick, -] (-3,0)--(-3 - 0.25 ,-0.25);
        \node[] at (-3,-0.5) {1};
        
        % less than 
        \draw[thick, -] (-1.5 -0.25, 0)--(-1.5 + 0.25, 0.25);
        \draw[thick, -] (-1.5 -0.25, 0)--(-1.5 + 0.25, -0.25);
        \node[] at (-1.5,-0.5) {2};

        % cross
        \draw[thick, -] (0,0)--(0 + 0.25 ,0.25);
        \draw[thick, -] (0,0)--(0 - 0.25 ,0.25);
        \draw[thick, -] (0,0)--(0 + 0.25 ,-0.25);
        \draw[thick, -] (0,0)--(0 - 0.25 ,-0.25);
        \node[] at (0,-0.5) {3};
        
        % cross
        \draw[thick, -] (1.5,0)--(1.5 + 0.25 ,0.25);
        \draw[thick, -] (1.5,0)--(1.5 - 0.25 ,0.25);
        \draw[thick, -] (1.5,0)--(1.5 + 0.25 ,-0.25);
        \draw[thick, -] (1.5,0)--(1.5 - 0.25 ,-0.25);
        \node[] at (1.5,-0.5) {4};
        
        \draw[thick, fill=white] (3,0) circle [radius=0.25cm];
        \node[] at (3,-0.5) {5};
        
        % greater than 
        \draw[thick, -] (4.5 +0.25, 0)--(4.5 - 0.25, 0.25);
        \draw[thick, -] (4.5 +0.25, 0)--(4.5 - 0.25, -0.25);
        \node[] at (4.5,-0.5) {6};
        
        \draw[thick, fill=white] (6,0) circle [radius=0.25cm];
        \node[] at (6,-0.5) {7};
        
        % cross
        %\draw[thick, -] (-6,0)--(-6 + 0.25 ,0.25);
        %\draw[thick, -] (-6,0)--(-6 - 0.25 ,0.25);
        %\draw[thick, -] (-6,0)--(-6 + 0.25 ,-0.25);
        %\draw[thick, -] (-6,0)--(-6 - 0.25 ,-0.25);
        
        % less than 
        %\draw[thick, -] (-3 -0.25, 0)--(-3 + 0.25, 0.25);
        %\draw[thick, -] (-3 -0.25, 0)--(-3 + 0.25, -0.25);
        
        % greater than 
        %\draw[thick, -] (-3 +0.25, 0)--(-3 - 0.25, 0.25);
        %\draw[thick, -] (-3 +0.25, 0)--(-3 - 0.25, -0.25);
        
    \end{tikzpicture}
\end{equation*}
Thus $\overline{\lambda} = (6\,4\,3\,1|-1\,-2\,-3\,-4)$, and $\lambda = (3\,2\,1\,1|-1\,-1\,-2\,-3)$.  Further, we have: $D = D_{\lambda}^{\Sigma}$ for $\Sigma \in \mathbb{S} $ given by word $\varepsilon^2 \delta \varepsilon \delta^2 \varepsilon \delta$. This can be done by computing the CTD of $\overline{\lambda}$ as is described in Definition \ref{CTD}.  
\end{example}

\begin{lemma}
\label{to dagger is a weight diagram}
    $\Phi(\lambda)$ is the weight diagram $D_{\lambda}^{\Sigma}$ for some $\Sigma\in \mathbb{S}$.
\end{lemma}

\begin{proof}
    Denote $D:=\Phi(\lambda)$.
    
    As before, define $$d:= \min\{p\in\Z| D_\lambda(p)=\circ \text{ and } p>\times_1\},\;\; \;  \times_\dagger := \max \{p\in \mathbb{Z}|D_{\lambda}(p)=\times \text{ and } \times_1 \le p < d\}.$$
    Recall that by the definition of the map $\Phi$, for any $p\neq \times_\dagger$, we have: $D(p)\in \{>, <, \times, \circ\}$, and $D(\times_\dagger)$ has either $>$ or several stacked symbols $\times$. 
    If $D(\times_\dagger)\in \{>, \times\}$, then a simple verification shows that $D= D_\lambda$.

    So from now on we assume that $D(\times_\dagger) = \times^s$ for some $s>1$. Note that $s$ is the number of symbols $\times$ in positions less or equal to $\times_\dagger$ in $D_\lambda$. 
    
    We also let $k\in\{1,\ldots,m\}$ be such that $\overline{\lambda}_{m-k+1}=\times_\dagger$; in particular, $k\geq s$.

    Consider the arrow diagram of $\lambda$ with arrow endpoints given by $k_1,\dots, k_m$. Denote $$I:= \{i\in \{1,\dots, m\}\,|\,\overline{\lambda}_i < \times_\dagger < k_i\} \;\; \text{ and } \;\; I':=\{i\in \{1,\dots, m\}\,|\,\overline{\lambda}_i < \times_\dagger \text{ and } k_i < \times_\dagger\}.$$ 
    Notice that for every $i \in I'$, we have: $D_\lambda(k_i)$ is $>$ (as $d$ is the first position to the right of $\times_1$ that is empty). Furthermore, since there are no $\circ$ symbols in $D_{\lambda}$ between positions $\times_1, \times_\dagger$, we have: $\times_\dagger$ is between the endpoints of every $(\times-\circ)$ sequence starting at $\times_1, \ldots, \times_{s-1}$. Hence $\sharp I \geq s>1$.

    For every $i>i'$ such that $i, i'\in I$, observe that we have $k_i > k_{i'}$ and so $k_i - \times_\dagger < k_{i'} - \times_\dagger$. Additionally, by \cref{ctd inequality lower bound ie},
     $M_i - (k_i-\overline{\lambda}_i-1) \le M_{i'} - (k_{i'}-\overline{\lambda}_{i'}-1)$. Due to this, we obtain:
    \begin{align}
    \label{ie for proof}
        M_i - (\times_\dagger -\overline{\lambda}_i-1) &= M_i - (k_i-\overline{\lambda}_i-1) + k_i - \times_\dagger \le M_{i'} - (k_{i'}-\overline{\lambda}_{i'}-1) + k_{i'} - \times_\dagger \nonumber\\&= M_{i'} - (\times_\dagger -\overline{\lambda}_{i'}-1)
    \end{align}
    
    Observe that for $i\in I$, $M_i - (k_i-\overline{\lambda}_i-1) \leq M_i - (\times_\dagger -\overline{\lambda}_i-1) < M_i$. 
    As such, for every $M_i - (k_i-\overline{\lambda}_i-1) \le j \le  M_i - (\times_\dagger -\overline{\lambda}_i-1)$ we have: $c_\lambda(\ror{i}{j}) = 1$.

    %Notice that $I\cup I' = \{m, m-1, \ldots, m-k+2\}$, so $\sharp (I\cup I') = k - 1$. 
    
    Take $\Sigma \in \mathbb{S}$ given by 
    $$B_\Sigma = \{\ror{i}{j}|i\in I \text{ and } 1\le j \le M_i - (\times_\dagger -\overline{\lambda}_i-1)\}\cup \{\ror{i}{j}|i\in I', 1\le j \le n\}$$ (note that this is a valid set by the inequality given in \cref{ie for proof}).

    We now consider the diagram $D^{\Sigma}_\lambda$. Notice that for every $i\in I'$ we have $(\overline{\lambda}_\Sigma)_i= k_i$ and for every $i\in I$ we have $(\overline{\lambda}_\Sigma)_i= \times_\dagger$. We deduce that $D^{\Sigma}_{\lambda}$ was obtained from $D_{\lambda}$ by moving the $s-1$ symbols from positions $\times_1, \ldots, \times_{s-1}$ to the position $\times_\dagger$, which implies: $D = D_{\lambda}^{\Sigma}$.
    
\end{proof}

With Lemma \ref{to dagger is a weight diagram} in mind, we give the following definition:
\begin{definition}\label{def:tail}
Let $\lambda \in \distHW$.
\begin{itemize}
    \item We denote by $\Sigma_\lambda \in \mathbb{S}$ the base for which  $\Phi(\lambda) = D_\lambda^{\Sigma_\lambda}$.
    \item We denote: $\lambda^{\dagger} := \lambda_{\Sigma_{\lambda}}$ (that is, $\lambda^{\dagger}$ is the highest weight of $L_{\distinguished}(\lambda)$ with respect to the base $\Sigma_{\lambda}$), and $\overline{\lambda}^{\dagger} := \lambda^{\dagger}+\rho_{\Sigma_{\lambda}}$.
    \item We denote: $$tail(\lambda):=s\left(\overline{\lambda}^{\dagger}\right)$$

\end{itemize}
\end{definition}

\begin{remark}
    Given $\lambda\in \distHW$, $\Phi(\lambda)$ is clearly the weight diagram of the ($\rho_{\Sigma_{\lambda}}$-shifted) weight $\overline{\lambda}^{\dagger}$.
\end{remark}

Recall that $longtail(\lambda):=\max_{\nu\in Hwt(\lambda)} s(\nu)$. The following was conjectured in \cite[Conjecture 6.3.1]{tailconjpaper}:
\begin{conj}[Tail Conjecture]\label{tail_conj} For any $\lambda\in \distHW$, we have: $$ tail(\lambda) = longtail(\lambda).$$
\end{conj}

The following is an easy consequence of \cref{lem:longtail_ineq_number_x_stacked} and the fact that $\overline{\lambda}^\dagger \in Hwt(\lambda)$:
\begin{lemma}
    Let $\lambda\in \distHW$, and let $i\geq 1$ be the maximal number of $\times$ symbols stacked in a position in $\Phi(\lambda)$. Then  we have: $$ i\leq tail(\lambda) \leq longtail(\lambda).$$
\end{lemma}
In the following subsection, we will disprove the Tail Conjecture, and show that this inequality is not always an equality.

\subsection{Construction of the counterexample to the Tail Conjecture} \label{tail conj counterexample}
\mbox{}

Consider the diagram $D\in \Lambda^{\dagger}$ given by 
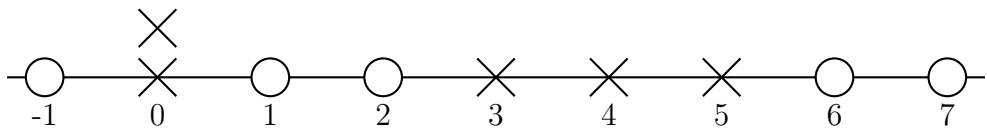
\begin{equation*}
    \label{dagger counterexample}
    \begin{tikzpicture}
    
        \draw[thick, -] (-6.5 ,0)--(6.5, 0);
        
        \draw[thick, fill=white] (-6,0) circle [radius=0.25cm];
        \node[] at (-6,-0.5) {-1};
        
        % cross
        \draw[thick, -] (-4.5,0)--(-4.5 + 0.25 ,0.25);
        \draw[thick, -] (-4.5,0)--(-4.5 - 0.25 ,0.25);
        \draw[thick, -] (-4.5,0)--(-4.5 + 0.25 ,-0.25);
        \draw[thick, -] (-4.5,0)--(-4.5 - 0.25 ,-0.25);

        % stacked cross
        \draw[thick, -] (-4.5,0 + 0.65)--(-4.5 + 0.25 ,0.25 + 0.65);
        \draw[thick, -] (-4.5,0 + 0.65)--(-4.5 - 0.25 ,0.25 + 0.65);
        \draw[thick, -] (-4.5,0 + 0.65)--(-4.5 + 0.25 ,-0.25 + 0.65);
        \draw[thick, -] (-4.5,0 + 0.65)--(-4.5 - 0.25 ,-0.25 + 0.65);
        \node[] at (-4.5,-0.5) {0};
        
        \draw[thick, fill=white] (-3,0) circle [radius=0.25cm];
        \node[] at (-3,-0.5) {1};
        
        \draw[thick, fill=white] (-1.5,0) circle [radius=0.25cm];
        \node[] at (-1.5,-0.5) {2};

        % cross
        \draw[thick, -] (0,0)--(0 + 0.25 ,0.25);
        \draw[thick, -] (0,0)--(0 - 0.25 ,0.25);
        \draw[thick, -] (0,0)--(0 + 0.25 ,-0.25);
        \draw[thick, -] (0,0)--(0 - 0.25 ,-0.25);
        \node[] at (0,-0.5) {3};
        
        % cross
        \draw[thick, -] (1.5,0)--(1.5 + 0.25 ,0.25);
        \draw[thick, -] (1.5,0)--(1.5 - 0.25 ,0.25);
        \draw[thick, -] (1.5,0)--(1.5 + 0.25 ,-0.25);
        \draw[thick, -] (1.5,0)--(1.5 - 0.25 ,-0.25);
        \node[] at (1.5,-0.5) {4};
        
        % cross
        \draw[thick, -] (3,0)--(3 + 0.25 ,0.25);
        \draw[thick, -] (3,0)--(3 - 0.25 ,0.25);
        \draw[thick, -] (3,0)--(3 + 0.25 ,-0.25);
        \draw[thick, -] (3,0)--(3 - 0.25 ,-0.25);
        \node[] at (3,-0.5) {5};
        
        % greater than 
        \draw[thick, fill=white] (4.5,0) circle [radius=0.25cm];
        \node[] at (4.5,-0.5) {6};
        
        \draw[thick, fill=white] (6,0) circle [radius=0.25cm];
        \node[] at (6,-0.5) {7};
        
        % cross
        %\draw[thick, -] (-6,0)--(-6 + 0.25 ,0.25);
        %\draw[thick, -] (-6,0)--(-6 - 0.25 ,0.25);
        %\draw[thick, -] (-6,0)--(-6 + 0.25 ,-0.25);
        %\draw[thick, -] (-6,0)--(-6 - 0.25 ,-0.25);
        
        % less than 
        %\draw[thick, -] (-3 -0.25, 0)--(-3 + 0.25, 0.25);
        %\draw[thick, -] (-3 -0.25, 0)--(-3 + 0.25, -0.25);
        
        % greater than 
        %\draw[thick, -] (-3 +0.25, 0)--(-3 - 0.25, 0.25);
        %\draw[thick, -] (-3 +0.25, 0)--(-3 - 0.25, -0.25);
        
    \end{tikzpicture}
\end{equation*}

Computing the weight $\lambda:=\Psi(D) \in \Lambda^+_{5|5}$ we find that $\overline{\lambda}= (5\,4\,3\,0\,-1|1\,0\,-3\,-4\,-5)$, and $D_{\lambda}:=D_{\lambda}^{\distinguished}$ is 

\begin{equation*}
    \begin{tikzpicture}
    
        \draw[thick, -] (-6.5 ,0)--(6.5, 0);
        
        % cross
        \draw[thick, -] (-6,0)--(-6 + 0.25 ,0.25);
        \draw[thick, -] (-6,0)--(-6 - 0.25 ,0.25);
        \draw[thick, -] (-6,0)--(-6 + 0.25 ,-0.25);
        \draw[thick, -] (-6,0)--(-6 - 0.25 ,-0.25);
        \node[] at (-6,-0.5) {-1};
        
        % cross
        \draw[thick, -] (-4.5,0)--(-4.5 + 0.25 ,0.25);
        \draw[thick, -] (-4.5,0)--(-4.5 - 0.25 ,0.25);
        \draw[thick, -] (-4.5,0)--(-4.5 + 0.25 ,-0.25);
        \draw[thick, -] (-4.5,0)--(-4.5 - 0.25 ,-0.25);
        \node[] at (-4.5,-0.5) {0};
        
        \draw[thick, fill=white] (-3,0) circle [radius=0.25cm];
        \node[] at (-3,-0.5) {1};
        
        \draw[thick, fill=white] (-1.5,0) circle [radius=0.25cm];
        \node[] at (-1.5,-0.5) {2};

        % cross
        \draw[thick, -] (0,0)--(0 + 0.25 ,0.25);
        \draw[thick, -] (0,0)--(0 - 0.25 ,0.25);
        \draw[thick, -] (0,0)--(0 + 0.25 ,-0.25);
        \draw[thick, -] (0,0)--(0 - 0.25 ,-0.25);
        \node[] at (0,-0.5) {3};
        
        % cross
        \draw[thick, -] (1.5,0)--(1.5 + 0.25 ,0.25);
        \draw[thick, -] (1.5,0)--(1.5 - 0.25 ,0.25);
        \draw[thick, -] (1.5,0)--(1.5 + 0.25 ,-0.25);
        \draw[thick, -] (1.5,0)--(1.5 - 0.25 ,-0.25);
        \node[] at (1.5,-0.5) {4};
        
        % cross
        \draw[thick, -] (3,0)--(3 + 0.25 ,0.25);
        \draw[thick, -] (3,0)--(3 - 0.25 ,0.25);
        \draw[thick, -] (3,0)--(3 + 0.25 ,-0.25);
        \draw[thick, -] (3,0)--(3 - 0.25 ,-0.25);
        \node[] at (3,-0.5) {5};
        
        % greater than 
        \draw[thick, fill=white] (4.5,0) circle [radius=0.25cm];
        \node[] at (4.5,-0.5) {6};
        
        \draw[thick, fill=white] (6,0) circle [radius=0.25cm];
        \node[] at (6,-0.5) {7};
        
        % cross
        %\draw[thick, -] (-6,0)--(-6 + 0.25 ,0.25);
        %\draw[thick, -] (-6,0)--(-6 - 0.25 ,0.25);
        %\draw[thick, -] (-6,0)--(-6 + 0.25 ,-0.25);
        %\draw[thick, -] (-6,0)--(-6 - 0.25 ,-0.25);
        
        % less than 
        %\draw[thick, -] (-3 -0.25, 0)--(-3 + 0.25, 0.25);
        %\draw[thick, -] (-3 -0.25, 0)--(-3 + 0.25, -0.25);
        
        % greater than 
        %\draw[thick, -] (-3 +0.25, 0)--(-3 - 0.25, 0.25);
        %\draw[thick, -] (-3 +0.25, 0)--(-3 - 0.25, -0.25);
        
    \end{tikzpicture}
\end{equation*}

Computing $\Sigma_\lambda$ as shown in Lemma \ref{to dagger is a weight diagram}, we find that $\Sigma_\lambda$ is given by the word $\eps^4 \delta \eps \delta^4$.

We now compute $\lambda^{\dagger}$ by  \cref{def:tail}, and obtain: $\overline{\lambda}^\dagger = (5\,4\,3\,0\,0|0\,0\,-3\,-4\,-5)$. Next, we obtain the following visualization diagram of $f_{\lambda^\dagger}$ (see \cref{def:f_nu}):
\begin{equation*}
    \label{dagger fnu}
\begin{tikzpicture}[anchorbase,scale=1.1]
% row 1
\node at (0,1.5) {$\star $};

% row 2
\node at (0.5,1) {$\circ $};
\node at (-0.5, 1) {$\circ $};

% row 3
\node at (1, 0.5) {$\circ $};
\node at (0, 0.5) {$\star $};
\node at (-1, 0.5) {$\circ $};

% row 4
\node at (1.5, 0) {$\circ $};
\node at (0.5, 0) {$\circ $};
\node at (-0.5, 0) {$\circ $};
\node at (-1.5, 0) {$\circ $};

% row 5
\node at (2, -0.5) {$\circ $};
\node at (1, -0.5) {$\circ $};
\node at (0, -0.5) {$\star $};
\node at (-1, -0.5) {$\circ $};
\node at (-2, -0.5) {$\circ $};

% row 6
\node at (1.5, -1) {$\circ $};
\node at (0.5, -1) {$\circ $};
\node at (-0.5, -1) {$\circ $};
\node at (-1.5, -1) {$\circ $};

% row 7
\node at (1, -1.5) {$\circ $};
\node at (0, -1.5) {$\star $};
\node at (-1, -1.5) {$\circ $};

% row 8
\node at (0.5, -2) {$\star $};
\node at (-0.5, -2) {$\star $};

% row 9
\node at (0, -2.5) {$\star $};

%\draw[-] (1.5,0.5)--(1.5,-1);

\end{tikzpicture}
\end{equation*} 

As such, a largest iso-set  orthogonal to $\overline{\lambda}^\dagger$ is $\{\ror{4}{1},\ror{5}{2}\}$ and so $$tail(\lambda):=s\left(\overline{\lambda}^{\dagger}\right)=2.$$

Since $\lambda \in \Lambda^+_{5|5}$ we may construct a CTD $c_{\lambda}$ for it. This CTD is given by the diagram

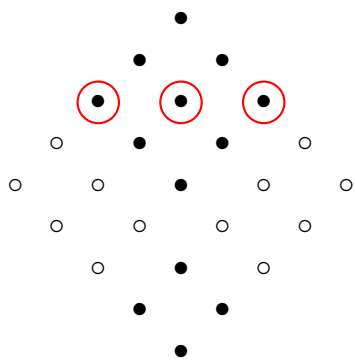
\begin{equation*}
\begin{tikzpicture}[anchorbase,scale=1.1]
% row 1
\node at (0,1.5) {$\bullet $};

% row 2
\node at (0.5,1) {$\bullet $};
\node at (-0.5, 1) {$\bullet $};

% row 3
\node at (1, 0.5) {$\bullet $};
\draw[thick, red] (1,0.5) circle [radius=0.25cm];
\node at (0, 0.5) {$\bullet $};
\draw[thick, red] (0,0.5) circle [radius=0.25cm];
\node at (-1, 0.5) {$\bullet $};
\draw[thick, red] (-1,0.5) circle [radius=0.25cm];

% row 4
\node at (1.5, 0) {$\circ $};
\node at (0.5, 0) {$\bullet $};
\node at (-0.5, 0) {$\bullet $};
\node at (-1.5, 0) {$\circ $};

% row 5
\node at (2, -0.5) {$\circ $};
\node at (1, -0.5) {$\circ $};
\node at (0, -0.5) {$\bullet $};
\node at (-1, -0.5) {$\circ $};
\node at (-2, -0.5) {$\circ $};

% row 6
\node at (1.5, -1) {$\circ $};
\node at (0.5, -1) {$\circ $};
\node at (-0.5, -1) {$\circ $};
\node at (-1.5, -1) {$\circ $};

% row 7
\node at (1, -1.5) {$\circ $};
\node at (0, -1.5) {$\bullet $};
\node at (-1, -1.5) {$\circ $};

% row 8
\node at (0.5, -2) {$\bullet $};
\node at (-0.5, -2) {$\bullet $};

% row 9
\node at (0, -2.5) {$\bullet $};

%\draw[-] (1.5,0.5)--(1.5,-1);

\end{tikzpicture}
\end{equation*} 

Based on this diagram, consider the incomparable subset $S =\{\ror{1}{3},\ror{2}{4},\ror{3}{5}\}$ of $\mathcal{R}$ (the elements of $S$ are marked by red circles in the diagram above). We have: $c_{\lambda}(\alpha)=1$ for all $\alpha\in S$. 

If we consider $\Sigma \in \mathbb{S}$ such that $\Sigma\cap \mathcal{R} = S$ we find that 
$\overline{\lambda}_\Sigma=(5\, 5\, 5\, 2\, 1|-1\,-2\,-5\,-5\,-5)$ and $D_{\lambda}^{\Sigma}$ is

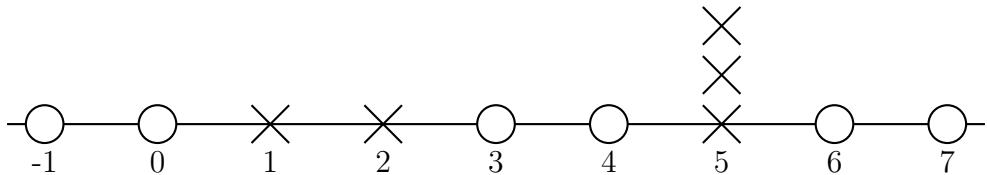
\begin{equation*}
    \begin{tikzpicture}
    
        \draw[thick, -] (-6.5 ,0)--(6.5, 0);
        
        % cross
        \draw[thick, fill=white] (-6,0) circle [radius=0.25cm];
        \node[] at (-6,-0.5) {-1};
        
        % cross
        \draw[thick, fill=white] (-4.5,0) circle [radius=0.25cm];
        \node[] at (-4.5,-0.5) {0};
        
        % cross
        \draw[thick, -] (-3,0)--(-3 + 0.25 ,0.25);
        \draw[thick, -] (-3,0)--(-3 - 0.25 ,0.25);
        \draw[thick, -] (-3,0)--(-3 + 0.25 ,-0.25);
        \draw[thick, -] (-3,0)--(-3 - 0.25 ,-0.25);
        \node[] at (-3,-0.5) {1};
        
        % cross
        \draw[thick, -] (-1.5,0)--(-1.5 + 0.25 ,0.25);
        \draw[thick, -] (-1.5,0)--(-1.5 - 0.25 ,0.25);
        \draw[thick, -] (-1.5,0)--(-1.5 + 0.25 ,-0.25);
        \draw[thick, -] (-1.5,0)--(-1.5 - 0.25 ,-0.25);
        \node[] at (-1.5,-0.5) {2};

        \draw[thick, fill=white] (0,0) circle [radius=0.25cm];
        \node[] at (0,-0.5) {3};
        
        \draw[thick, fill=white] (1.5,0) circle [radius=0.25cm];
        \node[] at (1.5,-0.5) {4};
        
        % cross
        \draw[thick, -] (3,0)--(3 + 0.25 ,0.25);
        \draw[thick, -] (3,0)--(3 - 0.25 ,0.25);
        \draw[thick, -] (3,0)--(3 + 0.25 ,-0.25);
        \draw[thick, -] (3,0)--(3 - 0.25 ,-0.25);

        % stacked cross
        \draw[thick, -] (3,0 + 0.65)--(3 + 0.25 ,0.25 + 0.65);
        \draw[thick, -] (3,0 + 0.65)--(3 - 0.25 ,0.25 + 0.65);
        \draw[thick, -] (3,0 + 0.65)--(3 + 0.25 ,-0.25 + 0.65);
        \draw[thick, -] (3,0 + 0.65)--(3 - 0.25 ,-0.25 + 0.65);
        % stacked cross
        \draw[thick, -] (3,0 + 2* 0.65)--(3 + 0.25 ,0.25 + 2* 0.65);
        \draw[thick, -] (3,0 + 2* 0.65)--(3 - 0.25 ,0.25 + 2* 0.65);
        \draw[thick, -] (3,0 + 2* 0.65)--(3 + 0.25 ,-0.25 + 2* 0.65);
        \draw[thick, -] (3,0 + 2* 0.65)--(3 - 0.25 ,-0.25 + 2* 0.65);
        \node[] at (3,-0.5) {5};
        
        \draw[thick, fill=white] (4.5,0) circle [radius=0.25cm];
        \node[] at (4.5,-0.5) {6};
        
        \draw[thick, fill=white] (6,0) circle [radius=0.25cm];
        \node[] at (6,-0.5) {7};
        
        % cross
        %\draw[thick, -] (-6,0)--(-6 + 0.25 ,0.25);
        %\draw[thick, -] (-6,0)--(-6 - 0.25 ,0.25);
        %\draw[thick, -] (-6,0)--(-6 + 0.25 ,-0.25);
        %\draw[thick, -] (-6,0)--(-6 - 0.25 ,-0.25);
        
        % less than 
        %\draw[thick, -] (-3 -0.25, 0)--(-3 + 0.25, 0.25);
        %\draw[thick, -] (-3 -0.25, 0)--(-3 + 0.25, -0.25);
        
        % greater than 
        %\draw[thick, -] (-3 +0.25, 0)--(-3 - 0.25, 0.25);
        %\draw[thick, -] (-3 +0.25, 0)--(-3 - 0.25, -0.25);
        
    \end{tikzpicture}
\end{equation*}

We have $3$ stacked symbols $\times$ in a single position here, so $s(\overline{\lambda}_{\Sigma}) \geq 3$ by \cref{lem:longtail_ineq_number_x_stacked}. This shows that 
$longtail(\lambda) = \max_{\nu\in \text{Hwt}(\lambda)}{s(\nu)}\ge s(\overline{\lambda}_\Sigma)\ge 3$.
So we see that in this example, $longtail(\lambda) > tail(\lambda)$, disproving the Tail Conjecture \ref{tail_conj}.

In fact, we will see in \cref{cor:CTD_vs_longtail_eq} that
$$ longtail(\lambda) = \sharp S = 3.$$

\subsection{Corrected Tail Conjecture: description in terms of CTD}

We would now like to discuss a correction to the Tail Conjecture \ref{tail_conj}. To be more specific, we will develop explicit formulas for $longtail(\lambda)$ which will allow us to compute this value without actually checking every base $\Sigma\in \mathbb{S}$. 

In this subsection, we will provide our first explicit formula for the value of $longtail(\lambda)$, using the change tracking diagram (CTD) $c_{\lambda}$ of $\lambda$. This formula is a very efficient tool for computing the value $longtail(\lambda)$, since it only requires us to compute the CTD of $\lambda$, compared to the original definition of $longtail(\lambda)$ which required us to consider all the bases in $\mathbb{S}$ one by one.

Recall that for $\alpha\in \mathcal{R}$, we have $c_{\lambda}(\alpha) = 1$ iff $\alpha$ changes the diagram of $\lambda$.

\begin{definition}
For any $\lambda\in \distHW$, denote by $$C_{\lambda}:=\{\alpha\in \mathcal{R}\,|\, c_{\lambda}(\alpha)=1\}$$ the set of right odd roots changing the diagram of $\lambda$. 
\end{definition}

\begin{lemma}\label{lem:CTD_vs_s_lam_ineq}
    Let $\lambda\in \distHW$, and consider its CTD $c_{\lambda}$. Let $S\subset C_{\lambda}$ be an incomparable set, and let $\Sigma \in \mathbb{S}$ such that $\Sigma\cap \mathcal{R} = S$ (see Lemma \ref{ror are incomparable})
    
    Then $s(\overline{\lambda}_\Sigma)\ge \sharp S$.
\end{lemma}

\begin{proof}
Recall that for any $\alpha\in \mathcal{R}$, $$c_{\lambda}(\alpha)=1 \;\;\;\Longleftrightarrow \;\;\; \forall \Sigma \in \mathbb{S} \; \text{ such that }\alpha \in \Sigma, \; (\alpha|\overline{\lambda}_\Sigma)=0.$$

As such it is immediately obvious that for any $\alpha\in S$,  $(\alpha|\overline{\lambda}_\Sigma) = 0$. Since $S$ is an incomparable set, we conclude that $s(\lambda_\Sigma)\ge \sharp S$. 
\end{proof}

The next corollary gives the first explicit way of computing $longtail(\lambda)$: it states that $longtail(\lambda)$ is the size of the maximal incomparable subset of $\mathcal{R}$ all of whose elements change the diagram of $\lambda$.
\begin{cor}\label{cor:CTD_vs_longtail_eq}
    Let $\lambda\in \distHW$. We have:
    \begin{equation}\label{eq:longtail_CTD}
    longtail(\lambda) = \max_{\Sigma\in \mathbb{S}} \sharp\Sigma\cap C_{\lambda}
    \end{equation}
\end{cor}
\begin{proof}
By \cref{lem:CTD_vs_s_lam_ineq}, any incomparable subset of $C_{\lambda}$ is of cardinality at most $longtail(\lambda)$. This shows that the left hand side of \eqref{eq:longtail_CTD} is greater or equal to the right hand side. 

    Let $\Sigma\in \mathbb{S}$, and let $S\subset \mathcal{R}$ be an incomparable set of maximal cardinality such that $(S|\overline{\lambda}_{\Sigma})=0$. Then by Proposition \ref{CTD 1 iff alpha perp nu}, every $\alpha \in S$ has $c_\lambda(\alpha)=1$ and therefore $S$ is an incomparable subset of $C_{\lambda}$. 
    
    Finally, any incomparable subset $S\subset C_{\lambda}$ can be written as $
S=\Sigma' \cap \mathcal{R}$ for some $\Sigma'\in \mathbb{S}$, proving that $$\sharp S \leq \max_{\Sigma\in \mathbb{S}} \sharp\Sigma\cap C_{\lambda}$$ This proves that the right hand side in \eqref{eq:longtail_CTD} is greater or equal to the left hand side, completing the proof. 
\end{proof}

\subsection{Corrected Tail Conjecture: description in terms of arrow and cap diagrams}

In this subsection, we will give explicit formulas for $longtail(\lambda)$ in terms of arrow and cap diagrams of $\lambda$. First, $longtail(\lambda)$ will be described in terms of the arrow diagram of $\lambda$. Secondly, we will describe it in terms of the cap diagrams of $\lambda$.

We begin with some auxiliary terminology and lemmas.

\begin{definition}
Let $\lambda\in \distHW$, and let $r\in \mathbb{Z}$. We say that an arrow or a cap in $D_{\lambda}$ {\bf goes over $r$} if the said arrow or cap starts in a position $p\leq r$ and ends in a position $p'>r$.
\end{definition}

\begin{lemma}\label{lem:rectangle_in_CTD}
    Let $\lambda \in \distHW$.
    
    Then for any incomparable $S \subset C_{\lambda}$, there exist $1\leq I'\leq I''\leq m $, $1\leq J'\leq J''\leq n$ such that $$S\subset \{\ror{i}{j} \, |\, I'\leq i\leq I'', \, J'\leq j\leq J''\}\subset C_{\lambda}.$$ Diagrammatically, this means that $C_{\lambda}$ contains a "rectangle" of right odd roots, which in turn contains $S$. 
\end{lemma}

Before proving Lemma \ref{lem:rectangle_in_CTD}, we provide an example to demonstrate the meaning of this lemma in terms of CTDs.

\begin{example}
    Let $\lambda \in \Lambda^+_{4|4}$ given by $\overline{\lambda}=(4\, 3\, 2\, 0|0\, -2\, -3\, -4)$. The diagram $D_\lambda$ is 

    
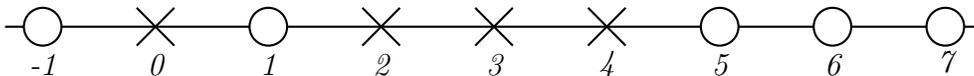
\begin{equation*}
    \begin{tikzpicture}
    
        \draw[thick, -] (-6.5 ,0)--(6.5, 0);
        
        \draw[thick, fill=white] (-6,0) circle [radius=0.25cm];
        \node[] at (-6,-0.5) {-1};
        
        % cross
        \draw[thick, -] (-4.5,0)--(-4.5 + 0.25 ,0.25);
        \draw[thick, -] (-4.5,0)--(-4.5 - 0.25 ,0.25);
        \draw[thick, -] (-4.5,0)--(-4.5 + 0.25 ,-0.25);
        \draw[thick, -] (-4.5,0)--(-4.5 - 0.25 ,-0.25);
        \node[] at (-4.5,-0.5) {0};
        
        \draw[thick, fill=white] (-3,0) circle [radius=0.25cm];
        \node[] at (-3,-0.5) {1};
        
        % cross
        \draw[thick, -] (-1.5,0)--(-1.5 + 0.25 ,0.25);
        \draw[thick, -] (-1.5,0)--(-1.5 - 0.25 ,0.25);
        \draw[thick, -] (-1.5,0)--(-1.5 + 0.25 ,-0.25);
        \draw[thick, -] (-1.5,0)--(-1.5 - 0.25 ,-0.25);
        \node[] at (-1.5,-0.5) {2};

        % cross
        \draw[thick, -] (0,0)--(0 + 0.25 ,0.25);
        \draw[thick, -] (0,0)--(0 - 0.25 ,0.25);
        \draw[thick, -] (0,0)--(0 + 0.25 ,-0.25);
        \draw[thick, -] (0,0)--(0 - 0.25 ,-0.25);
        \node[] at (0,-0.5) {3};
        
        % cross
        \draw[thick, -] (1.5,0)--(1.5 + 0.25 ,0.25);
        \draw[thick, -] (1.5,0)--(1.5 - 0.25 ,0.25);
        \draw[thick, -] (1.5,0)--(1.5 + 0.25 ,-0.25);
        \draw[thick, -] (1.5,0)--(1.5 - 0.25 ,-0.25);
        \node[] at (1.5,-0.5) {4};
        
        \draw[thick, fill=white] (3,0) circle [radius=0.25cm];
        \node[] at (3,-0.5) {5};
        
        % greater than 
        \draw[thick, fill=white] (4.5,0) circle [radius=0.25cm];
        \node[] at (4.5,-0.5) {6};
        
        \draw[thick, fill=white] (6,0) circle [radius=0.25cm];
        \node[] at (6,-0.5) {7};
        
        % cross
        %\draw[thick, -] (-6,0)--(-6 + 0.25 ,0.25);
        %\draw[thick, -] (-6,0)--(-6 - 0.25 ,0.25);
        %\draw[thick, -] (-6,0)--(-6 + 0.25 ,-0.25);
        %\draw[thick, -] (-6,0)--(-6 - 0.25 ,-0.25);
        
        % less than 
        %\draw[thick, -] (-3 -0.25, 0)--(-3 + 0.25, 0.25);
        %\draw[thick, -] (-3 -0.25, 0)--(-3 + 0.25, -0.25);
        
        % greater than 
        %\draw[thick, -] (-3 +0.25, 0)--(-3 - 0.25, 0.25);
        %\draw[thick, -] (-3 +0.25, 0)--(-3 - 0.25, -0.25);
        
    \end{tikzpicture}
\end{equation*}

    Consider the incomparable set $S = \{\ror{3}{2}, \ror{2}{3}, \ror{1}{4}\} \subset \mathcal{R}$. The diagram below depicts the CTD $c_\lambda$ where the elements of $S$ are circled in red. We draw in blue the smallest rectangle of right odd roots containing $S$:

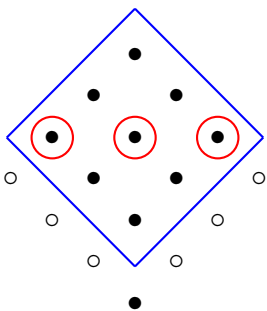
    \begin{equation*}
\begin{tikzpicture}[anchorbase,scale=1.1]
% row 1
\node at (0,1.5) {$\bullet $};

% row 2
\node at (0.5,1) {$\bullet $};
\node at (-0.5, 1) {$\bullet $};

% row 3
\node at (1, 0.5) {$\bullet $};
\draw[thick, red] (1,0.5) circle [radius=0.25cm];
\node at (0, 0.5) {$\bullet $};
\draw[thick, red] (0,0.5) circle [radius=0.25cm];
\node at (-1, 0.5) {$\bullet $};
\draw[thick, red] (-1,0.5) circle [radius=0.25cm];

% row 4
\node at (1.5, 0) {$\circ $};
\node at (0.5, 0) {$\bullet $};
\node at (-0.5, 0) {$\bullet $};
\node at (-1.5, 0) {$\circ $};

% row 5
\node at (1, -0.5) {$\circ $};
\node at (0, -0.5) {$\bullet $};
\node at (-1, -0.5) {$\circ $};

% row 6
\node at (0.5, -1) {$\circ $};
\node at (-0.5, -1) {$\circ $};

% row 7
\node at (0, -1.5) {$\bullet $};

\draw[thick, blue] (1.55,0.5)--(0,2.05);
\draw[thick, blue] (-1.55,0.5)--(0,2.05);
\draw[thick, blue] (1.55,0.5)--(0,-1.05);
\draw[thick, blue] (-1.55,0.5)--(0,-1.05);

\end{tikzpicture}
\end{equation*} 
\end{example}

\begin{proof}[Proof of Lemma \ref{lem:rectangle_in_CTD}]
    Let $S \subset C_{\lambda}$ be an incomparable subset. Since it is an incomparable set of right odd roots, it can be written as 
$S=\{\ror{i_1}{j_1}, \ror{i_2}{j_2},\ldots, \ror{i_K}{j_K}\}$ for some $K\geq 0$, where $i_1<i_2<\ldots<i_{K} $ and $j_1<j_2<\ldots<j_K$. 

Now, we use the fact that $S \subset C_{\lambda}$; Theorem \ref{theorem for ror change inequality} implies: 
\begin{equation}\label{eq:rect_ctd_eq_2}
    M_{i_l}-(k_{i_l} - \overline{\lambda}_{i_l} - 1) \leq j_l\leq M_{i_l}
\end{equation}
for every $l=1, \ldots, K$. 

We set: $$I':= i_1,\; J':=j_1,\; I'':=i_K, \; J'':=j_K. $$
Clearly, $I'<I''$ and $J'<J''$.

Next, $c_\lambda(\ror{I'}{J'})=c_\lambda(\ror{I''}{J''})=1$ so by \eqref{eq:rect_ctd_eq_2}, 
$$M_{I'}-(k_{I'} - \overline{\lambda}_{I'} - 1) \leq J', \;\; J''\leq M_{I''}$$
Thus for any $J'\leq j\leq J''$, we have: 
$$M_{I'}-(k_{I'} - \overline{\lambda}_{I'} - 1) \leq j\leq M_{I''}.$$

By \cref{ctd inequality lower bound ie}, for every $I'\leq i\leq I''$, we have:
$$M_i\geq M_{I''}, \;\;  M_{I'}-(k_{I'} - \overline{\lambda}_{I'} - 1)\geq M_i - (k_i - \overline{\lambda}_i - 1).$$ 

So for every $I'\leq i\leq I''$, $J'\leq j\leq J''$, we have: 
$$M_i - (k_i - \overline{\lambda}_i - 1) \leq j\leq M_i.$$

We conclude that 
$$S\subset \{\ror{i}{j} \, |\, I'\leq i\leq I'', \, J'\leq j\leq J''\}\subset C_{\lambda}$$ as required.

\end{proof}

\begin{remark}\label{rmk:rectangle_iso_set}
    In the setting of \cref{lem:rectangle_in_CTD}, let $k:=\min\{I''-I', J''-J'\}$(clearly, $k+1\geq \sharp S$). Consider the set $\{\ror{I'}{J'}, \ror{I'+1}{J'+1}, \ldots, \ror{I'+k}{J'+k}\}$. This set is incomparable, has cardinality $k+1$ and is contained in $\{\ror{i}{j} \, |\, I'\leq i\leq I'', \, J'\leq j\leq J''\}\subset C_{\lambda}$.

    In particular, if we take $K:=longtail(\lambda)$, the above argument implies that we can always choose an incomparable subset of $C_{\lambda}$ of size $K$ which has the form above. Formally, that means that there exist $1\leq I'\leq m - K$, $1\leq J'\leq n - K$ such that $$S':=\{\ror{I'}{J'}, \ror{I'+1}{J'+1}, \ldots, \ror{I'+K}{J'+K}\} \subset C_{\lambda}$$ is a maximal cardinality incomparable subset of $C_{\lambda}$. Moreover, $S'$ is contained (as a "diagonal") in a "rectangle" inside $C_{\lambda}$: $$S' \subset \{\ror{i}{j} \, |\, I'\leq i\leq I'+K, \, J'\leq j\leq J'+K\}\subset C_{\lambda}.$$ 

    Expanding on Corollary \ref{cor:CTD_vs_longtail_eq}, this lemma implies that $longtail(\lambda)$ is given by width of the largest black rectangle appearing in $c_\lambda$.
\end{remark}

\begin{cor}
    We have:
    $$longtail(\lambda) = \max_{A\in \mathcal{A}_{\lambda}} longtail(\lambda^A).$$
\end{cor}
\begin{proof}
    Let $S$ be an incomparable subset of $C_{\lambda}$. By \cref{lem:rectangle_in_CTD}, there exist $I'\leq  I''$, $J'\leq J''$ such that the "rectangle" 
    $$\{\ror{i}{j} \,|\, I'\leq i\leq I'', \, J'<j<J'' \} $$ contains $S$ and is itself contained in $C_{\lambda}$. Due to \cref{not same atom prop}, this entire rectangle lies in $\mathcal{R}^A$ for a single $\lambda$-atom $\lambda^A$. 
    
    Thus is we take an incomparable subset of $C_{\lambda}$ of maximal cardinality, it will also be an incomparable subset of maximal cardinality in $C_{\lambda^A}$ for some $\lambda$-atom $\lambda^A$, proving that  
     $$longtail(\lambda) \leq \max_{A\in \mathcal{A}_{\lambda}} longtail(\lambda^A).$$
     
     On the other hand, given any $\lambda$-atom $\lambda^A$ and an incomparable subset $S$ of $C_{\lambda^A}$, we have a corresponding incomparable subset $S'$ of $C_{\lambda} \cap \mathcal{R}^A$ by \cref{atom ctd prop}, and $\sharp S' = \sharp S$. Thus      $$longtail(\lambda) \geq \max_{A\in \mathcal{A}_{\lambda}} longtail(\lambda^A)$$
     and the statement is proved.
\end{proof}

\begin{lemma}\label{lem:max_arrows_over_point_and_rectangle}
    Let $\lambda\in \distHW$. Assume there exist $1\leq I'\leq I''\leq m$, $1\leq J'\leq J''\leq n$ such that $I''-I'=J''-J'$, $\{\ror{I'}{J'}, \ror{I'+1}{J'+1}, \ldots, \ror{I''}{J''}\} \subset C_{\lambda}$. Then the intersection of the half-closed intervals $[\overline{\lambda}_i, k_i)\subset \mathbb{R}$ for $I'\leq i\leq I''$ contains an integer point; that is, $$\bigcap_{i=I'}^{I''} \{r\in \Z\,|\, \overline{\lambda}_i\leq r< k_i\}  \neq \emptyset.$$
    Diagrammatically, if the CTD $c_\lambda$ contains a black rectangle whose corners are $\ror{I'}{J'}, \ror{I'}{J''}, \ror{I''}{J'},\ror{I''}{J''}$, then there exists a point $r\in\mathbb{Z}$ such that at least $\min\{I''-I',J''-J'\}$ arrows go over it in the arrow diagram of $\lambda$.  
\end{lemma}

\begin{proof}
Recall that $\overline{\lambda}_{I''}<\ldots <\overline{\lambda}_{I'}$ and by the definition of the arrow diagram, $k_{I''}<\ldots <k_{I'}$. Thus we only need to check that $k_{I''} >\overline{\lambda}_{I'} $.

Fix $J'\leq j\leq  J''$ and denote $t:=-\overline{\lambda}^{I''}_{-j}$. As we saw in the proof of \cref{lem:rectangle_in_CTD}, $$    M_{i}-(k_{i} - \overline{\lambda}_{i} - 1) \leq J' \leq j\leq J''\leq M_{i}$$ for any $i=I', \ldots, I''$.  

We use again Theorem \ref{theorem for ror change inequality} which implies that for any $i$,
\begin{enumerate}
    \item\label{itm:CTD_sq_1} 
$ -\overline{\lambda}^i_{-s} < -\overline{\lambda}^i_{-(s+1)}$ for any $s$,
    \item\label{itm:CTD_sq_2}  $\overline{\lambda}^i_i=\overline{\lambda}_i$,
    \item\label{itm:CTD_sq_3}  $M_i=\sharp\{s | -\overline{\lambda}^i_{-s}< k_{i}\}$ (the total number of symbols $<$ and $\times$ in $D^{\Sigma^{i}}_{\lambda}$ to the left of position $k_i$),
    \item\label{itm:CTD_sq_4}  $M_i-(k_i-\overline{\lambda}_i+1)=\sharp\{s | -\overline{\lambda}^i_{-s}< \overline{\lambda}_i\}$ (the total number of symbols $<$ and $\times$ in $D^{\Sigma^{i}}_{\lambda}$ to the left of position $\overline{\lambda}_i$).
\end{enumerate}

Due to the fact that $j\leq M_{I''}$, \eqref{itm:CTD_sq_3} implies: $t\leq k_{I''}$.

Next, by the same Theorem \ref{theorem for ror change inequality}, the reflection $r_{i, j}$ changes the diagrams $D_{\lambda}^{\Sigma^i}$ for $i=I'', I''-1, \ldots, I'+1$, so $$-\overline{\lambda}^{I'}_{-j}= t+I''-I'\leq k_{I''} + I''-I'.$$

Now, by \eqref{itm:CTD_sq_4} above we also have $$\sharp\{s | -\overline{\lambda}^{I'}_{-s}< \overline{\lambda}_{I'}\}\,=\,M_{
I'}-(k_{I'} - \overline{\lambda}_{I'} - 1)\, \leq \, j$$
hence $ \overline{\lambda}_{I'}= \overline{\lambda}^{I'}_{I'} \leq -\overline{\lambda}^{I'}_{-j}  $.

We conclude that for any $j=J', \ldots, J''=J'+I''-I'$, we have:
$ \overline{\lambda}_{I'}\leq -\overline{\lambda}^{I'}_{-j} \leq k_{I''} + I''-I'$. But by \eqref{itm:CTD_sq_1}, $\overline{\lambda}^{I'}_{-j}$ for $j=J', \ldots, J''=J'+I''-I'$ are $I''-I'+1$ distinct integers, so $ \overline{\lambda}_{I'}+I''-I'+1\leq k_{I''} + I''-I'$ and thus $\overline{\lambda}_{I'}\leq k_{I''}$ as required.
\end{proof}

\begin{cor}\label{cor:iso_set_in_CTD_vs_arrows_over}
Let $\lambda\in \distHW$, and let $S\subset C_{\lambda}$. Then there exists $r\in \mathbb{Z}$ such that $$\sharp S\,\leq \, \sharp\{i\in \{1,\dots,m\}\,\rvert\, \overline{\lambda}_i \leq r < k_i\}. $$
In other words, there exists a position in $D_{\lambda}$ which has at least $\sharp S$ arrows going over it.
\end{cor}
\begin{proof}
By \cref{rmk:rectangle_iso_set}, we may replace $S$ by an incomparable subset $S' \subset C_{\lambda}$ of the same size which has the form $S'=\{\ror{I'}{J'}, \ror{I'+1}{J'+1}, \ldots, \ror{I''}{J''}\}$ for some $1\leq I'\leq I''\leq m$, $1\leq J'\leq J''\leq n$ such that $$I''-I'\,=\,J''-J' \,=\, \sharp S'-1\,=\,\sharp S-1.$$ 
By \cref{lem:max_arrows_over_point_and_rectangle}, we know that there exists $r\in \mathbb{Z}$ such that $$\forall I'\leq i\leq I'', \;  \overline{\lambda}_i\leq r < k_i $$ so we conclude that $$\sharp S\, =\, \sharp S' \, = \, I''-I'+1\, =\, \sharp\{I'\leq i\leq I''\} \, \leq \,\sharp\{i\in \{1,\dots,m\}\,\rvert\, \overline{\lambda}_i\leq r < k_i\}. $$

\end{proof}

\begin{lemma}\label{lem:stacked_x_vs_arrows_over_ineq}
    Let $\lambda\in \distHW$. Let $r\in \Z$, and let $$k:=\sharp\{i\in \{1,\dots,m\}\,\rvert\, \overline{\lambda}_i \le r < k_i\} $$ be the number of arrows going over position $r$ or starting at position $r$ in $D_{\lambda}$.
    
    Then there exists $\Sigma\in \mathbb{S}$ such that $D_{\lambda}^{\Sigma}$ has at least $k$ (stacked) symbols $\times$ in a single position.
    
\end{lemma}

\begin{proof}

    Notice that for any $r$ such that $\sharp\{i\in \{1,\dots,m\}\,\rvert\, \overline{\lambda}_i \le r < k_i\} = 0$ there is nothing to prove. 
    
    Next, assume $r$ has at least one arrow going over it. 
    Observe that the intervals $[\overline{\lambda}_i,k_i)$ such that $r\in [\overline{\lambda}_i,k_i)$ must all intersect. Further, for $i> i'$ such that $ \overline{\lambda}_i \le r < k_i$ we have: $\overline{\lambda}_i < \overline{\lambda}_{i'} \le r < k_i < k_{i'}$. As such we deduce that for $$I=\max \{i\in \{1,\dots,m\}\,\rvert\, \overline{\lambda}_i \le r < k_i\}$$ we have $$\sharp \{i\in \{1,\dots,m\}\,\rvert\, \overline{\lambda}_i \le \overline{\lambda}_I < k_i\} \ge \sharp \{i\in \{1,\dots,m\}\,\rvert\, \overline{\lambda}_i \le r < k_i\}$$ and therefore we may assume $r = \overline{\lambda}_I$ for some $1\le I \le m$.

    Since $k_i < k_{i-1}$ and $\overline{\lambda}_i < \overline{\lambda}_{i-1}$, we deduce that there exists $I'$ such that for $I' \le i \le m$ we have both $\overline{\lambda}_i < \overline{\lambda}_I$ and $k_i \le \overline{\lambda}_I$. 
    For $I \le i < I'$ we have $\overline{\lambda}_i \leq \overline{\lambda}_I$ and $k_i < \overline{\lambda}_I$.
    Thus, $$k:=\sharp\{i\in \{1,\dots,m\}\,\rvert\, \overline{\lambda}_i \le \overline{\lambda}_I < k_i\}\,=\, I' - I\,\geq 1.$$

        For every $i$, denote $m_i:= M_i + 1 - (k_i - \overline{\lambda}_I)$. 
        
        For $I < i < i' < I'$ 
        we observe that $\overline{\lambda}_I < k_{i'} < k_i$ as such $ k_i - \overline{\lambda}_I > k_{i'} -\overline{\lambda}_I$. Since $M_i \ge M_{i'}$, we deduce that $m_i > m_{i'}$. Thus $\ror{i}{m_i}$ and $\ror{i'}{m_{i'}}$ are incomparable.

    Further, we observe that $$M_i + 1 - (k_i - \overline{\lambda}_i) \leq m_i= M_i + 1 - (k_i - \overline{\lambda}_I) \le M_i.$$ Hence by Theorem \ref{theorem for ror change inequality},$$\forall \, j, \;\; M_i + 1 - (k_i - \overline{\lambda}_i) \le j \le M_i + 1 - (k_i - \overline{\lambda}_I) \; \;\Longrightarrow \; \;c_\lambda (\ror{i}{j})=1.$$ 
    By the same theorem, for any $1\le j \le M_i + 1 - (k_i - \overline{\lambda}_i)$, we have: $c_\lambda(\ror{i}{j})=0$.

    Consider $\Sigma\in \mathbb{S}$ given by $\Sigma\cap \mathcal{R}=\{\ror{i}{m_i}\,|\,I < i < I'\}$ (as we saw above, this set is incomparable). 
    Then $$B_\Sigma = \{\ror{i}{j}\,|\, I < i<I',\; 1 \le j  <m_i\} \cup \{\ror{i}{j}|I' \le i \le m, \; 1\le j \le n\}$$
    
    Computing $\overline{\lambda}_\Sigma$ we find that for $m \le i \le I'$ we have $(\overline{\lambda}_\Sigma)_i = k_i$ and for $I < i < I'$ we have $(\overline{\lambda}_\Sigma)_i = (\overline{\lambda}_\Sigma)_I = \overline{\lambda}_I$.

 We now consider 2 cases: when $D_\lambda(\overline{\lambda}_I)$ is $>$ and when $D_\lambda(\overline{\lambda}_I)= \times$.

    %Denote $k := \sharp \{ i\in {1,\dots, m}|\overline{\lambda}_i \le \overline{\lambda}_I <k_i\}$.
    %As the number of $i$ such that $I < i < I'$ is $k-1$, we must
    
    If $D_\lambda(\overline{\lambda}_I)$ is $>$ then $k_{I'}=\overline{\lambda}_I$. So $(\overline{\lambda}_\Sigma)_{I'} = (\overline{\lambda}_\Sigma)_I$, and thus $k$ symbols $\times$ (for every $I < i \le I'$) have been moved to position $\overline{\lambda}_I$. Hence there must be at least $k$ $\times$ symbols in that position. 

    If $D_\lambda(\overline{\lambda}_I)= \times$ we find that $k_{I'}<\overline{\lambda}_I$. Hence precisely $k-1$ symbols $\times$ (for $I'<i<I$) were moved to position $\overline{\lambda}_I$, that already had a $\times$ symbol in it. 
    
    Thus $D^{\Sigma}_\lambda(\overline{\lambda}_I)$ has at $k$ symbols $\times$ in both cases, as required.

\end{proof}

\begin{example}
    Consider $\lambda \in \Lambda^+_{4|4}$ given by $\overline{\lambda}=(3\, 2\, 1\, 0|0\, -2\, -3\, -4)$.
    Below is the arrow diagram of $\lambda$:

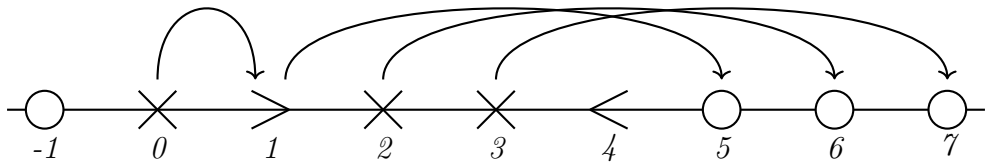
    \begin{equation*}
    \begin{tikzpicture}
    
        \draw[thick, -] (-6.5 ,0)--(6.5, 0);
        
        \draw[thick, fill=white] (-6,0) circle [radius=0.25cm];
        \node[] at (-6,-0.5) {-1};
        
        % times
        \draw[thick, -] (-4.5,0)--(-4.5 + 0.25 ,0.25);
        \draw[thick, -] (-4.5,0)--(-4.5 - 0.25 ,0.25);
        \draw[thick, -] (-4.5,0)--(-4.5 + 0.25 ,-0.25);
        \draw[thick, -] (-4.5,0)--(-4.5 - 0.25 ,-0.25);
        \node[] at (-4.5,-0.5) {0};
        
        % greater than 
        \draw[thick, -] (-3 +0.25, 0)--(-3 - 0.25, 0.25);
        \draw[thick, -] (-3 +0.25, 0)--(-3 - 0.25, -0.25);
        \node[] at (-3,-0.5) {1};
        
        % times
        \draw[thick, -] (-1.5,0)--(-1.5 + 0.25 ,0.25);
        \draw[thick, -] (-1.5,0)--(-1.5 - 0.25 ,0.25);
        \draw[thick, -] (-1.5,0)--(-1.5 + 0.25 ,-0.25);
        \draw[thick, -] (-1.5,0)--(-1.5 - 0.25 ,-0.25);
        \node[] at (-1.5,-0.5) {2};
            
        % times
        \draw[thick, -] (0,0)--(0 + 0.25 ,0.25);
        \draw[thick, -] (0,0)--(0 - 0.25 ,0.25);
        \draw[thick, -] (0,0)--(0 + 0.25 ,-0.25);
        \draw[thick, -] (0,0)--(0 - 0.25 ,-0.25);
        \node[] at (0,-0.5) {3};
        
        % le
        \draw[thick, -] (1.5 -0.25, 0)--(1.5 + 0.25, 0.25);
        \draw[thick, -] (1.5 -0.25, 0)--(1.5 + 0.25, -0.25);
        \node[] at (1.5,-0.5) {4};
        
        \draw[thick, fill=white] (3,0) circle [radius=0.25cm];
        \node[] at (3,-0.5) {5};
        
        \draw[thick, fill=white] (4.5,0) circle [radius=0.25cm];
        \node[] at (4.5,-0.5) {6};
        
        \draw[thick, fill=white] (6,0) circle [radius=0.25cm];
        \node[] at (6,-0.5) {7};

        % arrows
        \draw[thick, ->] (-4.5 ,0.4) .. controls (-4.5 ,1.65) and (-3 -0.2, 1.65) .. (-3 -0.2,0.4);
        \draw[thick, ->] (-3 +0.2,0.4) .. controls (-3 +0.2, 1.65) and (3, 1.65) .. (3,0.4);
        \draw[thick, ->] (-1.5, 0.4) .. controls (-1.5, 1.65) and (4.5, 1.65) .. (4.5, 0.4);
        \draw[thick, ->] (0, 0.4) .. controls (0, 1.65) and (6, 1.65) .. (6, 0.4);
        
        % cross
        %\draw[thick, -] (-6,0)--(-6 + 0.25 ,0.25);
        %\draw[thick, -] (-6,0)--(-6 - 0.25 ,0.25);
        %\draw[thick, -] (-6,0)--(-6 + 0.25 ,-0.25);
        %\draw[thick, -] (-6,0)--(-6 - 0.25 ,-0.25);
        
        % less than 
        %\draw[thick, -] (-3 -0.25, 0)--(-3 + 0.25, 0.25);
        %\draw[thick, -] (-3 -0.25, 0)--(-3 + 0.25, -0.25);
        
        % greater than 
        %\draw[thick, -] (-3 +0.25, 0)--(-3 - 0.25, 0.25);
        %\draw[thick, -] (-3 +0.25, 0)--(-3 - 0.25, -0.25);
        
    \end{tikzpicture}
\end{equation*}

    Observe that position $p=4$ has $3$ arrows going over it: namely, the arrows starting at $\overline{\lambda}_3, \overline{\lambda}_2, \overline{\lambda}_1$. 
    
    As we state in the proof, there exists a position to the left of $p=4$ that is the start of some arrow and also has $3$ arrows going over it (or starting at it). In this case it is position $\overline{\lambda}_1=3$, so $I=1$.

    We observe that $M_4= 1, M_3=M_2=M_1= 4$ and $k_4= 1, k_3= 5, k_2= 6, k_1= 7$. Computing $m_i$ for $i=1,2,3$, as defined in the proof, we find that $m_3= 3, m_2 = 2$ and $m_1 = 1$. 
    The base $\Sigma$ defined by $\Sigma \cap \mathcal{R}=\{\ror{3}{3}, \ror{2}{2}, \ror{1}{1}\}$ should have $3$ stacked symbols $\times$ in position $3$. 

    Below is the diagram for $B_\Sigma$, overlaid with the CTD $c_\lambda$:

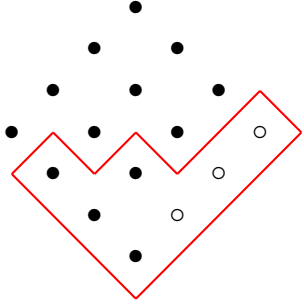
    \begin{equation*}
\begin{tikzpicture}[anchorbase,scale=1.1]
% row 1
\node at (0,1.5) {$\bullet $};

% row 2
\node at (0.5,1) {$\bullet $};
\node at (-0.5, 1) {$\bullet $};

% row 3
\node at (1, 0.5) {$\bullet $};
\node at (0, 0.5) {$\bullet $};
\node at (-1, 0.5) {$\bullet $};

% row 4
\node at (1.5, 0) {$\circ $};
\node at (0.5, 0) {$\bullet $};
\node at (-0.5, 0) {$\bullet $};
\node at (-1.5, 0) {$\bullet $};

% row 5
\node at (1, -0.5) {$\circ $};
\node at (0, -0.5) {$\bullet $};
\node at (-1, -0.5) {$\bullet $};

% row 6
\node at (0.5, -1) {$\circ $};
\node at (-0.5, -1) {$\bullet $};

% row 7
\node at (0, -1.5) {$\bullet $};

% zigzag
\draw[thick, red] (-1.5,-0.5)--(-1, 0);
\draw[thick, red] (-0.5,-0.5)--(-1, 0);
\draw[thick, red] (-0.5,-0.5)--(0, 0);
\draw[thick, red] (0.5,-0.5)--(0, 0);
\draw[thick, red] (0.5,-0.5)--(1.5, 0.5);
\draw[thick, red] (2,0)--(1.5, 0.5);

\draw[thick, red] (2,0)--(0, -2);
\draw[thick, red] (-1.5,-0.5)--(0, -2);

\end{tikzpicture}
\end{equation*} 

    Computing $\overline{\lambda}_\Sigma$ using the above diagram, we find that $\overline{\lambda}_\Sigma = ( 3\, 3\, 3\, 1| -3\, -3\, -3\, -4)$ and the diagram $D^{\Sigma}_\lambda$ is given by 

    
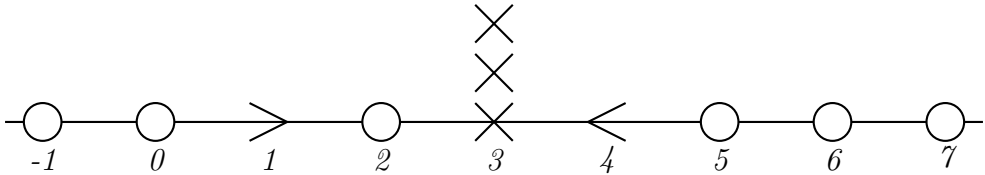
\begin{equation*}
    \begin{tikzpicture}
    
        \draw[thick, -] (-6.5 ,0)--(6.5, 0);
        
        \draw[thick, fill=white] (-6,0) circle [radius=0.25cm];
        \node[] at (-6,-0.5) {-1};
        
        \draw[thick, fill=white] (-4.5,0) circle [radius=0.25cm];
        \node[] at (-4.5,-0.5) {0};
        
        % greater than 
        \draw[thick, -] (-3 +0.25, 0)--(-3 - 0.25, 0.25);
        \draw[thick, -] (-3 +0.25, 0)--(-3 - 0.25, -0.25);
        \node[] at (-3,-0.5) {1};
        
        \draw[thick, fill=white] (-1.5,0) circle [radius=0.25cm];
        \node[] at (-1.5,-0.5) {2};
            
        % times
        \draw[thick, -] (0,0)--(0 + 0.25 ,0.25);
        \draw[thick, -] (0,0)--(0 - 0.25 ,0.25);
        \draw[thick, -] (0,0)--(0 + 0.25 ,-0.25);
        \draw[thick, -] (0,0)--(0 - 0.25 ,-0.25);
        % stacked cross
        \draw[thick, -] (0,0 + 0.65)--(0 + 0.25 ,0.25 + 0.65);
        \draw[thick, -] (0,0 + 0.65)--(0 - 0.25 ,0.25 + 0.65);
        \draw[thick, -] (0,0 + 0.65)--(0 + 0.25 ,-0.25 + 0.65);
        \draw[thick, -] (0,0 + 0.65)--(0 - 0.25 ,-0.25 + 0.65);
        % stacked cross
        \draw[thick, -] (0,0 + 2* 0.65)--(0 + 0.25 ,0.25 + 2* 0.65);
        \draw[thick, -] (0,0 + 2* 0.65)--(0 - 0.25 ,0.25 + 2* 0.65);
        \draw[thick, -] (0,0 + 2* 0.65)--(0 + 0.25 ,-0.25 + 2* 0.65);
        \draw[thick, -] (0,0 + 2* 0.65)--(0 - 0.25 ,-0.25 + 2* 0.65);
        \node[] at (0,-0.5) {3};
        
        % le
        \draw[thick, -] (1.5 -0.25, 0)--(1.5 + 0.25, 0.25);
        \draw[thick, -] (1.5 -0.25, 0)--(1.5 + 0.25, -0.25);
        \node[] at (1.5,-0.5) {4};
        
        \draw[thick, fill=white] (3,0) circle [radius=0.25cm];
        \node[] at (3,-0.5) {5};
        
        \draw[thick, fill=white] (4.5,0) circle [radius=0.25cm];
        \node[] at (4.5,-0.5) {6};
        
        \draw[thick, fill=white] (6,0) circle [radius=0.25cm];
        \node[] at (6,-0.5) {7};
        
        % cross
        %\draw[thick, -] (-6,0)--(-6 + 0.25 ,0.25);
        %\draw[thick, -] (-6,0)--(-6 - 0.25 ,0.25);
        %\draw[thick, -] (-6,0)--(-6 + 0.25 ,-0.25);
        %\draw[thick, -] (-6,0)--(-6 - 0.25 ,-0.25);
        
        % less than 
        %\draw[thick, -] (-3 -0.25, 0)--(-3 + 0.25, 0.25);
        %\draw[thick, -] (-3 -0.25, 0)--(-3 + 0.25, -0.25);
        
        % greater than 
        %\draw[thick, -] (-3 +0.25, 0)--(-3 - 0.25, 0.25);
        %\draw[thick, -] (-3 +0.25, 0)--(-3 - 0.25, -0.25);
        
    \end{tikzpicture}
\end{equation*}
    
    Notice that in this case the base $\Sigma$ is in fact $\Sigma_\lambda$ and $\overline{\lambda}_\Sigma = \overline{\lambda^\dagger}$, and thus $tail(\lambda)=3$. In this example, $longtail(\lambda)=3$, too, as can be seen in the next theorem. The equality between these values of course need not hold in general, as was shown in \cref{tail conj counterexample}. 
\end{example}

We now prove an explicit formula for $longtail(\lambda)$ in terms of the arrow diagram of $\lambda$.
\begin{theorem}
\label{thrm:max tail arrow}
Let $\lambda\in \distHW$. Then $$longtail(\lambda)=\max_{r\in\mathbb{Z}}\sharp\{i\in \{1,\dots,m\}\,\rvert\, \overline{\lambda}_i\leq r < k_i\}. $$
In other words, $longtail(\lambda)$ is the maximal number of arrows going over or starting at a single point in the arrow diagram of $\lambda$.
\end{theorem}
\begin{proof}
 Combining \cref{lem:stacked_x_vs_arrows_over_ineq} and \cref{lem:longtail_ineq_number_x_stacked} we obtain:
$$longtail(\lambda) \geq \max_{r\in\mathbb{Z}}\sharp\{i\in \{1,\dots,m\}\,\rvert\, \overline{\lambda}_i \leq r < k_i\}.$$

Now, recall from \cref{cor:CTD_vs_longtail_eq} that $longtail(\lambda)$ is the maximal cardinality of an incomparable subset of $C_{\lambda}$. Consider an incomparable subset $S\subset C_{\lambda}$ of cardinality $longtail(\lambda)$. By \cref{cor:iso_set_in_CTD_vs_arrows_over}, we conclude that 
$$longtail(\lambda) =\sharp S\leq \max_{r\in\mathbb{Z}}\sharp\{i\in \{1,\dots,m\}\,\rvert\, \overline{\lambda}_i\leq r < k_i\}$$ and we are done.

\end{proof}

We now wish to give an explicit formula for the value $longtail(\lambda)$ in terms of the cap diagram of $\lambda$. 

Recall that $\times_s$ denotes the position of the $s$-th symbol $\times$ in $D_{\lambda}$, numbered from left to right. We will use the following lemma:

\begin{lemma}
    Let $\lambda\in \distHW$ and let $a:=atyp(\lambda)$.
 
 Denote by $k_1, \ldots, k_n$ the endpoints of the arrows in $D_{\lambda}$. For every $s=1, \ldots, a$, let $c_s$ be the right endpoint of the cap starting at $\times_s$. 

    Let $r\in \mathbb{Z}$ be a position in the weight diagram $D_{\lambda}$. Then 
    $$\sharp\{i\in \{1,\dots,n\}\,|\,\overline{\lambda}_i \leq r < k_i \}= \sharp\{i\in \{1,\dots,a\}\,|\, \times_i \leq r < c_i \}$$
    In other words, the number of arrows going over position $r$ is the number of caps going over position $r$.

\end{lemma}
\begin{proof}

  Denote: $$A:=\{\times_i\, | \, 1\leq i\leq a,\,\times_i \leq r\}, \;\; A':=\{\overline{\lambda}_i\, | \, 1\leq i\leq n,\, \overline{\lambda}_i \leq r\}$$ and $$ E:=\{c_i\,|\, 1\leq i\leq a,\,c_i \leq r\}, \;\; E':=\{k_i\,|\, 1\leq i\leq n,\,k_i \leq r\}.$$

  The number of caps going over position $r$ is $$\sharp\{i\,|\, \times_i \leq r < c_i \} = \sharp A- \sharp E $$
  (we consider the number of caps starting to the left of $r$ or at $r$ and subtract the number of caps ending to the left of $r$ or at $r$).

  The number of arrows going over position $r$ is $$\sharp\{i\,|\, \overline{\lambda}_i \leq r \leq k_i \} = \sharp A' - \sharp E' $$
  (we consider the number of arrows starting to the left of $r$ or at $r$ and subtract the number of arrows ending to the left of $r$ or at $r$).
  
  Clearly we have: $A\subset A'$. Recall from Lemma \ref{x o seq cap end} that, $$\{c_1, \ldots, c_a\} = \{k_i\, |\, 1\leq i\leq n, \; D_{\lambda}(k_i)=\circ\}$$ Hence $E\subset E'$ and $$E'\setminus E = \{k_i \leq r\, |\, 1\leq i\leq n, \; D_{\lambda}(k_i)  \text{ is } > \} = \{\overline{\lambda}_i \leq r\, |\, 1\leq i\leq n, \; D_{\lambda}(\overline{\lambda}_i)  \text{ is } > \}$$ 
  The last equality is due to the fact that every symbol $>$ in $D_{\lambda}$ is the endpoint of some arrow. On the other hand, $$\{\overline{\lambda}_i<r\, |\, 1\leq i\leq n, \; D_{\lambda}(\overline{\lambda}_i)  \text{ is } > \}=A'\setminus A$$
  so $E'\setminus E=A'\setminus A$. This proves that $ \sharp A - \sharp E = \sharp A' - \sharp E' $ as required.
\end{proof}

The next statement is a direct corollary of \cref{thrm:max tail arrow}:
\begin{cor}\label{thrm:longtail_and_cap_diagrams_formula}

    Let $\lambda \in \distHW$ and let $a:=atyp(\lambda)$. Let $\{c_1,\dots,c_a\}$ be the right endpoints of the caps in the cap diagram of $\lambda$. Then $$longtail(\lambda)=\max_{r\in\mathbb{Z}}\sharp\{i\in \{1,\dots,a\}|\times_i \leq r < c_i \} $$ That is to say, $longtail(\lambda)$ is given by the maximal number of caps going over (strictly above) a single point in the cap diagram of $\lambda$.
\end{cor}

\bibliographystyle{amsalpha}
%\bibliography{bibliography}

\end{document}